\documentclass{amsart}

\usepackage[utf8]{inputenc}
\usepackage[center]{caption}
\usepackage{float}
\usepackage{hyperref}
\usepackage{amsmath}
\usepackage{amssymb}
\usepackage{amsthm}
\usepackage{relsize}
\usepackage{scalerel}
\usepackage{tikz-cd}
\usepackage{bm}
\usepackage{enumitem}
\usepackage{pict2e}
\usepackage{graphicx}

\title{Operadic actions on long knots and $2$--string links}

\author{Etienne Batelier}
\author{Julien Ducoulombier}

\newtheorem{Thm}{Theorem}[section]
\newtheorem{Lem}[Thm]{Lemma}
\newtheorem{Cor}[Thm]{Corollary}
\newtheorem{Prop}[Thm]{Proposition}

\newtheorem{Claim}{Claim}
\newtheorem*{Thm*}{Theorem}

\theoremstyle{definition}
\newtheorem{Rem}[Thm]{Remark}
\newtheorem{Ex}[Thm]{Example}
\newtheorem{Defn}[Thm]{Definition}

\newcommand{\cubes}[1]{\mathcal{C}_{#1}}
\newcommand{\cubeso}[1]{\mathcal{C}_{#1}^{\circ}}
\newcommand{\overcubes}[1]{\mathcal{C}_{#1}^{\infty}}

\newcommand{\Com}{\mathcal{C}om}
\newcommand{\As}{\mathcal{A}s}
\newcommand{\SCL}{\mathcal{SCL}}

\newcommand{\Top}{\mathbf{Top}}
\newcommand{\Alg}[1]{{#1}\text{--}\mathbf{Alg}}

\DeclareMathOperator{\Emb}{Emb}
\DeclareMathOperator{\Diff}{Diff}
\DeclareMathOperator{\Diffd}{Diff_d}
\DeclareMathOperator{\Homeo}{Homeo}
\DeclareMathOperator{\KDiff}{KDiff}

\DeclareMathOperator{\inter}{int}
\DeclareMathOperator{\id}{id}
\DeclareMathOperator{\conf}{conf}
\DeclareMathOperator{\supp}{supp}

\newcommand{\OO}{\mathcal{O}}
\newcommand{\EE}{\mathcal{E}}
\newcommand{\KK}{\mathcal{K}}
\newcommand{\KKhat}{\hat{\KK}}
\newcommand{\PP}{\mathcal{P}}
\newcommand{\PPhat}{\hat{\PP}}
\newcommand{\LL}{\mathcal{L}}
\newcommand{\LLhat}{\hat{\LL}}
\newcommand{\QQ}{\mathcal{Q}}
\newcommand{\QQhat}{\hat{\QQ}}
\newcommand{\SShat}{\hat{\mathcal{S}}}
\newcommand{\bb}[1]{\mathbb{#1}}
\newcommand{\del}{\partial}
\newcommand{\cof}{\hookrightarrow}
\newcommand{\up}{\mathsmaller\uparrow}
\newcommand{\down}{\mathsmaller\downarrow}
\newcommand{\updown}{\mathsmaller\updownarrow}
\newcommand{\varphihat}{\hat{\varphi}}
\newcommand{\downmapsto}{\rotatebox[origin=c]{-90}{$\mapsto$}}
\newcommand{\longdownmapsto}{\rotatebox[origin=c]{-90}{$\longmapsto$}}

\renewcommand{\u}[1]{\underline{#1}}

\newcommand{\ult}{\underline{t}}
\newcommand{\ulu}{\underline{u}}
\newcommand{\ulx}{\underline{x}}
\newcommand{\co}{\colon}
\newcommand{\underscore}{\_}
\newcommand{\qr}{\text}
\newcommand{\mathbold}{\bm}
\newcommand{\lk}{\mathrm{lk}}
\newcommand{\hoTot}{\mathrm{hoTot}}
\newcommand{\CAut}{\mathrm{CAut}}
\newcommand{\ext}{\mathrm{ext}}
\newcommand{\SO}{\mathrm{SO}}
\newcommand{\hash}{\#}
\def\faktor#1#2{#1 \big/ #2}

\begin{document}

\begin{abstract}
We realize the space of $2$--string links $\mathcal{L}$
as a free algebra over a colored operad denoted
by $\mathcal{SCL}$ (for ``Swiss-cheese for links'').
This result extends works of Burke and Koytcheff about the quotient of
$\mathcal{L}$ by its center, and is compatible with
Budney's freeness theorem for long knots. From an algebraic point of view,
our main result refines
Blaire, Burke and Koytcheff's theorem on the monoid of isotopy classes of
string links. Topologically,
it expresses the homotopy type of the isotopy class of a $2$--string link
in terms of the homotopy types
of the classes of its prime factors.
\end{abstract}

\maketitle

\tableofcontents
\addtocontents{toc}{\protect\setcounter{tocdepth}{1}} 

\section*{Introduction}

\subsection*{Motivation and context}

The study of knots and links is a vast subject that emerged in the late
nineteenth century and saw several renewals in the past thirty years. It
is subject to many different approaches,
being at the crossroads of topology, geometry, algebra, combinatorics
and physics.
The central theme in classical knot theory is the study of the isotopy
classes of knots,
ie the isotopy classes of embeddings in $\Emb(S^1, S^3)$.
They are the set of components $\pi_0\Emb(S^1, S^3)$.
A common method is to try to chop the knots into simpler pieces. 
Two ways of performing such a decomposition have proved themselves
particularly fruitful:
the prime decomposition and the satellite decomposition.
The former splits a knot as the connected sum of other knots called its
prime factors. The connected sum is a binary operation on $\pi_0\Emb(S^1,
S^3)$ denoted 
by $\hash$.
It endows the isotopy classes with a unital commutative monoidal structure.
Intuitively, the product $k_1 \hash k_2$ is the knot obtained by cutting
open $k_1$ and $k_2$
and closing them up into a single knot. This decomposition is fairly well
understood thanks to a theorem of 
Schubert~\cite{Schubert} stating that $\pi_0\Emb(S^1, S^3)$ is freely generated
as a monoid
by the isotopy classes of the prime knots.
There are infinitely many prime knots and they can be very different
in nature.
This is why it is often useful to further decompose the prime knots as
satellites of simpler knots.
The satellite construction also originates in Schubert's work.
It consists of a wide family of operations one can carry out on several
knots at a time.
Its rigorous definition is a bit more involved but clearly dispensed for
example in Cromwell's book~\cite{Cromwell}.
Jaco, Shalen, Johannson and Thurston played a major role in the study of
satellite knots,
with the results in \cite{JS,J,Thurston}. 
Although they are quite complicated, the satellite operations have the
advantage
of generating the whole space of knots from a fairly small and classifiable
class of knots.
Namely, Budney shows in~\cite{BudneyJSJ} --- refining a result of
\cite{Thurston} --- that
every knot can be obtained via successive satellite operations on hyperbolic
and torus knots. 

This depicts a duality between the two methods of decomposition:
one is simple but leads to potentially complicated primes,
while the other is more elaborate but has easier irreducible pieces.
A similar story can be told for links, ie for the isotopy classes of
embeddings in
$\Emb(S^1 \amalg \cdots \amalg S^1, S^3)$.
Adapting the connected sum $\hash$ to this setting takes some work:
problems arise once the components of a link are cut open, as there is no
canonical way to close them up.
However, a step-by-step adaptation of the satellite construction works
for links, and a decomposition
theorem exists as well. Its building blocks are the hyperbolic and
Seifert-fibered links. 

Nowadays, it is more common to study not only the set of components
$$\pi_0 \Emb(S^1 \amalg \cdots \amalg S^1, S^3), $$
but the full homotopy
type of the spaces of knots and links.
To adapt the decomposition approach described above, one needs to define
an analogue of the connected sum and
satellite operations on the space level, ie directly on the embeddings
and not between isotopy classes.
To rigorously carry out this task, one uses the space of long knots $\KK$
and the space of string links $\LL$.
Coupled with the language of operads, these spaces enable one to extend
the existing operations on
$\pi_0$ to the space level. More precisely, the different types of operations
one can carry out on knots and links
can be encoded in the action of an operad on a space of embeddings. This
new framework is due to
Budney in~\cite{Budney} and~\cite{BudneySplicing} for the case of long knots.
Namely, Budney constructs in~\cite{Budney} an action of the little
$2$--cubes operad
$\cubes{2}$ on a space $\KKhat$ homotopy equivalent to $\KK$, in such a
way that
all the induced operations descend to the connected sum on $\pi_0$.
He builds in~\cite{BudneySplicing} another action on $\KKhat$ of a more
intricate operad
which he calls the splicing operad. These operations correspond to the
satellite constructions in many ways.
In the case of string links, Burke and Koytcheff build in~\cite{KoytcheffInfect} a complex operad called the infection operad.
It is an adaptation of Budney's splicing operad and deals with the satellite
operations between string links.
The authors mentioned above not only prove the existence of such actions
but also their freeness,
refining the unique decomposition results on $\pi_0$.
It remains to find an operadic encoding of the connected sum of links,
if possible, leading to a free algebra. 

\subsection*{Present work}

The present paper aims to fill in this gap in the case of $2$--stranded
string links.
Unlike $\pi_0 \KK$, the monoid of isotopy classes of $2$--string links is
neither free nor
commutative. Indeed, as explained 
by Blair, Burke and Koytcheff in~\cite{KoytcheffPrime}, $\pi_0 \LL$
contains invertible elements in the form
of the pure braid group $KB_2$.
Together with three copies of $\pi_0\KK$, these invertible elements generate
the center of the monoid.
Burke and Koytcheff state a partial result in \cite{KoytcheffInfect}
by building a free action of the little $1$--cubes operad $\cubes{1}$
on a subspace of $2$--string links
that ignores the homotopy center. They mention as an open problem the
extension of such a
structure to the whole space of $2$--string links. Our main result provides
an answer to this particular question.
For this purpose, we introduce a four-colored operad $\SCL$ (standing for
``Swiss-cheese for links'') with set of colors
$S=\{o,\up,\down,\updown\}$. An $\SCL$--algebra is in particular a family
of spaces
$(X, A_{\up}, A_{\down}, A_{\updown})$, where $X$ is a $\cubes{1}$--algebra
and each
$A_{s}$, $s\in \{\up,\down,\updown\}$, is a $\cubes{2}$--algebra acting on $X$.
One can think of the $A_{s}$ as independent parts of the homotopy center
of $X$.
As in the case of Budney's $\cubes{2}$--action on long knots, we consider
a space $\LLhat$
homotopy equivalent to $\LL$ to prove a first result which can be summarized
as follows:

\begin{Thm*}[Theorem \ref{SCLActs}]
The family $(\LLhat,\KKhat,\KKhat,\KKhat)$ is 
	an $\SCL$--algebra.
In particular, the family $(\LL,\KK,\KK,\KK)$ is homotopy equivalent to
an explicit $\SCL$--algebra.
\end{Thm*}

The structure so obtained is compatible with Budney's action on
long knots and
Burke and Koytcheff's action on their subspace of noncentral string links.
It provides a complete understanding of the connected sum of $2$--string
links.
As expected, we also prove a freeness result, refining the structure
theorem for the
monoid of isotopy classes proved in \cite{KoytcheffPrime}. In order to do
this, we
discard the invertible elements by splitting $\pi_0\LLhat$ as a product
$\pi_0 \LLhat^{0} \times KB_2$ and prove:

\begin{Thm*}[Theorem \ref{freeSCL}]
The quadruplet of spaces $(\LLhat^0,\KKhat,\KKhat,\KKhat)$ is homotopy
equivalent as an $\SCL$--algebra
to the free $\SCL$--algebra generated by prime knots and links.
\end{Thm*}

In addition to these algebraic considerations, this result has
a homotopical significance as it expresses the
homotopy type of the isotopy class of a $2$--string link as a function of
the homotopy types of the classes
of its prime factors. This reduces the computation of the homotopy type
of the whole $\LL$ to
figuring out the homotopy types of the components of the primes. As
mentioned above, the latter can be further
decomposed thanks to Burke and Koytcheff's infection operad defined 
in~\cite{KoytcheffInfect}. 

\subsection*{Organization of the paper} 

We define in a \hyperlink{secon}{first section} the
different spaces of embeddings at stake here:
long knots, string links and their fattened versions $\KKhat$ and $\LLhat$.
The \hyperlink{sectw}{second section} sets up the operadic framework we use. We recall in
particular the notion of colored operad and
discuss the resulting algebras. We introduce along the way the operad that
will appear in our main result,
the Swiss-cheese operad for links $\SCL$. The \hyperlink{secth}{third section} aims to define
various operadic actions on the
spaces of knots and links. We recall the constructions of Budney, Burke
and Koytcheff's actions and unify them in a
single action of $\SCL$ on $2$--string links. Finally, the \hyperlink{secfo}{fourth section}
proves the freeness result
sketched above using low dimensional and homotopical tools. 

\subsection*{Upcoming projects}

Our main statements concern the space of string links on two strands.
One might naturally wonder what happens in the $k$--stranded case for some
$k > 2$.
The conjecture that Theorems \ref{SCLActs} and \ref{freeSCL} have adaptations
to arbitrary string links seems reasonable,
since most of the techniques used 
in Subsection \ref{FreenessResultsthDimensionalTopology}
naturally generalize to the $k$--stranded case. However, lots of difficulties
arise, even at the level of isotopy classes.
Theorem \ref{freeSCL} generalizes Blair, Burke and Koytcheff's explicit
model for the monoid
$\pi_0 \LL$, but it does not provide an alternative proof for it.
Actually, Blair, Burke and Koytcheff's result is used in the very first
line of the proof of our Theorem \ref{freeSCL}.
Thus, if one wants to adapt Theorem \ref{freeSCL} to the $k$--stranded case,
some preliminary work on the monoid of
isotopy classes of string links on $k$ strands is necessary. 
A key point is the understanding of its center.
The commutation between string links on $k$ strands has already been
characterized in \cite{KoytcheffPrime},
but there are other relations among prime links that remain to be
understood. For example, the invertible $k$--stranded string links
are the pure braids on $k$ strands, which already admit a wide family of
fairly complex relations amongst themselves.
Moreover, these units are not central anymore, which makes it harder to
study prime decompositions.
A solution to these difficulties could be to accept a less explicit
construction for the replacement of $\SCL$, maybe a definition by
induction on $k$. The operad for $k$--stranded string links would rely on
a large set of colors, and would restrict to the operad for
$(k{-}1)$--stranded string links on some subcollections of colors.
The existence and freeness of an action could then be easier to prove,
but the difficulty is only shifted
towards understanding these potentially massive operads. 

Another interesting question concerns the Goodwillie--Weiss calculus,
introduced to study embedding spaces in \cite{Weiss2,Weiss}.
In the context of knots and links, this theory gives rise to two towers
of fibrations $\{T_{k}\KKhat\}$ and $\{T_{k}\LLhat\}$,
converging to the so-called polynomial approximations $T_{\infty}\KKhat$
and  $T_{\infty}\LLhat$, respectively.
Unfortunately, the natural applications $\iota_{\KKhat}: \KKhat\rightarrow
T_{\infty}\KKhat$ and
$\iota_{\LLhat}: \LLhat\rightarrow  T_{\infty}\LLhat$ are not weak
equivalences,
but they preserve a lot of homotopical information.
In particular, we know from Budney, Conant, Koytcheff and Sinha~\cite{Conant} 
that the map $\KKhat\rightarrow
T_{k}\KKhat$ is a finite type-$(k{-}1)$ knot invariant
and it has been conjectured that $T_{k}\KKhat$ is actually the universal
finite type-$(k{-}1)$ invariant.
This conjecture is already proved rationally in 
Voli\'{c}'s thesis \cite{Volic}.
Moreover, the polynomial approximations can be simplified and identified
to homotopy totalizations using the multiplicative
Kontsevich operad $K_{3}$ obtained as a compactification of configurations
of points in $\mathbb{R}^{3}$.
Briefly speaking, one has the identifications
\[
\KKhat
{\iota_{\KKhat}}
T_{\infty} \KKhat
{\,\mu_{\KKhat}} \hoTot(K_{3}\circ SO(3)), \quad
\LLhat
{\iota_{\LLhat}}
T_{\infty} \LLhat
{\mu_{\LLhat}} \hoTot(K^{2}_{3}\circ SO(3)),
\]
where $K^{2}_{3}(k)=K_{3}(2k)$ is a shifted version of the Kontsevich operad.
The applications $\mu_{\KKhat}$ and $\mu_{\LLhat}$ have been proved to be
weak equivalences by Sinha in~\cite{Sinha06}
and Munson and Voli\'{c} in~\cite{Munson14}, respectively. We know that the spaces
$\KKhat$, $T_{\infty} \KKhat$
and $\hoTot(K_{3}\circ SO(3))$ are $\cubes{2}$--algebras by \cite{Budney},
\cite{Boavida13} and \cite{Ducoulombier},
respectively. However, it is still unknown if $\iota_{\KKhat}$ and
$\mu_{\KKhat}$ are $\cubes{2}$--algebra maps.
All these questions can be extended to the colored case. From the present
work, the family $(\LLhat,\KKhat,\KKhat,\KKhat)$ is
equipped with an explicit $\SCL$--algebra structure. We believe that similar
structures exist for the families
\[
\begin{gathered}
\bigl(T_{\infty}\LLhat,T_{\infty}\KKhat,T_{\infty}\KKhat,T_{\infty}\KKhat\bigr),
\\
\bigl(\hoTot(K^{2}_{3}\circ \SO(3)),\hoTot(K_{3}\circ \SO(3)),
\hoTot(K_{3}\circ \SO(3)),\hoTot(K_{3}\circ \SO(3))\bigr),
\end{gathered}
\]
and that the zigzag of morphisms induced by $\iota_{\KKhat}$,
$\mu_{\KKhat}$,
$\iota_{\LLhat}$ and $\mu_{\LLhat}$ between the corresponding families
are morphisms of $\SCL$--algebras. 

\subsection*{Acknowledgements}

The authors are indebted to Thomas Willwacher and Robin Koytcheff for
answering numerous questions
and for their comments on the first version of this paper.
They are grateful to Dev Sinha and Victor Turchin for discussions leading
to the problem solved in the present work.
The two authors are also thankful to the referee for a very careful read
and the many helpful suggestions.
Finally, the authors acknowledge the generous support of ETH Z\"{u}rich and
WWU M\"{u}nster.
The second author is partially supported by the grant ERC-2015-StG 678156
GRAPHCPX.

\section*{General framework and notations}
\label{GeneralFrameworkandNotations}

We set up here the global 
framework we work in as well as some notations
that might not be completely standard. 

\subsection*{Topological spaces}

By spaces, we understand compactly generated Hausdorff spaces.
They form a full subcategory of topological spaces that we denote by $\Top$
by slight abuse of notation.
Many useful properties of $\Top$ have been introduced by Steenrod in
\cite{Steen}.
The standard Quillen model structure has then been adapted for it by Hovey
in \cite{Hov}.
It is a convenient category in the sense that the natural curryfication map
\[
\Top(X \times Y, Z) \cong \Top(X, \Top(Y, Z))
\]
is a homeomorphism for any three spaces $X$, $Y$ and $Z$ in $\Top$.
The need to restrict ourselves to such a subcategory arises from the
following fact:
when defining an action of an algebraic structure that it also a topological
space $A$ on a space $X$,
one can ask for the continuity of either $A \times X \rightarrow X$ or $A
\rightarrow \Top(X, X)$.
The homeomorphism above gives the equivalence between these two approaches
and
enables one to go back and forth between both frameworks.
This will be useful when dealing with operadic actions. 

\subsection*{Operations on maps}

Let $f \co A \rightarrow X$ and $g \co B \rightarrow Y$ be maps between
spaces. We use the following notations.
\begin{itemize}
\item $f \amalg g$ is the map between coproducts $A\amalg B \rightarrow
X \amalg Y$.
\item $f \oplus g$ is the map $A \amalg B \rightarrow X$ when $X=Y$.
\item $f \times g$ is the map between products $A\times B \rightarrow X
\times Y$.
\item $(f, g)$ is the map $A \rightarrow X \times Y$ when $A=B$.
\item $A^{\times n}$ is the product of $n$ copies of $A$ and $f^{\times n}$
is the map
$A^{\times n} \rightarrow B^{\times n}$.
\item $B^{\amalg n}$ is the coproduct of $n$ copies of $B$ and $f^{\amalg
n}$ is the map
$A^{\amalg n} \rightarrow B^{\amalg n}$.
\end{itemize}

\subsection*{Smooth manifolds}

When discussing manifolds, we think of usual (possibly bordered) $C^{\infty}$
manifolds.
We write $I=[0, 1]$ for the unit interval, $J=[-1, 1]$ for the
$1$--dimensional unit disk 
and $J^k = J^{\times k}$ for the $k$--dimensional unit cube.
The set of $C^{\infty}$ maps between two manifolds $M$ and $N$ is 
denoted by $C^{\infty}(M, N)$
and topologized with the usual $C^{\infty}$--topology described in
\cite{Hirsch}.
The space of embeddings, immersions, submersions or diffeomorphisms between
manifolds are
topologized as subspaces of the latter.
This turns diffeomorphism groups into topological groups and makes every
composition map continuous. 

\hypertarget{secon}{}
\section{Embedding spaces}
\label{EmbeddingSpaces}

This section aims to review the construction of various spaces of embeddings,
namely spaces of knots and $2$--stranded links.
We start by recalling the definition of the usual space of knots and
introduce three variations:
the long knots $\KK$, the framed long knots $EC(1, D^2)$ and the fat long
knots $\KKhat$.
These spaces are meant to ease algebraic and homotopical manipulations.
We also discuss the classical monoid structure on the space of knots,
its interactions with these spaces and finally adapt these constructions
to $2$--stranded links. 

\subsection{Knot spaces}
\label{EmbeddingSpacesKnot}

The first instance of a space of knots arises as the space of embeddings
$\Emb(S^1, S^3)$.
Its components $\pi_0\Emb(S^1, S^3)$ are the isotopy classes of knots in
the $3$--sphere and are
the central object of study in knot theory.
The class of the standardly embedded circle $S^1 \cof \mathbb{R}^3 \subset
S^3$ is called the trivial knot or unknot.
Given two (isotopy classes of) knots $k_1$ and $k_2$, one can define the
connected sum
$k_1 \hash k_2$ in various ways, as done for instance in \cite{Cromwell}.
Intuitively, $k_1 \hash k_2$ is obtained by cutting open $k_1$ and $k_2$
and closing them back into a single knot. An example is provided in Figure \ref{connectedSum}.
This operation turns out to be 
associative, commutative and unital with
the unknot as unit.
This turns $\pi_0\Emb(S^1, S^3)$ into a commutative monoid.
The nontrivial elements $k$ which admit no nontrivial factorization
$k=k_1 \hash k_2$ are called prime.
They are in a sense the most elementary knots.
However, there are infinitely many of them and a further decomposition
developed in \cite{BudneyJSJ}
suggests that they form a fairly wide class of knots.
The monoid structure on $\pi_0\Emb(S^1, S^3)$ is completely understood
thanks to a theorem of Schubert:

\begin{Thm}[Schubert~\cite{Schubert}]
The monoid $\pi_0\Emb(S^1, S^3)$ is the free commutative monoid generated
by prime knots.
\end{Thm}

\begin{figure}
\small
\def\svgwidth{\hsize}
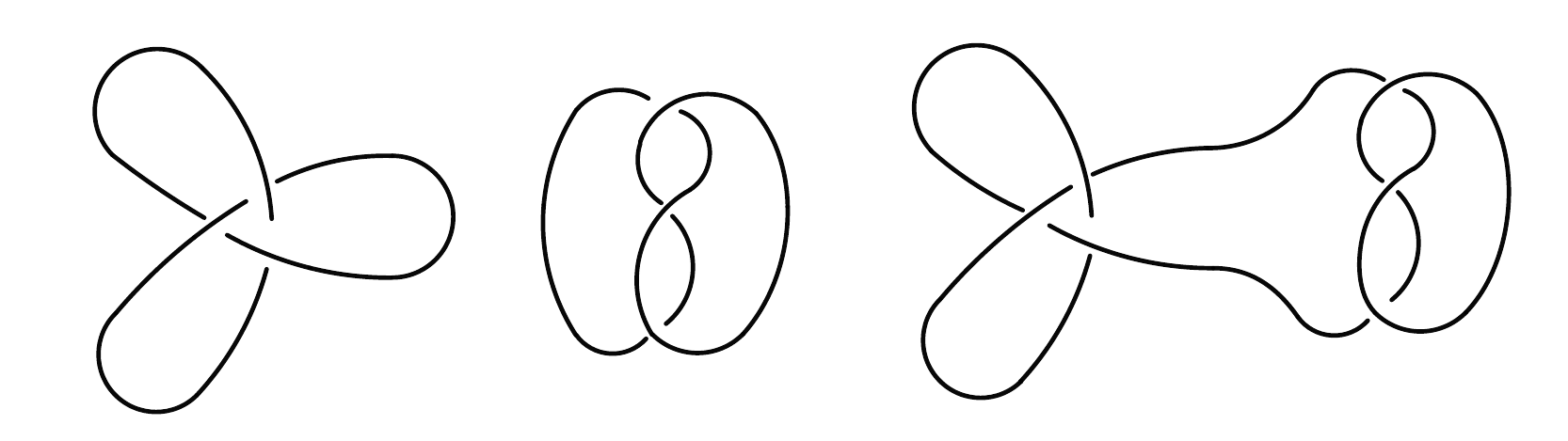
\caption{Illustration of the connected sum of two knots.}
\label{connectedSum}
\end{figure}

We now introduce long knots. They are a mild variation of usual knots
for which
the connected sum is easier to deal with. Let $\imath\co \mathbb{R}\rightarrow
\mathbb{R}^3$, with $\imath(t)= (t, 0, 0)$ be the standard embedding of
the real line in $\mathbb{R}^3$.

\begin{Defn}
A \emph{long knot} is an embedding $\mathbb{R} \cof \mathbb{R}^3$ that
agrees with the standard embedding $\imath$
outside of $J = [-1, 1]$ and maps the interior of $J$ in the interior of
$J \times D^2 \subset \mathbb{R}^3$.
The space of long knots is denoted by $\KK$.
One can alternatively think of a long knot as a proper embedding $J \cof
J \times D^2$
whose values and derivatives at $\del J$ match those of $\imath$.
\end{Defn}

With these conditions on the embeddings, it is natural to define
a binary stacking operation
between long knots as follows. Let $L, R\co \mathbb{R}^3 \rightarrow
\mathbb{R}^3$ be the maps sending
$(x, y, z)$ to $(\frac{x-1}{2}, y, z)$ and $(\frac{x+1}{2}, y, z)$,
respectively.
Given $k_1$ and $k_2$ in $\KK$, we define the concatenation of $k_1$
and $k_2$ as the long knot
that restricts to $t \mapsto L \circ k_1(2t+1)$ on $[-1, 0]$ and to $t
\mapsto R \circ k_2(2t-1)$ on $[0, 1]$.
This operation and its commutativity up to homotopy are illustrated in Figure 
\ref{knotCommutativity}.
We still denote this operation 
by $\hash$ as it is
the analogue of the connected sum in the following sense.
Each long knot is linear outside of $J$ and can therefore be extended to
an embedding $S^1 \cof S^3$
by compactifying the domain and codomain. This specifies an inclusion $\KK
\cof \Emb(S^1, S^3)$
which turns out to be a bijection on $\pi_0$. It is 
easy to verify that
the concatenation of two knots is sent to their connected sum.
The isotopy classes $\pi_0\KK$ therefore inherit a monoid structure. When
$\PP$ denotes the collection of long knots which are prime, Schubert's
theorem applies and gives:

\begin{Thm}
The monoid $\pi_0\KK$ is the free commutative monoid on the basis $\pi_0\PP$.
\end{Thm}

\begin{figure}
\small
\def\svgwidth{\hsize}
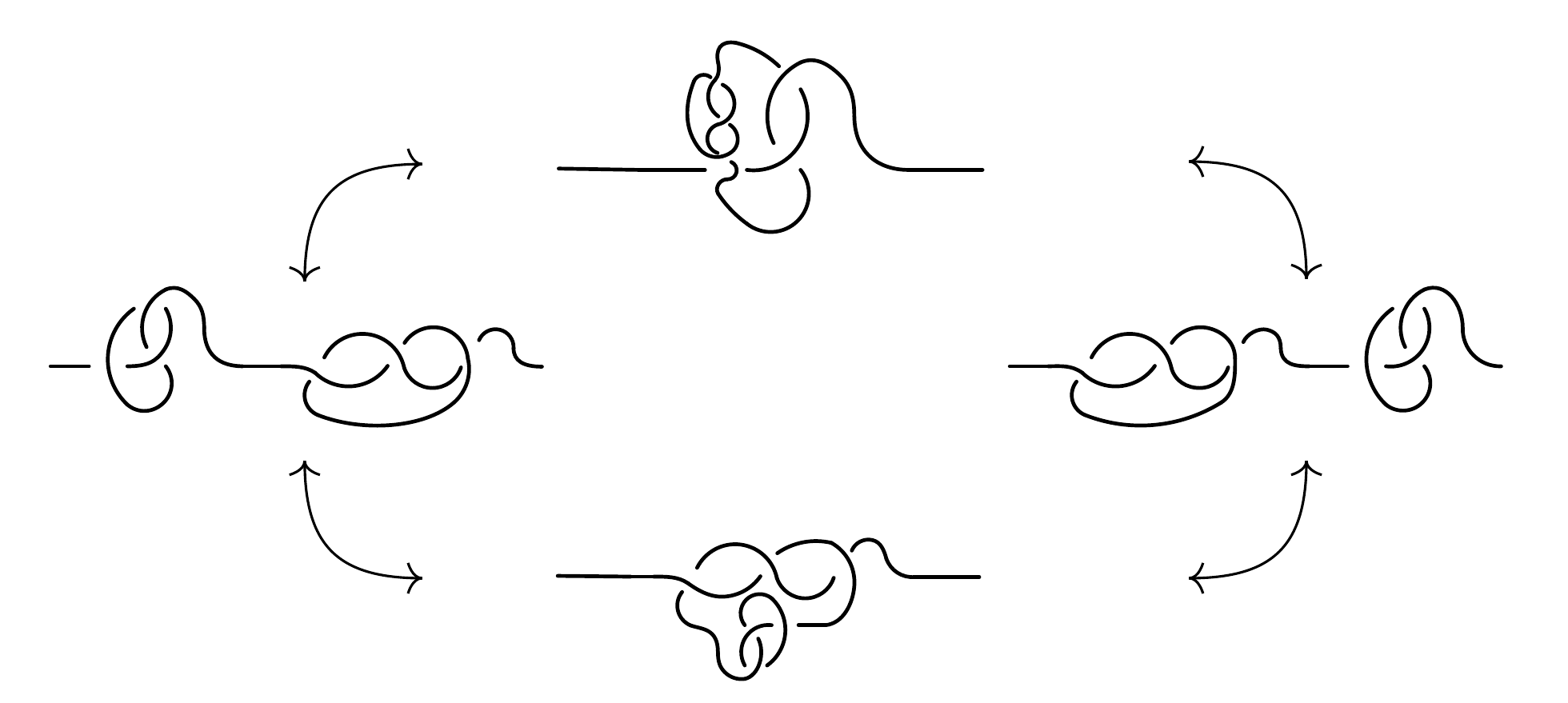
\caption{Illustration of the commutativity in the monoid $\pi_0 \KK$.}
\label{knotCommutativity}
\end{figure}

We now define framed long knots. The latter are meant to approximate the
long knots defined above while
being composable. The idea is to thicken the strand $\mathbb{R}$ into a
long tube $\mathbb{R} \times D^2$.
This definition is due to Budney and originates in \cite{Budney}. We define:

\begin{Defn}[Budney~\cite{Budney}]
A framed long knot is an embedding
$\mathbb{R}\times D^2 \rightarrow \mathbb{R} \times D^2$
that restricts to the identity outside of $J \times D^2$.
The space of framed long knots is denoted by $EC(1, D^2)$.
When $\supp(f)$ denotes the support of an embedding $f\co \mathbb{R} \times
D^2 \cof \mathbb{R} \times D^2$,
ie the closure of 
$\{(t, x) \in \mathbb{R} \times D^2 \mid f(t, x) \neq (t, x)\}$, 
the condition for $f$ to lie in $EC(1, D^2)$
can be reformulated as $\supp(f) \subset J \times D^2$.
\end{Defn}

Note that this space is still equipped with a stacking operation
$\hash$ defined just as in the case
of long knots. It also still descends to an associative, commutative unital
pairing on isotopy classes.
There is a restriction map $EC(1, D^2) \rightarrow \KK$, 
defined by $f \mapsto f|_{\mathbb{R} \times (0, 0)}$, that preserves the
concatenation. But, it does not induce a bijection on~$\pi_0$. To see this,
consider the
diffeomorphism $r$ of $J \times D^2$ that progressively performs a full
turn rotation on the disk factor.
One can parametrize $r\co (t, x) \mapsto (t, e^{i\pi(t+1)}x)$.
This diffeomorphism can be isotoped about $\del J \times D^2$ to be extended
into an element of $EC(1, D^2)$.
Now, for any long knot $k$ and $f$ an extension in $EC(1, D^2)$,
each composite $f \circ r^{\circ n}$ also extends $k$ but no two are
isotopic.
This produces infinitely many components in the fiber over $k$, which
shows that the restriction map
does not induce a bijection on $\pi_0$.
As of here, this twisting 
phenomenon is an unwanted byproduct of the
thickening process.
We will get rid of it when defining $\KKhat$ as an unframed subspace of
$EC(1, D^2)$.
To do so, we need to quantify the framing of a knot, which we do via the
framing number.

\begin{Defn}
\label{DefFramingNumber}
We define the \emph{framing number} $\omega(f)$ of a framed long knot $f$ 
of $EC(1, D^2)$
as the linking number $\lk(f|_{\mathbb{R} \times (0,0)}, f|_{\mathbb{R}
\times (0, 1)})$.
Strictly speaking, the linking number is only defined between disjoint
closed curves.
We deal with this problem by identifying the curves above with their
extension by compactification $S^1 \rightarrow S^3$
and isotoping $f|_{\mathbb{R} \times (0, 1)}$ about the point at infinity
to keep the disjointness.
\end{Defn}

Intuitively, the linking number counts the number of times a
closed curve winds around 
another.
Here, we think of $\omega(f)$ as the number of times a curve on the surface
of the knot
wraps around the core $f|_{\mathbb{R} \times (0,0)}$.
This provides a way to quantify the framing of the elements of $EC(1,
D^2)$ and we have $\omega(r^{\circ n})=n$. Another example is provided in Figure
\ref{framingNumber}.
Since the linking number can be computed by counting crossings in diagrams,
it is easy to see that $\omega$ is additive with respect to the stacking
product.
It is also isotopy invariant and therefore descends to a morphism of monoids
$\pi_0 EC(1, D^2) \rightarrow \mathbb{Z}$. 

\begin{figure}
\small
\def\svgwidth{\hsize}
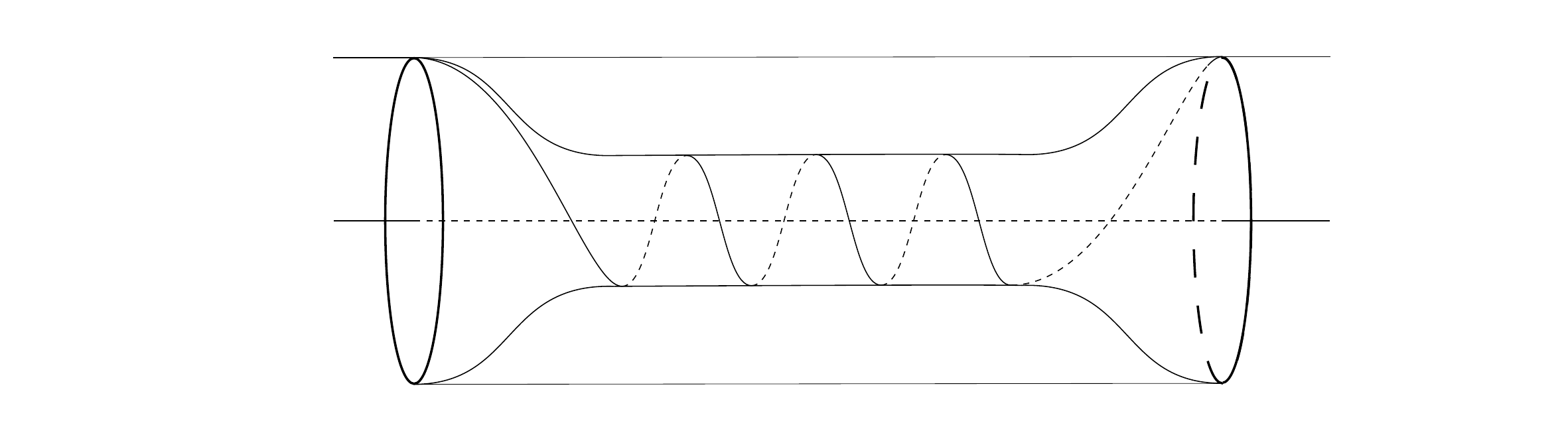
\caption{A framed long knot $f$ with framing number $4$.}
\label{framingNumber}
\end{figure}

We are finally able to define $\KKhat$, our 
preferred model to approximate
$\KK$. We simply set:

\begin{Defn}[Budney \cite{Budney}]
The space of \emph{fat long knots} is the subspace $\KKhat = \omega^{-1}(0)$.
\end{Defn}

Fat long knots are stable under $\hash$ thanks to the additivity
of $\omega$.
The primes in $\KKhat$ are denoted by $\PPhat$.
There still is a restriction map $\KKhat \rightarrow \KK$ which preserves
this structure.
$\KKhat$~is a good approximation for $\KK$ in the sense that:

\begin{Prop}[Budney \cite{Budney}]
The restriction map $\KKhat \rightarrow \KK$ is a homotopy equivalence.
\end{Prop}

In particular, we get an isomorphism on $\pi_0$ which enables
us to 
transfer Schubert's theorem
to fat long knots:

\begin{Cor}
\label{KnotStructureThm}
The monoid $\pi_0\KKhat$ is the free commutative monoid on the basis
$\pi_0 \PPhat$.
\end{Cor}

\subsection{Link spaces}
\label{EmbeddingSpacesLink}

We now adapt these constructions for $2$--links.
Most of this 
generalization work has already been carried out by Burke
and Koytcheff in \cite{KoytcheffInfect}.
The space of usual $2$--links arises as $\Emb(S^1 \amalg S^1, S^3)$.
There is no canonical version of a connected sum operation here, as there
is no preferred strand in each link
for one to merge.
As in the case of long knots, the space of string links $\LL$ deals with
this problem by setting a framework
where a stacking operation is naturally defined.
Let $\imath_2\co \mathbb{R} \amalg \mathbb{R} \cof \mathbb{R}^3$ be the
embedding of two copies of the real line in $\mathbb{R}^3$ mapping the
first copy as $t \mapsto \bigl(t, 0, \tfrac{1}{2}\bigr)$ and the other one as $t
\mapsto \bigl(t, 0, \tfrac{-1}{2}\bigr)$.
We refer to $\imath_2$ as the standard embedding for links with two
strands. We then define:

\begin{Defn}
A \emph{$2$--string link} is an embedding $\mathbb{R} \amalg \mathbb{R}
\cof \mathbb{R}^3$ that agrees with $\imath_2$
outside of $J \amalg J$ and maps the interior of $J \amalg J$ in the
interior of $J \times D^2 \subset \mathbb{R}^3$.
The space of $2$--string links is denoted by $\LL$.
One can alternatively think of a $2$--string link as a proper embedding $J
\amalg J \cof J \times D^2$
whose values and derivatives at $\del J \amalg \del J$ match those of
$\imath_2$.
\end{Defn}

A binary stacking operation $\hash$ can now be defined on $\LL$
as in the case of long knots.
It turns $\pi_0\LL$ into a monoid with unit $\imath_2$.
A $2$--string link is said to be prime if it is not invertible but cannot
be factored without an invertible element.
There is an injection $\LL \cof \Emb(S^1 \amalg S^1, S^3)$ obtained by
closing a truncated link
$f|_{J \amalg J}$ with two fixed smooth curves from $\bigl(-1, \tfrac{ \pm
1}{2}\bigr)$ to $\bigl(1, \tfrac{ \pm 1}{2}\bigr)$
as illustrated
in Figure \ref{stringLinknotLink}. 
But, this inclusion does not induce a bijection on
isotopy classes.
Indeed, there are pairs of nonisotopic $2$--string links that yield isotopic
links once closed as shown in Figure \ref{stringLinknotLink}.
Therefore, studying string links slightly differs from usual link theory.

\begin{figure}
\small
\def\svgwidth{\hsize}
\begingroup%
  \makeatletter%
  \providecommand\color[2][]{%
    \errmessage{(Inkscape) Color is used for the text in Inkscape, but the package 'color.sty' is not loaded}%
    \renewcommand\color[2][]{}%
  }%
  \providecommand\transparent[1]{%
    \errmessage{(Inkscape) Transparency is used (non-zero) for the text in Inkscape, but the package 'transparent.sty' is not loaded}%
    \renewcommand\transparent[1]{}%
  }%
  \providecommand\rotatebox[2]{#2}%
  \newcommand*\fsize{\dimexpr\f@size pt\relax}%
  \newcommand*\lineheight[1]{\fontsize{\fsize}{#1\fsize}\selectfont}%
  \ifx\svgwidth\undefined%
    \setlength{\unitlength}{6236.22047244bp}%
    \ifx\svgscale\undefined%
      \relax%
    \else%
      \setlength{\unitlength}{\unitlength * \real{\svgscale}}%
    \fi%
  \else%
    \setlength{\unitlength}{\svgwidth}%
  \fi%
  \global\let\svgwidth\undefined%
  \global\let\svgscale\undefined%
  \makeatother%
  \begin{picture}(1,0.34090909)%
    \lineheight{1}%
    \setlength\tabcolsep{0pt}%
    \put(0,0){\includegraphics[width=\unitlength,page=1]{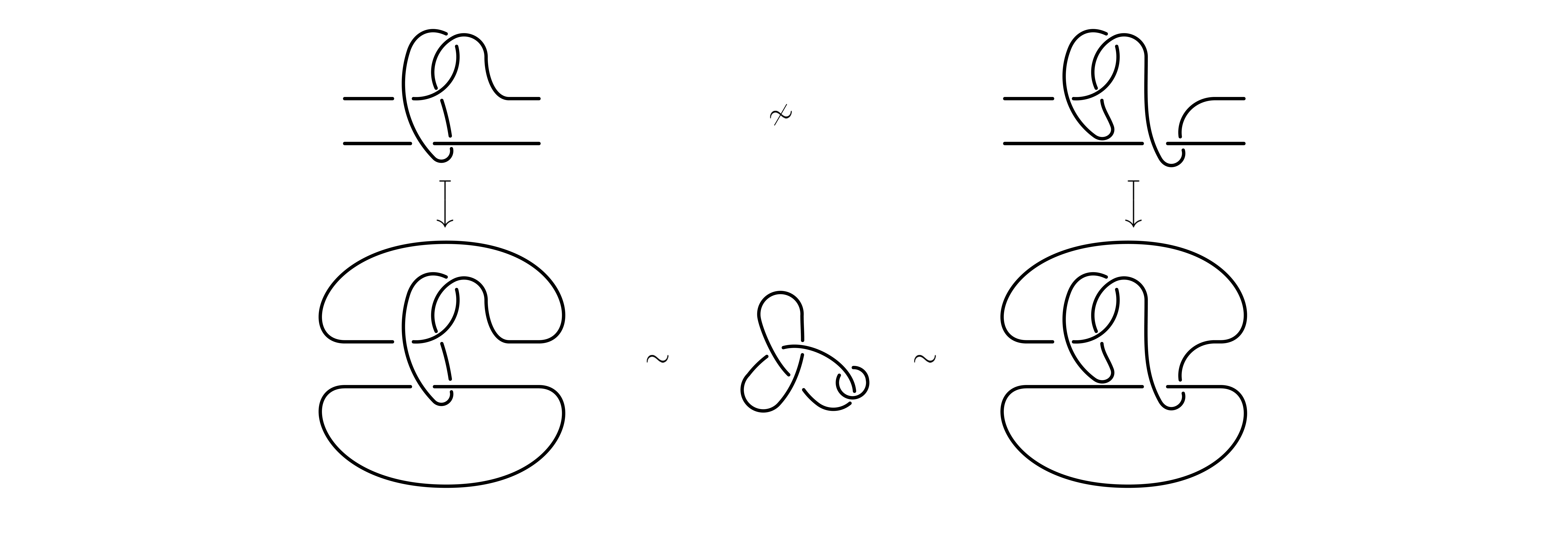}}%
    \put(0.49527211,0.26787899){\color[rgb]{0,0,0}\makebox(0,0)[lt]{\lineheight{1.25}\smash{\begin{tabular}[t]{l}$\not\sim$\end{tabular}}}}%
    \put(0.41447232,0.11158787){\color[rgb]{0,0,0}\makebox(0,0)[lt]{\lineheight{1.25}\smash{\begin{tabular}[t]{l}$\sim$\end{tabular}}}}%
    \put(0.58501176,0.11158787){\color[rgb]{0,0,0}\makebox(0,0)[lt]{\lineheight{1.25}\smash{\begin{tabular}[t]{l}$\sim$\end{tabular}}}}%
    \put(0.28371271,0.20680058){\color[rgb]{0,0,0}\makebox(0,0)[lt]{\lineheight{1.25}\smash{\begin{tabular}[t]{l}$\downmapsto$\end{tabular}}}}%
    \put(0.72267928,0.20680058){\color[rgb]{0,0,0}\makebox(0,0)[lt]{\lineheight{1.25}\smash{\begin{tabular}[t]{l}$\downmapsto$\end{tabular}}}}%
  \end{picture}%
\endgroup%

\caption{The two top string links are not isotopic (closing the first one
up with two vertical lines
results in a trivial knot, doing so with the second one yields a trefoil
knot).
The corresponding links are however isotopic.}
\label{stringLinknotLink}
\end{figure}

Let us spend some time to investigate $\pi_0\LL$.
We first identify the invertible elements.
There is a canonical map from the pure braid group on two strands to
$\pi_0\LL$ sending a pure braid
to its isotopy class as a string link. It is a morphism of monoids that
only maps to units in $\pi_0\LL$ since
braids form a group. This association is easily 
shown to be injective:
the linking number map $\lk\co \LL \rightarrow \Emb(S^1 \amalg S^1, S^3)
\rightarrow \mathbb{Z}$
descends to a left inverse when one identifies the pure braids on two
strands with the integers in the natural way.
This provides a whole collection of invertible elements and it turns out
that every unit in $\pi_0\LL$ is
of this form. Observe as well that these invertible links commute with
every other link: an isotopy
exhibiting this relation is suggested by Figure \ref{braidsCommute}.
Let now $\LL^0$ be the preimage of $0$ through the linking number map $\lk$.
The injection of the braid group provides a section in the short exact
sequence
\[
\pi_0\LL^0 \rightarrow \pi_0\LL \rightarrow \mathbb{Z}.
\]
Thus $\pi_0\LL$ splits as $\pi_0\LL^0 \times \mathbb{Z}$ and we can focus
on the first factor.

\begin{figure}
\small
\def\svgwidth{\hsize}
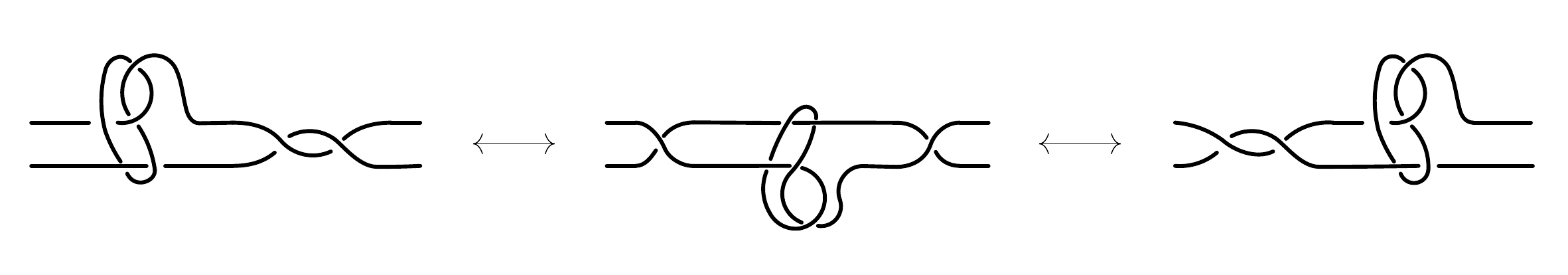
\caption{Illustration of the commutation between a braid and an arbitrary
string link.}
\label{braidsCommute}
\end{figure}

The monoid $\pi_0\LL^0$ is not commutative but it contains several copies
of $\pi_0\KK$ in its center.
Indeed, consider the injective morphism $\varphi^\up\co \pi_0\KK \rightarrow
\pi_0\LL^0$ mapping
a long knot $k$ to the string link whose upper strand is knotted according
to $k$ and does not interact with
the unknotted lower strand. An illustration of $\varphi^\up$ is given in Figure 
\ref{Phi}.
The image of $\varphi^\up$ is a copy of $\pi_0\KK$ lying in $\pi_0\LL^0$, and
one can build a similar morphism $\varphi^\down$ by switching the roles
of the strands.
A third copy of the knot monoid can be found as follows.
Consider the morphism $\varphi^\updown$ that sends a knot $k$ to the link
whose strands are parallel
and unlinked but knotted according to $k$.
This $\varphi^{\updown}$ maps to a third copy of $\pi_0 \KK$ in $\pi_0
\LL^0$.

\begin{figure}[b]
\small
\def\svgwidth{\hsize}
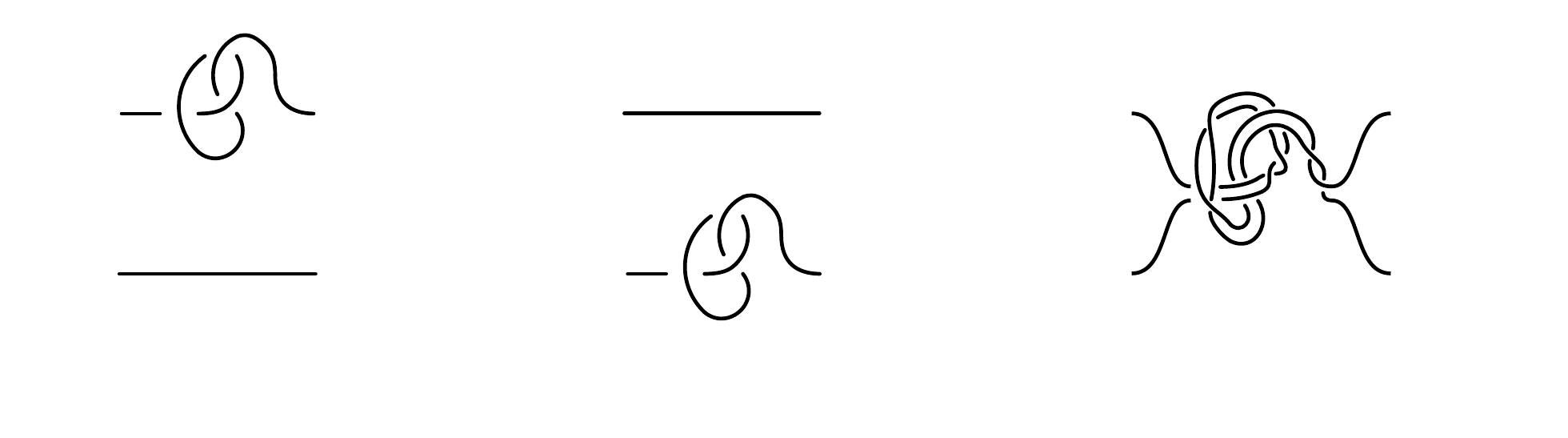
\caption{Illustration of the morphisms $\varphi^s$, $s \in \{\up, \down,
\updown\}$.}
\label{Phi}
\end{figure}

The images of any two of these morphisms intersect only in the component
of $\imath_2$ so that
$\pi_0\KK ^{\times 3}$ actually lives in $\pi_0\LL^0$.
Every string link in the image of a $\varphi^s$, for 
$s \in \{\up, \down, \updown\}$, commutes with any other link.
The structure theorem for $2$--string links proved by Blair, Burke and
Koytcheff in \cite{KoytcheffPrime}
actually shows that the center of $\pi_0\LL^0$ is generated by the images
of the maps $\varphi^s$,
alongside 
the fact that the remaining links are freely generated by
some prime elements.
More precisely, when $\QQ$ denotes the prime $2$--string links in $\LL$
that do not belong to the image of a
$\varphi^s$ and when $\QQ^0=\QQ \cap \LL^0$, one has the following result:

\begin{Thm}[Blair, Burke and Koytcheff~\cite{KoytcheffPrime}]
\label{LinkStructureThm}
The monoid $\pi_0\LL^0$ is isomorphic to the product of $\pi_0\KK^{\times
3}$ and the free
(noncommutative) monoid on the basis $\pi_0\QQ^0$.
Moreover, an 
isomorphism is induced by the inclusion
$\pi_0\QQ^0 \cof \pi_0\LL^0$ and the maps~$\varphi^s$.
\end{Thm}

The string links generated by the images of $\varphi^\up$ and
$\varphi^\down$ are called split.
Such a string link can also be characterized by the existence of a properly
embedded disk separating the two strands
in the complement. The string links generated by the image of
$\varphi^\updown$ and invertible elements
are called double cables. They can alternatively be defined as the links
whose strands are parallel.
Note that Theorem \ref{LinkStructureThm} above tells us that the center
of $\pi_0\LL$
precisely consists of the split links, double cables and their products,
and that any other string link only commutes with central elements. 

We now introduce framed $2$--string links. They are the $2$--stranded analogue
of framed long knots as
they arise by thickening the two strands.
Let $\iota \co D^2 \amalg D^2 \cof D^2 $ be the embedding that rescales the
disks to make their
radii $\tfrac{1}{8}$ and translates them so that they are centered at $\bigl(0,
\tfrac{\pm 1}{2}\bigr)$.
We refer to $\id_{\mathbb{R}} \times \iota$ as the standard embedding
$(\mathbb{R} \times D^2) \amalg (\mathbb{R} \times D^2) \cof \mathbb{R}
\times D^2$.

\begin{Defn}
\label{DefFramedStringLink}
A \emph{framed $2$--string link} is an embedding
$$(\mathbb{R} \times D^2) \amalg (\mathbb{R} \times D^2) \cof \mathbb{R}
\times D^2$$
that restricts to the standard embedding outside of $(J \times D^2) \amalg
(J \times D^2)$
and maps the interior of $(J \times D^2) \amalg (J \times D^2)$ in the
interior of $J \times D^2$.
The space of framed $2$--string links is denoted by $EC_\iota(1, D^2)$.
When $\supp_\iota(f)$ denotes the closure of
$\{(t, x)  \in (\mathbb{R} \times D^2) \amalg (\mathbb{R} \times D^2) \mid
f(t, x) \neq (t, \iota(x))\}$ for any embedding $f$,
the condition for $f$ to lie in $EC_\iota(1, D^2)$ can be reformulated as
$\supp_\iota(f) \subset (J \times D^2) \amalg (J \times D^2)$ and
$f(\inter((J \times D^2) \amalg (J \times D^2))) \subset J \times D^2$.
\end{Defn}

Framed $2$--string links again dispose of a binary concatenation
operation $\hash$
and a restriction map $EC_\iota(1, D^2) \rightarrow \LL$ preserving it.
This endows the isotopy classes $\pi_0 EC_\iota(1, D^2)$ with a monoid
structure with the standard
embedding as unit.
An obstruction for the restriction map to be a homotopy equivalence is
again the framing of each strand.
We define the framing number $\omega$ of an element of $EC_\iota(1, D^2)$
as in Definition \ref{DefFramingNumber}, except that it now consists of a 
pair of integers,
one for each strand.

\begin{Defn}
Let $f$ be a framed $2$--string link with strands $f^\up$ and $f^\down$.
We define the upper framing number $\omega_\up(f)$
as the linking number $\lk({f^\up}|_{\mathbb{R} \times (0,0)},
{f^\up}|_{\mathbb{R} \times (0, 1)})$.
The lower framing number $\omega_\down(f)$ is defined the same way and
the whole framing number
$\omega(f)$ is the pair of integers $(\omega_\up(f), \omega_\down(f))$.
\end{Defn}

\begin{figure} 
\small
\def\svgwidth{\hsize}
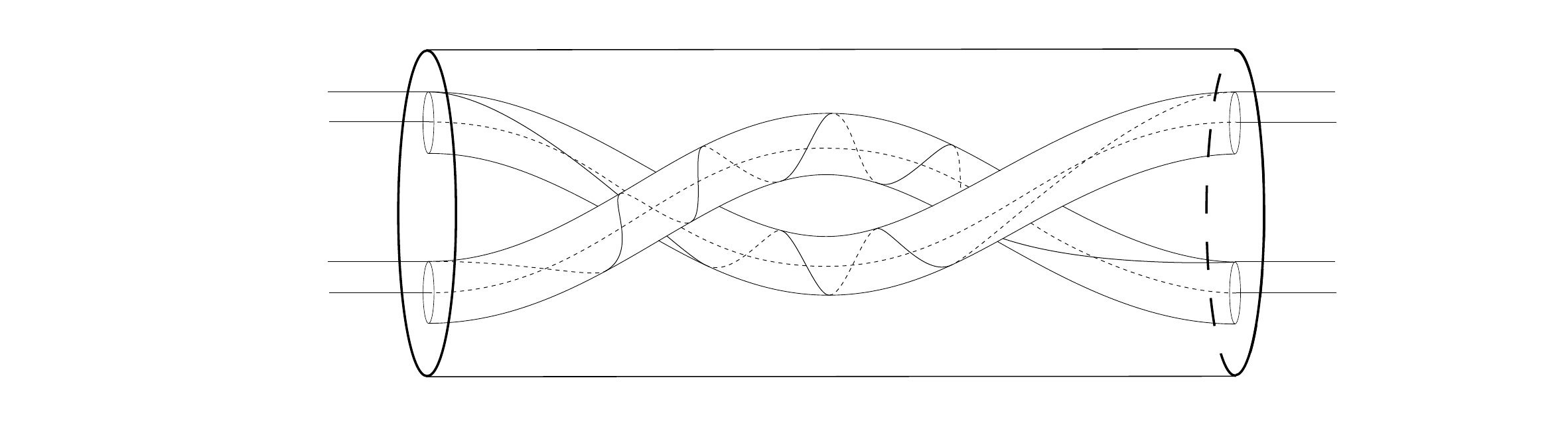
\caption{A framed $2$--string link $f$ with framing number $(3, -5)$.}
\end{figure}

The framing number is isotopy invariant and additive with respect
to the concatenation for the same
reasons as before. This makes it descend to a monoid morphism
$\pi_0 EC_\iota(1, D^2) \rightarrow \mathbb{Z}^{\times 2}$.
We are now able to get rid of this twisting 
phenomenon by defining:

\begin{Defn}
The space of \emph{fat $2$--string links} $\LLhat$ is the subspace
$$\omega^{-1}(0, 0) \subset EC_\iota(1, D^2).$$
\end{Defn}

Fat $2$--string links are stable under concatenation. We denote
by $\LLhat^0$, $\QQhat$ and $\QQhat^0$
the elements of $\LLhat$ whose restrictions to $(\mathbb{R} \times
(0,0))\amalg (\mathbb{R} \times (0,0))$ lie in
$\LL^0$, $\QQ$ and $\QQ^0$, respectively. $\LLhat$ is a good approximation
for $\LL$ in the sense that:

\begin{Prop}[Burke and Koytcheff~\cite{KoytcheffInfect}]
The restriction map $\LLhat \rightarrow \LL$ is a homotopy equivalence.
\end{Prop}

In particular, the monoid $\pi_0\LLhat$ is completely understood
and has the same structure as
$\pi_0\LL$, 
made explicit in Theorem \ref{LinkStructureThm}. 

The remainder of this paper is dedicated to the elaboration of an algebraic
structure on the space level of
long knots and string links. The stacking products are examples of binary
operations on the space level
that relate to the monoids $\pi_0 \KK$ and $\pi_0 \LL$.
We aim to find a refinement of these operations into a more subtle structure,
in order to generalize 
Theorems~\ref{KnotStructureThm}
and~\ref{LinkStructureThm}
to the space level.
These structures will be formalized as operadic actions and the isomorphisms
described in the theorems above
will generalize as equivariant homotopy equivalences.
Budney's Theorem~11 in \cite{Budney} precisely answers this problem in
the case of knots,
and Burke and Koytcheff's Theorem~6.8 in \cite{KoytcheffInfect} partially
deals with the case of $2$--string links.
In the following sections, we recall the work presented in these two papers
and treat the case of $2$--string links in a wider manner. 

\hypertarget{sectw}{}
\section{Operads}
\label{Operads}

The purpose of this section is to recall the definition of (colored)
operads, set up some notations and introduce
the two operads of prime interest in this paper: the little cubes operad
$\cubes{n}$ and a $4$--colored version
of the Swiss-cheese operad that we call $\SCL$ for ``Swiss-cheese for
links''.
We also discuss the different types of algebras these objects encode and
review free models for these structures. 

\subsection{Colored operads and their algebras}

We start by reviewing the notion of colored operad. Let $S$ be a set
of colors.
We denote by $S^\star$ the collection of tuples of elements of $S$.
In other words, $S^\star$ is the union $\coprod_{n \geq 0} S^{\times n}$.
Each $S^{\times n}$ is naturally a right $\Sigma_n$--space. The length of
a tuple $\ult \in S^\star$ is denoted by $|\ult|$ and, for every $s \in S$, $|\ult|_s$ is the number of times $s$
appears in $\ult$.
Given two tuples of colors $\ult$, $\ulu$ and an integer $i \leq |\ult|$,
we denote by
$
\ult \circ_i \ulu$ the
$(|\ult|{+}|\ulu|{-}1)$--tuple
$$
\ult \circ_i \ulu=(t_1, \dots, t_{i-1}, u_1, \dots, u_{|\ulu|}, t_{i+1},
\dots, t_{|\ult|}).
$$
Note that $|\ult \circ_i \ulu|_s=|\ult|_s+|\ulu|_s$ whenever $s \neq
t_i$ and $|\ult|_s+|\ulu|_s-1$
for $s=t_i$. We write $s^n$ for the tuple $(s, \dots, s) \in S^{\times n}$
for every $s\in S$ and $n\geq 0$.
We are 
now in the right framework to define:

\begin{Defn}
A \emph{colored operad} $\OO$ over the colors $S$ consists of the following
combined data:
\begin{enumerate}[leftmargin=*, label=(\roman*), font=\itshape]
\item for every $\ult \in S^\star$ and $s \in S$, a space $\OO(\ult; s)$;
\item for every $\ult$, $\ulu \in S^\star$ and $i \leq |\ult|$,
operadic compositions
\[
\circ_i\co \OO(\ult; s) \times \OO(\ulu; t_i) \rightarrow \OO(\ult
\circ_i \ulu; s)
\]
satisfying the usual associativity, symmetric and unital conditions -- the
latter are thoroughly detailed in \cite{Yau};
\item for every permutation $\sigma \in \Sigma_n$, $\ult$ an $n$--tuple
of colors and $s$ a color, a map
\[
\sigma^\ast\co \OO(\ult; s) \rightarrow \OO(\ult\sigma; s)
\]
such that $\tau^\ast \circ \sigma^\ast = (\sigma \circ \tau)^\ast$ and
$\id^\ast=\id$;
\item  for every color $s$, a unit $1_s \in \OO(s; s)$.
\end{enumerate}
\unskip
The elements of $\OO(\ult; s)$ are called operations with inputs $\ult$
and output $s$.
The units~$1_s$ are sometimes 
referred to as
identities. 
We also often write $a\sigma$ for $\sigma^\ast(a)$ and $a \circ_i b$
for $\circ_i(a, b)$.
A morphism of colored operads $f\co \OO \rightarrow \PP$ is a collection
of maps
$f_{(\ult; s)} \co \OO(\ult; s) \rightarrow \PP(\ult; s)$ preserving
operadic compositions,
symmetric actions and units.
\end{Defn}

\begin{Ex}
The prototypical examples of colored operads are the endomorphism operads.
Let $X=(X_s)_{s \in S}$ be an $S$--tuple of spaces. Given a vector of
colors $\ult$,
we write $X^{\times \ult}$ for the product $\prod_i X_{t_i}$. Now, for
any output color $s$,
we set $\EE_X(\ult; s)$ to be the mapping space $\Top(X^{\times \ult},
X_s)$.
An element $\sigma \in \Sigma_{|\ult|}$ can act on an $f \in \EE_X(\ult;
s)$ by permuting its entries,
resulting in an element of $\EE_X(\ult\sigma; s)$. Also, when $\ulu$
is another vector of colors and
$g$ lies in $\EE_X(\ulu; t_i)$, one can inject $g$ into the $i^{\text{th}}$
entry of $f$ to get the composite
$f \circ_i g \in \EE_X(\ult \circ_i \ulu; s)$. This specifies the data
of a colored operad $\EE_X$ on the colors $S$.
\end{Ex}

\begin{Defn}
A colored operad $\OO$ over a single color $s$ is called an \emph{operad}. In
this case, we write $\OO(n)$ for
$\OO(s^n; s)$ and call it the space in arity $n$. The unit $1_s$ is simply
denoted by $1$.
The operadic compositions are now maps $\OO(n) \times \OO(m) \rightarrow
\OO(n+m-1)$ and
the symmetric structure turns each $\OO(n)$ into a right $\Sigma_n$--space.
\end{Defn}

Operads are useful to specify categories of algebraic objects.
This is formalized via operadic actions which we define now.

\begin{Defn}
Let $\OO$ be a colored operad over the set of colors $S$ and $X=(X_s)_{s
\in S}$
an $S$--tuple of spaces. We say that $X$ is an \emph{$\OO$--algebra} if it
comes with a morphism of operads
$\kappa\co \OO \rightarrow \EE_X$. In other words, an \emph{$\OO$--algebra
structure} on $X$ is a collection of
maps
\[
\kappa_{(\ult; s)} \co \OO(\ult; s) \rightarrow \Top(X^{\times \ult},
X_s)
\]
preserving the operadic compositions, identities and symmetric actions
described in~\cite{Yau}.
They may also be thought of as maps $\OO(\ult; s) \times X^{\times \ult}
\rightarrow X_s$,
and we shall use each framework when it is more convenient.
A \emph{morphism of $\OO$--algebras} $f\co X \rightarrow Y$ is a collection
of maps $f_s\co X_s \rightarrow Y_s$ preserving the operadic actions.
\end{Defn}

\begin{Ex}
\label{Comm}
Consider the operad obtained with the one point space in every arity and
trivial symmetric actions and
operadic compositions. An action of this operad on a space $X$ is the data
of a single map
$X^{\times n} \rightarrow X$ for every nonnegative $n$.
One readily checks that the required relations listed in \cite{Yau}
correspond to the associativity and
commutativity of $X^{\times 2} \rightarrow X$, as well as the fact that the
element specified by the zeroth map
$X^{\times 0} \rightarrow X$ acts as a unit. In other words, an action of
this operad on $X$ is a
commutative topological monoid structure on $X$. This justifies the
terminology $\Com$ for this operad.
\end{Ex}

\begin{Ex}
\label{Ass}
Consider the operad whose space in arity $n$ is the discrete symmetric
group $\Sigma_n$ as an evident
right $\Sigma_n$--space with the following operadic composition. For every
$\sigma \in \Sigma_n$,
$\tau \in \Sigma_m$ and $i \leq n$, $\sigma \circ_i \tau$ permutes $\{1,
\dots, n+m-1\}$ according to $\sigma$
while treating $\{i, \dots, i+m-1\}$ as a single block, then shuffles the
latter internally according to $\tau$.
An action of this operad on a space $X$ is a data of a map $X^{\times n}
\rightarrow X$ for every
ordering of $\{1, \dots, n\}$. One readily checks that the required
relations correspond to the associativity of
$X^{\times 2} \rightarrow X$ and the fact that the element $X^{\times 0}
\rightarrow X$ acts as a unit.
In other words, the algebras over this operad are the 
not necessarily
commutative topological monoids.
This operad is called the associative operad and is denoted by~$\As$. 
\end{Ex}

\subsection{Free algebras}
\label{OperadsFreeAlg}

Before we introduce the two operads that will act on the spaces of knots
and links, we take some time to discuss
free algebras. When $\OO$ is a colored operad over the set of colors $S$,
the $\OO$--algebras and their morphisms form a category denoted by $\Alg{\OO}$.
There is a forgetful functor $\mathcal{U}\co \Alg{\OO} \rightarrow \Top^{\times
S}$ mapping an
$\OO$--algebra to its underlying $S$--tuple of spaces.
By free $\OO$--algebra, we understand the left adjoint
$\OO[\underscore]$ to the forgetful functor $\mathcal{U}$. In other words,
the free $\OO$--algebra generated by
$X=(X_{s})_{s\in S}$ should lead to a bijection
\[
\Alg{\OO}(\OO[X], Y) \cong \Top^{\times S}(X, \mathcal{U}(Y))
\]
for every $\OO$--algebra $Y$. A well-known model for $\OO[X]$ is obtained
as follows.
For every vector $\mathbold{x}=(x_1, \dots, x_n) \in X^{\times \ult}$
and permutation $\sigma$,
we 
will write $\sigma \mathbold{x}$ for $(x_{\sigma^{-1}(1)}, \dots,
x_{\sigma^{-1}(n)}) \in X^{\times \ult\sigma^{-1}}$. We set
\[
\OO[X]_s= \faktor{\coprod_{\ult} \OO(\ult; s) \times X^{\times
\ult}}{\sim},
\]
where $\sim$ identifies each $(a, \mathbold{x})$ with $(a\sigma,
\sigma^{-1}\mathbold{x})$ for every permutation $\sigma$.
The action of $\OO$ is obtained by composing in the $\OO(\ult;s)$ factor.
The desired bijection above is then a formal verification.
When $\OO$ is an uncolored operad, $\sim$ corresponds to the
$\Sigma_n$--orbits and we can simplify
\[
\OO[X]=\coprod_n \OO(n) \times_{\Sigma_n} X^{\times n}.
\]
We conclude this paragraph with a quick observation about free algebras.
When $\OO$ is a colored operad with set of colors $S$, the components
$\pi_0\OO$ naturally inherit an
operad structure. Similarly, if $X=(X_{s})_{s \in S}$ is an $\OO$--algebra,
then the components $\pi_0 X=(\pi_0 X_{s})_{s \in S}$ inherit a
$\pi_0\OO$--algebra structure. Finally, the following result will come in
handy when proving
that some actions yield free algebras in Section \ref{FreenessResults}.

\begin{Prop}
\label{pizeFreenessCommute}
Let $X$ be an $S$--tuple of spaces. Then, $\pi_0(\OO[X])$ is a model for
the free algebra $\pi_0\OO[\pi_0X]$.
\end{Prop}

\begin{proof}
We have the two models
\[
\pi_0(\OO[X]_s)=\pi_0\biggl(\faktor{\coprod_{\ult} \OO(\ult; s)
\times X^{\times \ult}}{\sim}\biggr), \quad
\pi_0\OO[\pi_0X]_s=\faktor{\coprod_{\ult} \pi_0\OO(\ult; s)
\times \pi_0X^{\times \ult}}{\sim}.
\]
Since $\pi_0$ commutes with products and coproducts, the right-hand side
is equal to the quotient of
$\pi_0\bigl(\coprod_{\ult} \OO(\ult; s) \times X^{\times \ult}\bigr)$ by the
relations
$[a, \mathbold{x}] \sim [a\sigma, \sigma^{-1} \mathbold{x}]$. The left-hand
side also matches this description so both spaces are
the same.
It is a tautological verification to see that the action of $\pi_0\OO$
is the same under these identifications. 
\end{proof}

\subsection{The little cubes operad}

We introduce the little cubes operad $\cubes{n}$ and quickly discuss the
algebras it encodes.
It is an uncolored operad originated in \cite{May} in order to understand
iterated loop spaces.
Our treatment is very similar to the one of Budney in \cite{Budney}.

\begin{Defn}
The real functions of the form $x \mapsto ax+b$ for some positive $a$
are said to be affine increasing.
A \emph{little $n$--cube} is an application $L \co J^n \rightarrow J^n$
of the form
$L=l^1 \times \cdots \times l^n$ for some affine increasing functions $l^i$.
The space of $k$ \emph{overlapping little $n$--cubes} $\overcubes{n}(k)$
is the set of configurations of
$k$ little $n$--cubes $J^n \amalg \cdots \amalg J^n \rightarrow J^n$.
We set $\overcubes{n}(0)$ to be the one point space. Given an element $L
\in \overcubes{n}(k)$,
we write its decomposition in little $n$--cubes $L=\bigoplus_i L^i$.
Each $L^i$ decomposes uniquely in affine increasing functions $l^{i,1}
\times \cdots \times l^{i,n}$,
so writing $l^{i, j} \co x \mapsto a_{i, j}x + b_{i, j}$ gives rise to
an injection
$\overcubes{n}(k) \cof \mathbb{R}^{2nk}\co L \mapsto (a_{1, 1}, b_{1, 1},
\cdots, a_{k, n}, b_{k, n})$.
This is used to transfer a topology on $\overcubes{n}(k)$.
Considering the $C^{\infty}$--topology actually has the same outcome.
\end{Defn}

An element of $\overcubes{n}(k)$ is represented by a drawing of
the images of its little $n$--cubes.
We now equip the family of spaces $\overcubes{n}(k)$ with an operadic
structure.
Thereafter, we define the little $n$--cubes operad as a suboperad by adding
a disjointness conditions on the cubes.

\begin{Defn}
\label{OverLittleCubesOperad}
The \emph{overlapping little $n$--cubes operad} $\overcubes{n}$ is the
operad specified as follows.
\begin{enumerate}[leftmargin=*, label=(\roman*), font=\itshape]
\item The space in arity $k$ is the set of configurations  of little
$n$--cubes $\overcubes{n}(k)$.
\item For every positive integer $k$, $l$ and $i \leq k$, the operadic
composition is given by
\[
\begin{gathered}
\circ_i \co \overcubes{n}(k) \times \overcubes{n}(l) \rightarrow
\overcubes{n}(l+k-1); \\
(L, P)
\mapsto
L^1 \oplus \cdots \oplus L^{i-1} \oplus L^i \circ P^{1} \oplus \cdots
\oplus L^i \circ P^{l} \oplus L^{i+1} \oplus \cdots \oplus L^k.
\end{gathered}
\]
Composing $L \in \overcubes{n}(k)$ with the one point in $\overcubes{n}(0)$
discards the $i^{\text{th}}$ cube of $L$.
\item The action of $\sigma \in \Sigma_l$ on $L \in \overcubes{n}(l)$
permutes the little $n$--cubes of $L$,
ie 
$$L\sigma= \bigoplus_i L^{\sigma(i)} . $$
\item The unit is the identity little $n$--cube $\id_{J^n} \in
\overcubes{n}(1)$.
\end{enumerate}
\end{Defn}

\begin{figure} 
\small
\def\svgwidth{\hsize}
\input{operadicCompositionC2Infinity.pdf_tex}
\caption{Illustration of the operadic composition $\circ_2\co \overcubes{2}(4)
\times \overcubes{2}(2) \rightarrow \overcubes{2}(5)$.}
\end{figure}

\begin{Defn}
\label{LittleCubesOperad}
Two little $n$--cubes are said to be \emph{almost disjoint} if the interiors
of their images are disjoint.
The space of $k$ little $n$--cubes $\cubes{n}(k)$ is the subspace of
$\overcubes{n}(k)$ consisting of
pairwise almost disjoint little $n$--cubes. We still set $\cubes{n}(0)$
to be the one point space.
Operadic compositions in the overlapping little cubes operad preserve the
property of being almost disjoint so
the subspaces $\cubes{n}(k)$ forms a suboperad $\cubes{n}$ called the
little $n$--cubes operad.
\end{Defn}

\begin{Rem}
Budney defines in this article \cite{BudneySplicing} an operad
$\cubes{n}^\prime$ which he calls
``the operad of overlapping little $n$--cubes''. The resemblance with our
terminology for $\overcubes{n}$
is only coincidental. The two objects are not equivalent
(Budney's $\cubes{n}^\prime$ is equivalent to $\cubes{n+1}$,
while $\overcubes{n}$ has contractible underlying spaces).
The operad $\cubes{n}^\prime$ does not appear in this article, so there
should be no confusion.
\end{Rem}

We conclude this paragraph by taking a look at the operad $\pi_0 \cubes{n}$.
Given $k$ little $n$--cubes $L \in \cubes{n}(k)$, restricting $L$ to the
center of each cube
leads to an injective map $\{1, \dots, k\} \cof J^n$, ie an element of
the configuration space $\conf_k(J^n)$.
Conversely, given $k$ points $\ulx$ in $\conf_k(J^n)$, one gets an element
of $\cubes{n}(k)$ by
considering the identical cubes centered at $\ulx$ whose size is the
maximal size that keeps them almost disjoint.
Intuitive straight line homotopies show that these two maps are homotopy
inverses, so that the homotopy type
of $\cubes{n}(k)$ is the one of $\conf_k(J^n)$.
In particular, when $n>1$, each $\cubes{n}(k)$ is path connected
so $\pi_0\cubes{n}=\Com$ from Example \ref{Comm}.
The free algebras over this operad are the free commutative monoids.
In dimension $1$, the isotopy classes of $\cubes{1}(k)$ are indexed by
the orderings of $\{1, \dots, n\}$,
so $\pi_0\cubes{1}=\As$ from Example \ref{Ass}.
The free $\As$--algebras are the free monoids. 

\subsection{The operad $\SCL$} 

We go through the construction of the Swiss-cheese operad for links $\SCL$.
It is a $4$--colored operad that is also defined in terms of little cubes.
This terminology is motivated by the fact that $\SCL$ restricts to the
$2$--colored Swiss-cheese operad on
several pairs of colors. We then conclude by investigating the operad of
components $\pi_0\SCL$. 

\begin{Defn}
A little $n$--cube $L=l^1\times \cdots \times l^n$ is said to \emph{meet
the lower face of the unit cube}
if $l^n(-1)=-1$. Visually, this happens when the image of $L$ intersects
$J^{n-1} \times -1 \subset J^n$.
The configurations of $k$ almost disjoint little $n$--cubes meeting the
lower face of $J^n$ is
denoted by $\cubeso{n}(k)$ and these spaces form an operad $\cubeso{n}$
just as in the case of $\cubes{n}$.
\end{Defn}

Consider the set of four colors $S=\{o, \up, \down, \updown\}$.
The notation $o$ is meant to remind one of the ``open'' color in the 
Swiss-cheese operad as it will play a similar
role. The other symbols call up to the upper and lower strands of a
string link.

\begin{Defn}
The Swiss-cheese operad for links $\SCL$ is specified as follows.
\begin{enumerate}[leftmargin=*, label=(\roman*), font=\itshape]
\item For $s \in \{\up, \down, \updown\}$, the only inputs $\ult$ that
do not lead to an empty $\SCL(\ult; s)$
are the monochromatic ones such that $\ult=s^n$. In that case, we set
$\SCL(s^n; s)=\cubes{2}(n)$.
When $s=o$, we set
\(\SCL(\ult; o)\) to be those \(L \in \overcubes{2}(|\ult|)\) such that 
each $L^i$ with $t_i=o$ meets the lower face of $J^n$ 
while being almost disjoint from any other cube, 
the $L^i$ with $t_i= \up$ 
are almost disjoint from each other and
the same holds for $t_i = \down$ and $\updown$).
\item The operadic compositions, symmetric actions and units are inherited
from $\overcubes{2}$.
\end{enumerate}
\end{Defn}

An element of $\SCL(\ult; s)$ is represented by a drawing of
the images of its little $n$--cubes.
We decorate the numbering of each little cube with its associated color
to distinguish the cubes that simply
happen to meet the lower face of $J^n$ from those that have to.
The output color also appears as an index of the whole drawing.
Figure \ref{exampleSCL} gives an example.
It is immediate from its definition that $\SCL$ restricts to the (cubic)
Swiss-cheese operad on the pairs of colors
$\{o, \up\}$, $\{o, \down\}$ and $\{o, \updown\}$. It also clearly restricts
to the little cubes operad $\cubes{2}$ on
$\up$, $\down$ and $\updown$.
 Originally, the $2$--dimensional Swiss-cheese operad is a $2$--colored
 operad on the set of colors
$\{o,c\}$. In some sense, the color $o$ encodes a homotopy associative
algebra and the color $c$ describes
part of its center. In the case of $2$--string links, the center of
$\pi_0\LL^0$ decomposes as
$\pi_0\KK^{\times 3}$. To encode this extra information, we split the
color $c$ into three independent colors
$\{\up,\down,\updown\}$.

\begin{figure} 
\small
\def\svgwidth{\hsize}
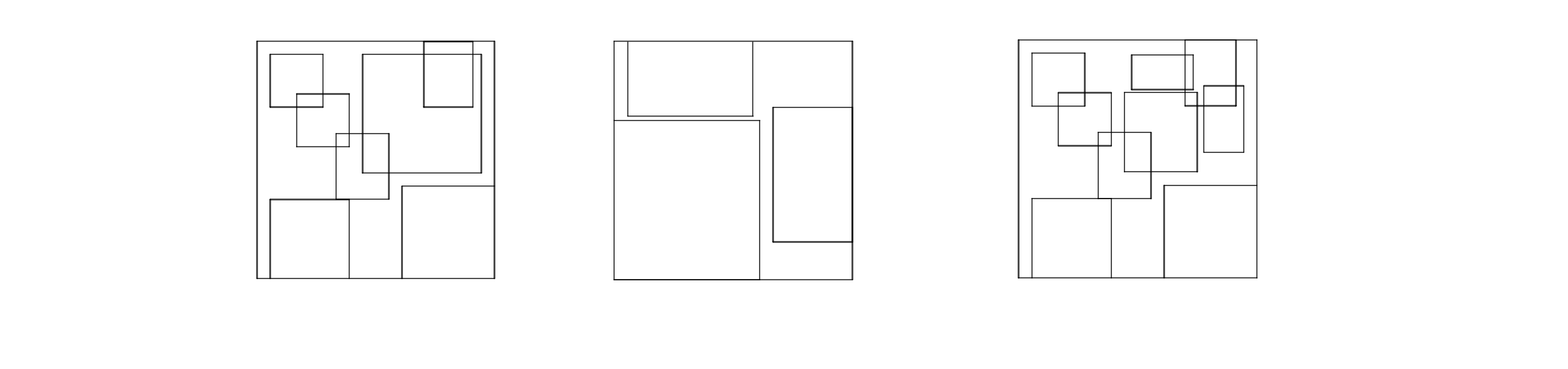
\caption{Illustration of the operadic composition in $\SCL$, more precisely~of 
$
\circ_{3} \co \SCL(\updown,o, \up, \updown, \down, o, \up; o)
\times \SCL(\up,\up,\up;\up) \rightarrow 
\SCL(\updown,o,\up,\up,\up,\updown,\down,o,\up;o)$.} 
\label{exampleSCL}
\end{figure}

We now investigate $\pi_0 \SCL$ and its algebras. Observe that for every $k$ and $n>1$, 
the
projection map
$\cubeso{n}(k) \rightarrow \cubes{n-1}(k)$ is a homotopy equivalence. 
A homotopy inverse is obtained by inflating $(n{-}1)$--cubes into $n$--cubes
of some fixed height.
This reasoning can readily be adapted to show that $\SCL(\ult; o)$ is
homotopy equivalent to the product
\[
\cubes{1}(|\ult|_o) \times \cubes{2}(|\ult|_\up) \times
\cubes{2}(|\ult|_\down) \times \cubes{2}(|\ult|_\updown).
\]
In particular, $\pi_0\SCL(\ult;s)$ is either
\begin{itemize}
\item empty if $s \neq o$ and $\ult \neq s^n$,
\item a single point if $s \neq o$ and $\ult = s^n$,
\item the discrete space $\Sigma_{|\ult|_o}$ when $s=o$.
\end{itemize}
\unskip
Therefore, an action of $\pi_0\SCL$ on	$X=(X_o, X_\up, X_\down, X_\updown)$
is the data of a
monoid structure on each space, such that the $X_s$ are commutative for
$s \neq o$ and act on $X_o$ with
compatible actions. With this description, it is easy to see with the
universal property of free objects that the free
such quadruplet on the basis $(A, B, C, D)$ is given by
\[
\displaylines{
\pi_0\SCL[A, B, C, D]
\hfill \cr \hfill
=\bigl(\As[A] \times \Com[B] \times \Com[C] \times \Com[D]
, \Com[B], \Com[C], \Com[D]\bigr),
}
\]
where the last three monoids act on the first one on their respective
factor. 

\hypertarget{secth}{}
\section{Operadic actions}
\label{OperadicActions}

We gather here the objects introduced in the previous sections and endow
the spaces $\KKhat$ and $\LLhat$ with
operadic actions. In a first paragraph, the fat long knots $\KKhat$ are
equipped with a $\cubes{2}$--algebra structure
originally exhibited by Budney in \cite{Budney}. In the case of fat
$2$--string links, there is a $\cubes{1}$--algebra structure
that follows from Burke and Koytcheff's work in \cite{KoytcheffInfect}. We
recall its construction in a second paragraph and
extend it to an action of $\SCL$ in a third one. 

\subsection{Budney's action on fat long knots}

We start with the little $2$--cubes action on $\KKhat$ originated in
\cite{Budney}.
Following Budney's work, we first define an action of the affine increasing
automorphisms of
$\mathbb{R}$ on the self-embeddings of $\mathbb{R} \times D^2$, then
proceed to extend it to the $2$--dimensional little cubes operad $\cubes{2}$.
We denote by $\CAut_1$ the group of real affine increasing functions.
A little $1$--cube is identified with its natural extension to the real
line so that $\cubes{1}(1)$ lives in $\CAut_1$.
We topologize $\CAut_1$ as we topologized $\cubes{1}(1)$, which coincides
with the $C^{\infty}$--topology and
turns it into a topological group.

\begin{Prop}[Budney~\cite{Budney}]
\label{actionCAuton}
The topological group $\CAut_1$ acts on the space of embeddings
$\Emb(\mathbb{R} \times D^2, \mathbb{R} \times D^2)$ via
\[
\begin{gathered}
\CAut_1 \times \Emb(\mathbb{R} \times D^2, \mathbb{R} \times D^2)
\rightarrow \Emb(\mathbb{R} \times D^2, \mathbb{R} \times D^2); \\
(L, f) \mapsto (L \times \id_{D^2}) \circ f \circ (L^{-1} \times
\id_{D^2}).
\end{gathered}
\]
Moreover, this restricts to an action of $\cubes{1}(1)$ on $EC(1, D^2)$
that we note $(L, f) \mapsto Lf$.
\end{Prop}

\begin{proof}
That this map defines a valid action of 
a topological group is immediate.
To prove the statement about the restriction, we just need to check that
$Lf$ restricts to the identity outside of
$J \times D^2$, provided $L \in \cubes{1}(1)$ and $f \in EC(1, D^2)$. For
any $t \notin J$,
$L^{-1}(t)$ does not lie in $J$ because $L(J) \subset J$. So, for every
$x \in D^2$, $(L^{-1} \times \id_{D^2})(t, x)$
is outside of $J \times D^2$, where $f$ restricts to the identity.
Thus $Lf(t, x)=(t, x)$, which proves the second part of the proposition. 
\end{proof}

\begin{figure}
\small
\def\svgwidth{\hsize}
\begingroup%
  \makeatletter%
  \providecommand\color[2][]{%
    \errmessage{(Inkscape) Color is used for the text in Inkscape, but the package 'color.sty' is not loaded}%
    \renewcommand\color[2][]{}%
  }%
  \providecommand\transparent[1]{%
    \errmessage{(Inkscape) Transparency is used (non-zero) for the text in Inkscape, but the package 'transparent.sty' is not loaded}%
    \renewcommand\transparent[1]{}%
  }%
  \providecommand\rotatebox[2]{#2}%
  \newcommand*\fsize{\dimexpr\f@size pt\relax}%
  \newcommand*\lineheight[1]{\fontsize{\fsize}{#1\fsize}\selectfont}%
  \ifx\svgwidth\undefined%
    \setlength{\unitlength}{1984.2519685bp}%
    \ifx\svgscale\undefined%
      \relax%
    \else%
      \setlength{\unitlength}{\unitlength * \real{\svgscale}}%
    \fi%
  \else%
    \setlength{\unitlength}{\svgwidth}%
  \fi%
  \global\let\svgwidth\undefined%
  \global\let\svgscale\undefined%
  \makeatother%
  \begin{picture}(1,0.17285714)%
    \lineheight{1}%
    \setlength\tabcolsep{0pt}%
    \put(0,0){\includegraphics[width=\unitlength,page=1]{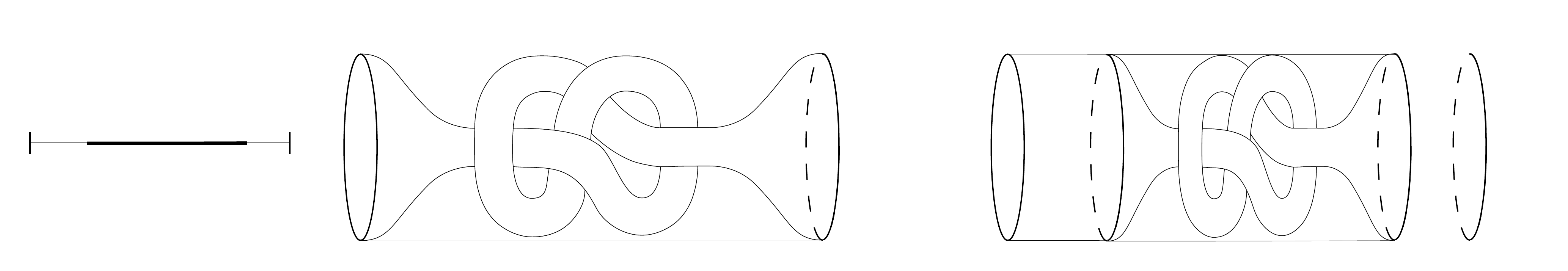}}%
    \put(0.56486529,0.07959863){\color[rgb]{0,0,0}\makebox(0,0)[lt]{\lineheight{1.25}\smash{\begin{tabular}[t]{l}$\longmapsto$\end{tabular}}}}%
    \put(0.10561352,0.09185617){\color[rgb]{0,0,0}\makebox(0,0)[lt]{\lineheight{1.25}\smash{\begin{tabular}[t]{l}$1$\end{tabular}}}}%
    \put(0,0){\includegraphics[width=\unitlength,page=2]{CAut1Knot.pdf}}%
  \end{picture}%
\endgroup%

\caption{Illustration of the action of $\cubes{1}(1)$ on $EC(1, D^2)$.}
\label{CAutonKnot}
\end{figure}

\begin{Defn}[Budney~\cite{Budney}]
We define two operations and a partial order on little $2$--cubes.
\begin{enumerate}[leftmargin=*, label=(\roman*), font=\itshape]
\item Given a little $2$--cube $L = l^1 \times l^2$, we write $L_{\pi}$
for the little $1$--cube $l^1$.
When $L \in \cubes{2}(k)$, $L_{\pi}$ denotes $\bigoplus_i L^i_{\pi}$.
These little $1$--cubes may overlap so $L_{\pi}$ lies in $\overcubes{1}(k)$
but not necessarily in $\cubes{1}(k)$.
\item For every $L=l^1 \times l^2 \in \cubes{2}(1)$, let $L_t$ denote the
number $l^{2}(-1)$.
Again, if $L \in \cubes{2}(k)$, 
then $L_t$ is the $k$--tuple of reals $(L^1_t,
\dots, L^k_t) \in J^k$.
\item We define a partial order on the little cubes $L^i$ of an element
$L \in \cubes{2}(k)$.
This binary relation is the order generated by setting $L^i < L^j$ if
and only if $L^i_t < L^j_t$ and the interiors of $L^i_\pi$ and $L^j_\pi$
intersect.
Then, a permutation $\sigma \in \Sigma_k$ is said to \emph{order} $L$
if the mapping $i \mapsto L^{\sigma(i)}$
is nondecreasing.
\end{enumerate}
\end{Defn}

\begin{figure}[b]
\small
\def\svgwidth{\hsize}
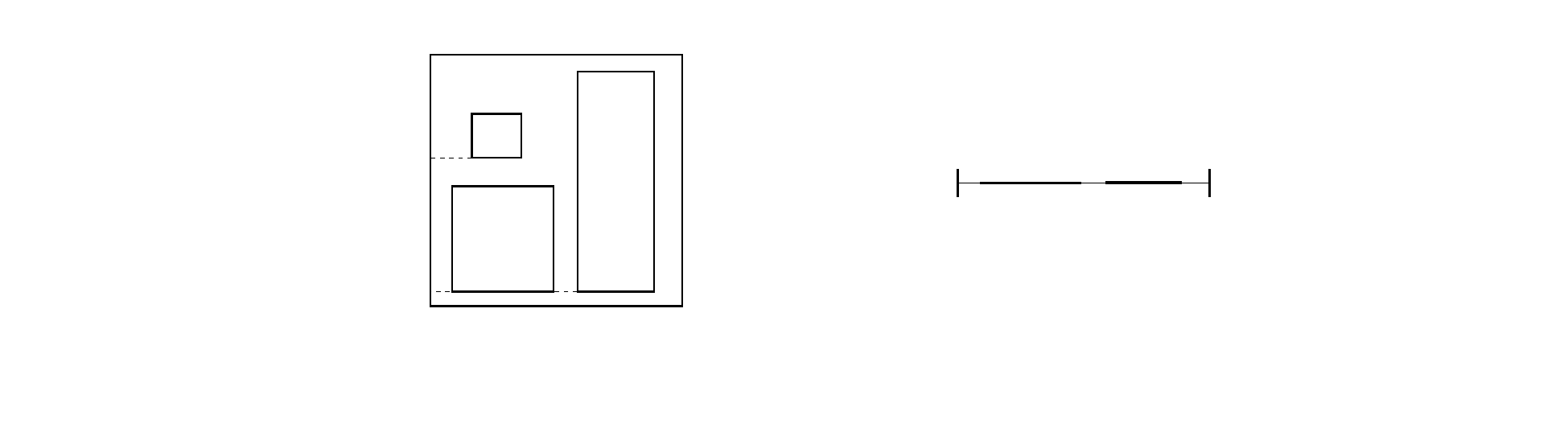
\caption{Illustration of the operations $(\underscore)_\pi$ and $(\underscore)_t$. The permutations $\id$, $(12)$ and $(23)$ order $L$.}
\end{figure}

We are now ready to define an action of $\cubes{2}$ on $EC(1,
D^2)$, ie a map  of operads
$\kappa \co \cubes{2} \rightarrow \EE_{EC(1, D^2)}$.
For every element $L \in \cubes{2}(k)$ and permutation $\sigma \in \Sigma_k$
that orders~$L$, we set
\[
\begin{gathered}
\kappa_k(L) \co EC(1, D^2)^{\times k} \rightarrow EC(1, D^2); \\
\mathbold{f}=(f_i)_i \mapsto (L^{\sigma(1)}_\pi f_{\sigma(1)})\circ
\cdots \circ (L^{\sigma(k)}_\pi f_{\sigma(k)}).
\end{gathered}
\]
One can be assured that $\kappa_k(L)$ does not depend on $\sigma$ as follows.
Two choices for $\sigma$ differ by a sequence of transpositions $(ab)$
such that $L^a$ and $L^b$
are incomparable, ie such that $L^a_\pi$ and $L^b_\pi$ are almost disjoint.
\begin{figure}
\small
\def\svgwidth{\hsize}
\begingroup%
  \makeatletter%
  \providecommand\color[2][]{%
    \errmessage{(Inkscape) Color is used for the text in Inkscape, but the package 'color.sty' is not loaded}%
    \renewcommand\color[2][]{}%
  }%
  \providecommand\transparent[1]{%
    \errmessage{(Inkscape) Transparency is used (non-zero) for the text in Inkscape, but the package 'transparent.sty' is not loaded}%
    \renewcommand\transparent[1]{}%
  }%
  \providecommand\rotatebox[2]{#2}%
  \newcommand*\fsize{\dimexpr\f@size pt\relax}%
  \newcommand*\lineheight[1]{\fontsize{\fsize}{#1\fsize}\selectfont}%
  \ifx\svgwidth\undefined%
    \setlength{\unitlength}{1332.28346457bp}%
    \ifx\svgscale\undefined%
      \relax%
    \else%
      \setlength{\unitlength}{\unitlength * \real{\svgscale}}%
    \fi%
  \else%
    \setlength{\unitlength}{\svgwidth}%
  \fi%
  \global\let\svgwidth\undefined%
  \global\let\svgscale\undefined%
  \makeatother%
  \begin{picture}(1,0.40425532)%
    \lineheight{1}%
    \setlength\tabcolsep{0pt}%
    \put(0,0){\includegraphics[width=\unitlength,page=1]{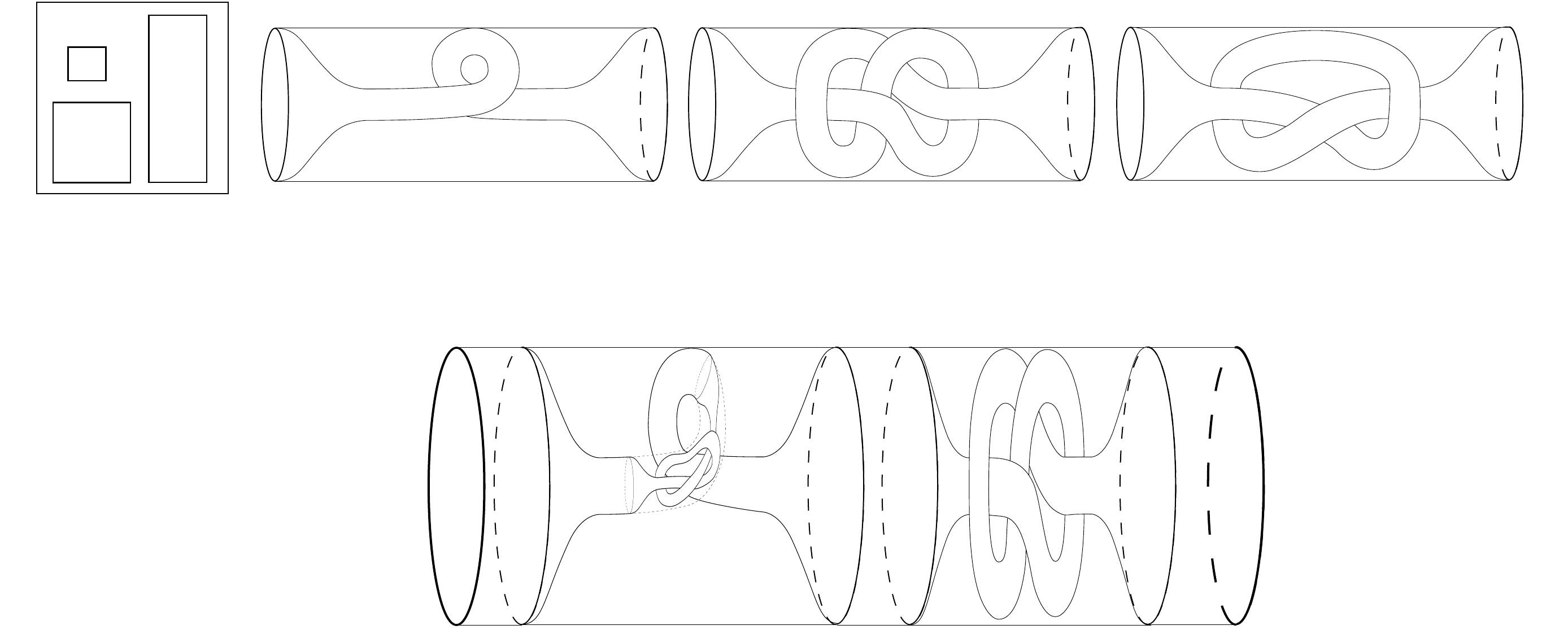}}%
    \put(0.58009144,0.23157485){\color[rgb]{0,0,0}\makebox(0,0)[lt]{\lineheight{1.25}\smash{\begin{tabular}[t]{l}$\kappa_3$\end{tabular}}}}%
    \put(0.0526352,0.30849291){\color[rgb]{0,0,0}\makebox(0,0)[lt]{\lineheight{1.25}\smash{\begin{tabular}[t]{l}$1$\end{tabular}}}}%
    \put(0.1080851,0.33213654){\color[rgb]{0,0,0}\makebox(0,0)[lt]{\lineheight{1.25}\smash{\begin{tabular}[t]{l}$2$\end{tabular}}}}%
    \put(0.04935705,0.35432863){\color[rgb]{0,0,0}\makebox(0,0)[lt]{\lineheight{1.25}\smash{\begin{tabular}[t]{l}$3$\end{tabular}}}}%
    \put(0.56138377,0.2314643){\color[rgb]{0,0,0}\makebox(0,0)[lt]{\lineheight{1.25}\smash{\begin{tabular}[t]{l}$\longdownmapsto$\end{tabular}}}}%
  \end{picture}%
\endgroup%

\caption{Illustration of Budney's action on fat long knots.}
\end{figure}
Then, $\supp(L^a_\pi f_a)$ and $\supp(L^b_\pi f_b)$ are almost disjoint as well
so both orders of composition yield the same outcome.
For the continuity, consider for every $\tau \in \Sigma_k$ the map
\[
\begin{gathered}
\kappa_k^{\tau} \co \cubes{2}(j) \times  EC(1, D^2)^{\times k} \rightarrow
EC(1, D^2); \\
(L, \mathbold{f}) \mapsto (L^{\tau(1)}_\pi f_{\tau(1)}) \circ \cdots
\circ (L^{\tau(k)}_\pi f_{\tau(k)}).
\end{gathered}
\]
Each $\kappa_k^\tau$ is continuous and coincides with $\kappa_k$ on
$$F_\tau=\{L \in \cubes{2}(k) \mid \tau \text{ orders } L\} \times EC(1,
D^2)^{\times k}.$$
The sets $F_{\tau}$ are closed and cover $\cubes{2}(k) \times EC(1,
D^2)^{\times k}$ so $\kappa_k$ is continuous.
In arity $0$, we set $\kappa_0$ to be the map sending the single point in
$\cubes{2}(0) \times EC(1, D^2)^{\times 0}$ to $\id_{\mathbb{R} \times D^2}$.

\begin{Thm}[Budney~\cite{Budney}]
\label{CtwActsOnEC-on-Dtw-}
The operations $\kappa$ turn $EC(1, D^2)$ into a $\cubes{2}$--algebra.
\end{Thm}

Even though a proof of this result is readily available in
\cite{Budney}, we provide one here as the
methods and ideas at stake will be reused before the end of this section.

\begin{proof}
The operation $\kappa_1$ clearly maps the basepoint $\id_{J^2} \in
\cubes{2}(1)$ to the identity.
We need to check the compatibility of $\kappa$ with  the symmetric group
actions and the operadic compositions.
We start with the symmetric structure. Recall that a permutation $\tau$
acting on the right of $L \in \cubes{2}(k)$
yields $\bigoplus_i L^{\tau(i)}$. It also acts on the left of
$$\mathbold{f}=(f_i)_i \in EC(1, D^2)^{\times k}$$
to give
$\tau \mathbold{f}=(f_{\tau^{-1}(i)})_i$. Thus, if $\sigma$ is a permutation
that orders $L\tau$, then $\tau \circ \sigma$ orders~$L$. 
This proves the needed equality
\[
\kappa_k(L\tau, \mathbold{f})
=(L^{\tau \circ \sigma(1)}_\pi f_{\sigma(1)}) \circ \cdots \circ (L^{\tau
\circ \sigma(k)}_\pi f_{\sigma(k)})
=\kappa_k(L, \tau\mathbold{f}).
\]
We are left to prove that $\kappa$ preserves operadic compositions.
Given little $2$--cubes $L \in \cubes{2}(k)$, $P \in \cubes{2}(l)$ and an
integer $i \leq k$,
we need to show that 
$$\kappa_{k+l-1}(L \circ_i P)=\kappa_k(L) \circ_i \kappa_l(P) . $$
Let $\sigma$ and $\tau$ be permutations that respectively order $L$ and $P$.
Unravelling the definition of $\kappa$ shows that the desired equality
boils down to checking that
$\sigma \circ_i \tau$ orders $L \circ_i P$.
Recall the definitions of $\sigma \circ_i \tau$ and $L \circ_i P$:
\begin{itemize}
\item $\sigma \circ_i \tau$ shuffles the interval $\{i, \dots,	i+l-1\}$
according to $\tau$, then permutes
$\{1, \dots, k+l-1\}$ according to $\sigma$ while treating the shuffled
interval as a single block.
\item $L \circ_i P$ is obtained from $L$ by replacing $L^i$ with $\bigoplus_r
L^i \circ P^r$.
\end{itemize}
\unskip
If $L^i$ and $L^j$ are incomparable, $L^i \circ P^r$ and $L^j$ also are,
so the result follows. 
\end{proof}

This recently developed structure on the space of framed long
knots generalizes the stacking operation in the
following sense: acting with two side-by-side rectangles of width $1$
on two knots results in their concatenation.
In particular, the $\Com$--algebra structure on $\pi_0 EC(1, D^2)$
induced from
$\cubes{2} \rightarrow \EE_{EC(1, D^2)}$ is the monoid structure described
in Subsection \ref{EmbeddingSpacesKnot}.

\begin{Thm}[Budney~\cite{Budney}]
\label{KhatSubCtwAlgebra}
The fat long knots $\KKhat$ form a sub--$\cubes{2}$--algebra of $EC(1,D^2)$.
\end{Thm}

\begin{proof}
As mentioned in Subsection \ref{EmbeddingSpacesKnot}, the framing number
$\omega$ descends to a morphism of monoids
$\pi_0 EC(1, D^2) \rightarrow \mathbb{Z}$. Recall also that $\pi_0 \cubes{2}$
is the commutative operad $\Com$.
The integers $\mathbb{Z}$ form an abelian group and thus a commutative
monoid.
They can therefore be seen as a $\cubes{2}$--algebra via the structure map
$\cubes{2} \rightarrow \Com \rightarrow \EE_{\mathbb{Z}}$.
In this framework, the framing number $\omega$ is a $\cubes{2}$--algebra
morphism, hence the result. 
\end{proof}

We conclude this paragraph with a quick discussion about $\kappa$. Let $L$
be an element of $\cubes{2}(k)$.
The heights of the little cubes of $L$ only appear in the formula of
$\kappa_k(L)$
to dictate a composition order. This is done via an ordering permutation,
which we defined
as an element $\sigma \in \Sigma_k$ such that the mapping $i \mapsto
L^{\sigma(i)}$
is nondecreasing. Here, one can replace the word ``decreasing'' with
``increasing'' and define another action
with the same formula. We refer to it as \emph{Budney's reverse action}.
There is no substantial difference between these two versions of $\kappa$,
nor is there a reason to prefer one or the other.
We still introduce the two of them now, as they will both play a role in
the next paragraphs.
Informally, the need for a reversed action arises because knots yielding
split links must be tied at the beginning
of a composition, while knots yielding cables must be tied at the end. 

\subsection{Burke and Koytcheff's actions on fat $2$--string links}

This paragraph is a first step towards an adaptation of Budney's work to
$2$--string links.
Namely, we build an action of $\cubes{1}$ on $EC_\iota(1, D^2)$ and $\LLhat$.
This structure has already been exhibited by Burke and Koytcheff \cite[Theorem~6.8]{KoytcheffInfect},
with $\cubes{1}$ appearing as a suboperad of a way bigger object called
the infection operad.
As before, we start with an action of $\CAut_1$ on the embeddings
$(\mathbb{R} \times D^{2})\amalg (\mathbb{R} \times D^{2}) \cof
\mathbb{R}\times D^2$,
then proceed to extend it to $\cubes{1}$. 

\begin{Prop}
The topological group $\CAut_1$ acts on
$\Emb((\mathbb{R} \times D^{2})^{\amalg 2}, \mathbb{R} \times D^{2})$ via
	\[
\begin{gathered}
\CAut_1 \times \Emb((\mathbb{R} \times D^{2})^{\amalg 2} , \mathbb{R}
\times D^{2})
\rightarrow \Emb((\mathbb{R} \times D^{2})^{\amalg 2}, \mathbb{R}
\times D^{2}); \\
(L, l) \mapsto (L \times \id_{D^2})\circ l \circ (L^{-1} \times
\id_{D^2})^{\amalg 2}.
\end{gathered}
\]
Moreover, this restricts to an action of $\cubes{1}(1)$ on $EC_\iota(1,
D^2)$ that we note $(L, l) \mapsto Ll$.
\end{Prop}

\begin{proof}
The proof of this is very similar to the proof of Proposition \ref{actionCAuton}:
the fact that the formula above specifies a valid action of a topological
group is still clear and the restriction
statement is proved just as in the case of framed long knots. 
\end{proof}

To distinguish Budney's action from the one we build now, we
denote the structure map by $\lambda$.
The space $\cubes{1}(0)$ in arity $0$ still consists of a single point
that $\lambda_0$ maps to $\id_{\mathbb{R}}\times \iota$
from Definition \ref{DefFramedStringLink}.
For any positive integer $k$ and $L \in \cubes{1}(k)$, we set $\lambda_k(L)$
to the map that concatenates $k$ framed string links according to the
configuration of intervals $L$.
That is, for every $\mathbold{f}=(f_i)_i \in EC_\iota(1, D^2)^{\times k}$,
\[
\begin{gathered}
\lambda_k(L)(\mathbold{f}) \co (\mathbb{R} \times D^2)\amalg(\mathbb{R}
\times D^2) \rightarrow \mathbb{R} \times D^2;\\
(t, x) \mapsto
\begin{cases}
L^i f_i(t, x) & \text{when }t \in L^i(J), \\
(t, \iota(x)) & \text{elsewhere}.
\end{cases}
\end{gathered}
\]
The embeddings patch in a differentiable way because the little cubes are
almost disjoint.
The outcome lies in $EC_\iota(1, D^2)$ because
$\supp_\iota(L^if_i)=(L^i \times \id_{D^2})^{\amalg 2}(\supp_\iota(f_i))$.

\begin{figure}
\small
\def\svgwidth{\hsize}
\begingroup%
  \makeatletter%
  \providecommand\color[2][]{%
    \errmessage{(Inkscape) Color is used for the text in Inkscape, but the package 'color.sty' is not loaded}%
    \renewcommand\color[2][]{}%
  }%
  \providecommand\transparent[1]{%
    \errmessage{(Inkscape) Transparency is used (non-zero) for the text in Inkscape, but the package 'transparent.sty' is not loaded}%
    \renewcommand\transparent[1]{}%
  }%
  \providecommand\rotatebox[2]{#2}%
  \newcommand*\fsize{\dimexpr\f@size pt\relax}%
  \newcommand*\lineheight[1]{\fontsize{\fsize}{#1\fsize}\selectfont}%
  \ifx\svgwidth\undefined%
    \setlength{\unitlength}{1006.2992126bp}%
    \ifx\svgscale\undefined%
      \relax%
    \else%
      \setlength{\unitlength}{\unitlength * \real{\svgscale}}%
    \fi%
  \else%
    \setlength{\unitlength}{\svgwidth}%
  \fi%
  \global\let\svgwidth\undefined%
  \global\let\svgscale\undefined%
  \makeatother%
  \begin{picture}(1,0.45070423)%
    \lineheight{1}%
    \setlength\tabcolsep{0pt}%
    \put(0,0){\includegraphics[width=\unitlength,page=1]{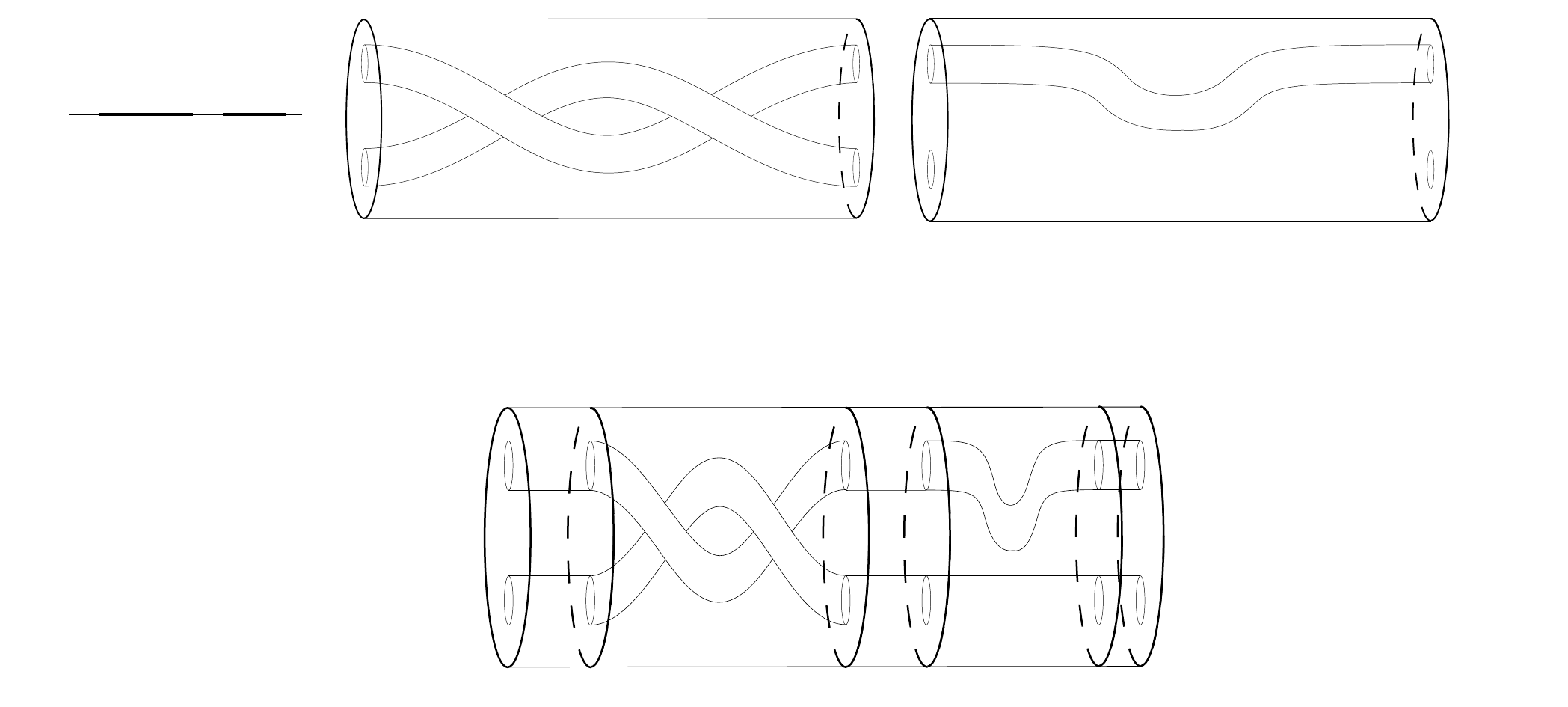}}%
    \put(0.57167982,0.24307625){\color[rgb]{0,0,0}\makebox(0,0)[lt]{\lineheight{1.25}\smash{\begin{tabular}[t]{l}$\lambda_2$\end{tabular}}}}%
    \put(0.09109751,0.38961121){\color[rgb]{0,0,0}\makebox(0,0)[lt]{\lineheight{1.25}\smash{\begin{tabular}[t]{l}$1$\end{tabular}}}}%
    \put(0.15735023,0.38998766){\color[rgb]{0,0,0}\makebox(0,0)[lt]{\lineheight{1.25}\smash{\begin{tabular}[t]{l}$2$\end{tabular}}}}%
    \put(0,0){\includegraphics[width=\unitlength,page=2]{exampleKoytcheff.pdf}}%
    \put(0.55708064,0.24353616){\color[rgb]{0,0,0}\makebox(0,0)[lt]{\lineheight{1.25}\smash{\begin{tabular}[t]{l}$\longdownmapsto$\end{tabular}}}}%
  \end{picture}%
\endgroup%

\caption{Illustration of Burke and Koytcheff's action on fat $2$--string
links.}
\end{figure}

\begin{Thm}[Burke and Koytcheff~\cite{KoytcheffInfect}]
\label{ConActsOnECiota-on-Dtw-}
The operations $\lambda$ turn $EC_\iota(1, D^2)$ into a $\cubes{1}$--algebra.
\end{Thm}

\begin{proof}
It is clear that $\lambda_1(\id_{J})$ is the identity on $EC_\iota(1, D^2)$.
We check the compatibility with the symmetric actions.
Let $\tau$ be a permutation, $L \in \cubes{1}(k)$ and $\mathbold{f} \in
EC_\iota(1, D^2)^{\times k}$.
To prove the desired
$\lambda_k(L\tau,\mathbold{f})=\lambda_k(L,\tau\mathbold{f})$,
we show that these maps agree on the images of every little cube of $L$.
This is enough as they clearly restrict to $\id_{\mathbb{R}}\times \iota$
outside of these intervals.
For every $i \leq k$, the left-hand side of the equation restricts to
$(L\tau)^i f_i$ on $(L\tau)^i(J^{k})$.
The right-hand side restricts to $L^{\tau(i)}f_{\tau^{-1} \circ
\tau(i)}=(L\tau)^i f_i$ so we are done.
The associative compatibility is verified the same way. 
\end{proof}

As in the case of knots, the stacking operation arises as a
special case of this recently developed action.
More precisely, acting with two side-by-side intervals of width $1$
on two string links results in their
concatenation. Therefore, the $\As$--algebra structure on $\pi_0 EC_\iota(1,
D^2)$ induced from
$\lambda \co \cubes{1} \rightarrow \EE_{EC_\iota(1, D^2)}$ is the monoid
structure on
$\pi_0EC_\iota(1, D^2)$ discussed in Subsection \ref{EmbeddingSpacesLink}.
Moreover, as in the case of knots, we can restrict ourselves to unframed
embeddings:

\begin{Thm}
\label{LhatSubConAlgebra}
The fat $2$--string links $\LLhat$ form a sub--$\cubes{1}$--algebra of
$EC_\iota(1, D^2)$.
\end{Thm}

\begin{proof}
Just as in Theorem \ref{KhatSubCtwAlgebra}, $\mathbb{Z}^{\times 2}$ is a
group that we can think of as 
an $\As$--algebra
and therefore a $\cubes{1}$--algebra.
This turns the framing number $\omega$ into a morphism of
$\cubes{1}$--algebras, hence the result. 
\end{proof}

\subsection{The action of $\SCL$ on fat $2$--string links}

This section aims to merge the two actions defined above into a single
$\SCL$--algebra structure on the spaces of
fat long knots and fat $2$--string links. More precisely, we build an action
of $\SCL$ on the quadruplet of spaces
$X=(X_o, X_\up, X_\down, X_\updown)$, where $X_o = EC_\iota(1, D^2)$
and $X_s=EC(1, D^2)$ for every
$s \in \{\up, \down, \updown\}$. We start with a lemma to ease the
construction. 

\begin{Lem}
There is a map
	\[
EC_\iota(1, D^2) \times EC(1, D^2)^{\times 3} \rightarrow EC_\iota(1,
D^2); \quad
(l, k_\up, k_\down, k_\updown) \mapsto
k_\updown \circ l \circ [k_\up \amalg k_\down].
\]
\end{Lem}

\begin{proof}
The continuity of this application immediately follows from the continuity
of the composition in the
$C^{\infty}$--topology. The purpose of this lemma is actually to check that
$k_\updown \circ l \circ [k_\up \amalg k_\down]$ indeed lives in $EC_\iota(1,
D^2)$.
This follows from the inclusions $\supp(k_s) \subset J \times D^2$ 
for $s \in
\{\up, \down, \updown\}$, and
$l(\inter((J \times D^2)\amalg(J \times D^2))) \subset J \times D^2$. 
\end{proof}

This map is in some way a combination of the morphisms $\varphi^s$
from Subsection 
\ref{EmbeddingSpacesLink}. Indeed, if one restricts this application
to the subspace
$\{\id_{\mathbb{R}}\times \iota\} \times EC(1, D^2) \times
\{\id_{\mathbb{R}\times D^2}\}^{\times 2}$,
the formula becomes
$k_\up \mapsto (\id_{\mathbb{R}} \times \iota) \circ [k_\up \amalg
\id_{\mathbb{R}\times D^2}]$,
which is a fattened version of $\varphi^\up$.
We denote it by $\varphihat^\up$. The same goes for $\down$. In the case
of $\updown$,
one is left with $k_\updown \mapsto k_\updown \circ (\id_{\mathbb{R}}
\times \iota)$. This map sends a long knot
$k_\updown$ to the string link whose strands are parallel and knotted
according $k_\updown$. In other words,
it is again a fattened version of $\varphi^\updown$ that we
denote by 
$\varphihat^\updown$. 

We are now ready to define the morphism $\mu \co \SCL \rightarrow \EE_X$
for the new action.
As the values of $\mu_{(\ult; s)}$ heavily depend on $(\ult; s)$,
defining $\mu$ takes several steps. 

We start by specifying the values of $\mu$ in monochromatic cases.
The operad $\SCL$ restricts to the little cubes operad $\cubes{2}$ on
the colors
$s \in \{\up, \down\}$. We set $\mu$ to Budney's action on these colors.
In other words, $\mu_{(s^k; s)}=\kappa_k$ from Theorem \ref{CtwActsOnEC-on-Dtw-}.
When $s=\updown$, we similarly set $\mu_{(\updown^k; \updown)}$ to Budney's
reverse action,
which we will still denote by $\kappa$ by slight abuse of notation.
On the color $o$, $\SCL$ restricts to the operad $\cubeso{2}$.
There is a morphism  $(\underscore)_\pi \co \cubeso{2} \rightarrow \cubes{1}$
and we set $\mu$ to the composite
$\lambda \circ (\underscore)_\pi$ on this suboperad.

For every $s \in \{\up, \down, \updown\}$, the only input
colors $\ult$
that do not lead to an empty $\SCL(\ult; s)$ are the monochromatic ones
such that $\ult=s^k$.
Thus, we are left to specify $\mu_{(\ult; s)}$ when $s=o$ and the inputs
are mixed.
To this end, we introduce color sorting functions. Let $\ult \in
S^{\star}$. Consider four injective (not necessarily increasing) maps
\[
\alpha_s \co [|\ult|_s]=\{1, \dots, |\ult|_s\} \rightarrow [|\ult|]=\{1,
\dots, |\ult|\},\ s\in S,
\]
whose disjoint images cover $[|\ult|]$ and such that $t_{\alpha_{s}(i)}=s$
for every $i \in [|\ult|_s]$.
These maps regroup inputs of the same color and are said to sort the colors
of $\ult$.
Observe that each $\alpha_s$ lifts to a map $\SCL(\ult; o) \rightarrow
\cubes{2}(|\ult|_s)$
that discards the little cubes whose colors are different from $s$,
\[
\alpha_s \co \SCL(\ult; s) \rightarrow \cubes{2}(|\ult|_s); 
\quad
L \mapsto \bigoplus_i L^{\alpha_s(i)}.
\]
Discarding embeddings also yields a map
\[
\alpha_s \co X^{\times \ult} \rightarrow X_s^{\times |\ult|_s};
\quad
\mathbold{f} \mapsto
\alpha_s \mathbold{f}=(f_{\alpha_s(1)}, \dots, f_{\alpha_s(|\ult|_s)}).
\]
The behavior of these lifts with respect to the symmetric structures
on $\SCL$ and $X^{\times \ult}$ is captured by the following relations:
for every $\sigma \in \Sigma_{|\ult|}$, $\tau \in \Sigma_{|\ult|_s}$,
$L \in \SCL(\ult; o)$ and $\mathbold{f} \in X^{\times \ult}$,
\begin{align*}
(\sigma \circ \alpha_s) L
	& = \bigoplus_i L^{\sigma \circ \alpha_s(i)}
= \bigoplus_i (L\sigma)^{\alpha_s(i)}
=\alpha_s(L \sigma), 
	\quad &
(\sigma \circ \alpha_s) \mathbold{f}
	& = \alpha_s (\sigma^{-1} \mathbold{f}), 
	\cr
(\alpha_s \circ \tau) L
	& = \bigoplus_i L^{\alpha_s \circ \tau(i)}
= \bigoplus_i (\alpha_s L)^{\tau(i)}
=(\alpha_sL)\tau,
	\quad &
(\alpha_s \circ \tau) \mathbold{f}
	& =\tau^{-1} (\alpha_s \mathbold{f}).
\end{align*}
We can finally combine the previous actions and define $\mu_{(\ult;
o)}(L)$ as the map
\[
\begin{gathered}
\mu_{(\ult; o)}(L)\co X^{\times \ult}	\rightarrow EC_\iota(1, D^2); \\
\mathbold{f}
\mapsto
\kappa ( \alpha_\updown L, \alpha_\updown \mathbold{f} )
\circ		\lambda ( \alpha_o L_\pi, \alpha_o \mathbold{f})
\circ	\bigl[\kappa ( \alpha_\up L,\alpha_\up \mathbold{f} )
\amalg	\kappa ( \alpha_\down L, \alpha_\down \mathbold{f} )
\bigr], 
\end{gathered}
\]
where the $\kappa$ on the left-hand side refers to Budney's reverse action
and the other two to Budney's regular action. 

The continuity of $\mu$ is immediate and its values do not depend
on the color sorting functions:
if one chooses to replace $\alpha_o$ with ${\alpha_o}^\prime$,
there is a permutation $\tau$ such that ${\alpha_o}^\prime=\alpha_o
\circ \tau$.
Then, the relations above give
\begin{align*}
	  \lambda ( {\alpha_o}^\prime L_\pi,  {\alpha_o}^\prime
	  \mathbold{f} )
& = \lambda ( (\alpha_o \circ \tau) L_\pi, (\alpha_o \circ \tau)
\mathbold{f} )\\
& = \lambda ( (\alpha_o L_\pi)\tau, \tau^{-1}(\alpha_o
\mathbold{f}) )\\
& = \lambda ( \alpha_o L_\pi, \alpha_o \mathbold{f}).
\end{align*}
The same argument with $\kappa$ shows that the remaining $\alpha_s$ can
be replaced as well.

\begin{figure} 
\small
\def\svgwidth{\hsize}
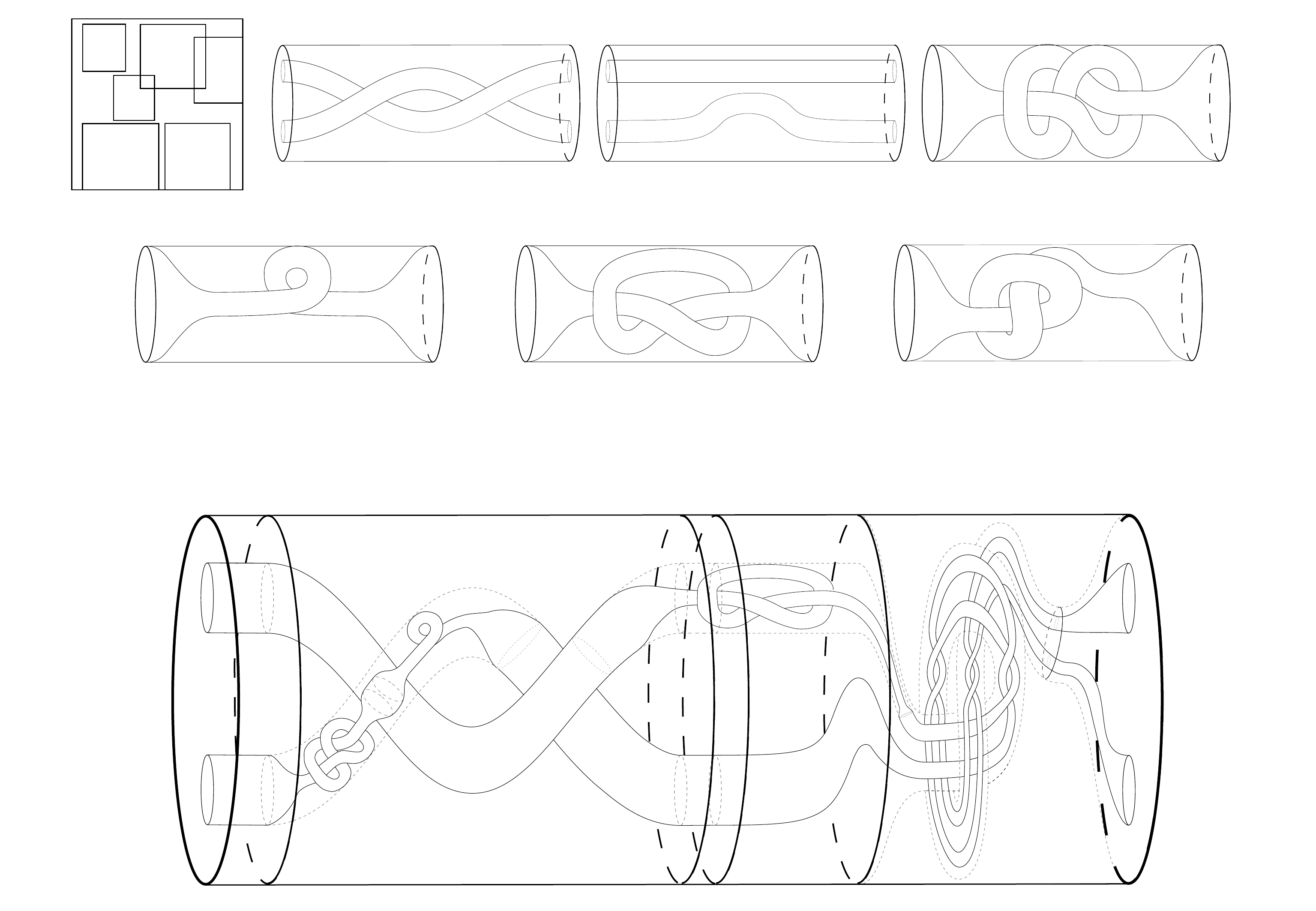
\caption{Illustration of the action of $\SCL$ on fat $2$--string links and
fat long knots.}
\end{figure}

\begin{Thm}
\label{SCLActs}
The operations $\mu$ turn the quadruplet $X$ into 
an $\SCL$--algebra.
\end{Thm}

\begin{proof}
First of all, it is clear that every $\mu_{(s;s)}$ sends $\id_{J^2}$
to the identity.
The symmetric compatibility is also quickly verified:
when functions $(\alpha_s)_s$ sort the colors of $\ult$, the composites
$(\sigma^{-1} \circ \alpha_s)_s$ sort the colors of $\ult\sigma$ and the
needed equality follows. 

We are left to check the compatibility with the operadic composition.
Let $L \in \SCL(\ult; s)$ and $P \in \SCL(\ulu; t_i)$ for some $i$.
We need to show that $\mu(L \circ_i P)=\mu(L) \circ_i \mu(P)$.
The validity of Budney, Burke and Koytcheff's actions
(Theorems~\ref{CtwActsOnEC-on-Dtw-} and~\ref{ConActsOnECiota-on-Dtw-})
implies the result when $\ult$ and $\ulu$ are monochromatic of the
same color.
Since there is no operation with output color $s \in \{\up, \down, \updown\}$
and input colors $\ult \neq s^k$, we may assume that $s=o$ and $\ult
\neq o^k$.
When evaluated in embeddings $\mathbold{f} \in X^{\times \ult \circ_i
\ulu}$, the desired equality reads
\[
	\displaylines{
\mu_{(\ult \circ_i \ulu; o)} (L \circ_i P, \mathbold{f} )
\hfill \cr \hfill
=
\mu_{(\ult; o)} \bigl(L, f_1, \dots, f_{i-1},
\mu_{(\ulu; t_i)} (P, f_{i}, \dots, f_{i+|\ulu|-1} ),
f_{i+|\ulu|}, \dots, f_{|\ult \circ_i \ulu|}\bigr).
	}
\]
We split cases and unravel the definition of $\mu$ on both sides of this
equation. 

Assume first that $t_i=\up$. This forces $\ulu=\up^{|\ulu|}$.
Let $(\alpha_s)_s$ sort the colors of $\ult$. We may ask for
$\alpha_\up(|\ult|_\up)=i$.
We sort the colors $s \neq \up$ in $\ult \circ_i \ulu$
with functions $(\gamma_s)_s$ mapping $j$ to $\alpha_s(j)$ if $\alpha_s(j)<i$
or to $\alpha_s(j)+|\ulu|$ if $\alpha_s(j)>i$.
The reason for this choice is $\gamma_s (L \circ_i P)=\alpha_s L$.
For the remaining $\gamma_\up$, we use the same construction on
$[|\ult|_\up-1]$ and
extend it to $[|\ult \circ_i \ulu|_\up]$ via the increasing map onto
the interval $\{i+j,\ j < |\ulu|\}$.
The equality reduces to
\[
K_{\updown} \circ \Lambda \circ [ K^L_\up \amalg K_\down ]
=K_{\updown} \circ \Lambda \circ [ K^R_\up \amalg K_\down ],
\vspace{-5pt}
\]
where 
\begin{align*}
\Lambda & = \lambda (\alpha_o L_\pi, \gamma_o \mathbold{f}),
\\
K^L_\up & = \kappa (\gamma_\up(L \circ_i P), \gamma_\up \mathbold{f} ), 
\\
K_{\updown} & = \kappa (\alpha_\updown L, \gamma_\updown \mathbold{f}), 
\\
K^R_\up & =\kappa (\alpha_\up L, f_{\alpha_\up(1)}, \dots, f_{\alpha_\up(|\ult|_\up-1)}, \kappa(P, f_{i}, \dots, f_{|\ulu|+i-1}) ),
\\ 
K_\down	& =\kappa (\alpha_\down L, \gamma_\down \mathbold{f} ).
\end{align*}
Furthermore,
$\gamma_\up(L \circ_i P)=\alpha_\up L \circ_{|\ult|_\up}
P$, so the validity of Budney's action
(Theorem \ref{CtwActsOnEC-on-Dtw-}) completes the proof in this case.
The same manipulations treat the cases $t_i=\down$ and $\updown$. 

We are left to treat the case $s = o$ and $t_i = o$.
Let $(\alpha_s)_s$ and $(\beta_s)_s$ be color sorting functions for $\ult$
and $\ulu$, respectively.
We may again ask for $\alpha_o(|\ult|_o)=i$. Let $(\gamma_s)_s$ be the
color sorting functions for
$\ult \circ_i \ulu$ one naturally builds from $(\alpha_s)_s$ and
$(\beta_s)_s$.
More precisely, $\gamma_s$ agrees with $\alpha_s$ on ${\alpha_s}^{-1}(\{l \mid l <i\})$,
with $\alpha_s+|\ulu|$ on ${\alpha_s}^{-1}(\{l \mid \nobreak l >i\})$ and
maps the remaining interval to the inputs $s$ in $\{i+l \mid l < |\ulu|\}$
according to~$\beta_s$.
These choices are the ones giving $\gamma_s(L \circ_i P)=\alpha_sL \oplus
(L^i \circ \beta_sP)$
for every $s$ in $\{\up, \down, \updown\}$ and $\gamma_o(L \circ_i
P)=\alpha_oL \circ_{|\ult|_o} \beta_o P$. 
The left-hand side of the equality reads
$K^L_\updown \circ \Lambda^{L} \circ [ K^L_\up \amalg K^L_\down ]$, where
\begin{align*}
\Lambda^L & =\lambda (\alpha_oL \circ_{|\ult|_o} \beta_o P, \gamma_o
\mathbold{f}) \qr{and} 
\cr
K^L_s & = \kappa (\alpha_sL \oplus (L^i \circ \beta_s P), \gamma_s
\mathbold{f} )
\qr{for every} s \in \{\up, \down, \updown\}.
\end{align*}
On the other hand, the right-hand side of the equality is
$K^R_\updown \circ \Lambda^R \circ [ K^R_\up \amalg K^R_\down ]$,
where
\begin{align*}
\Lambda^R & = \lambda (\alpha_oL, f_{\gamma_o(1)}, \dots,
f_{\gamma_o(|\ult|_o-1)},
\mu_{(\ulu;o)} (P, f_{i}, \dots, f_{i+|\ulu|-1} ))
\qr{and} 
\cr
K^R_s & = \kappa (\alpha_sL, f_{\gamma_s(1)}, \dots,
f_{\gamma_s(|\ult|_s)}) \qr{for every} s \in \{\up, \down, \updown\}.
\end{align*}
But $\mu_{(\ulu; o)} (P, f_{i}, \dots, f_{i+|\ulu|-1} )$
is itself of the form
${K_\updown}^{\prime} \circ \Lambda^{\prime} \circ [ {K_\up}^{\prime}
\amalg {K_\down}^{\prime} ]$, where
\begin{align*}
\Lambda^{\prime} & = \lambda (\beta_o P, f_{\gamma_o(|\ult|_o)}, \dots,
f_{\gamma_o(|\ult \circ_i \ulu|_o)}) \qr{and} 
\cr
{K_s}^{\prime} & = \kappa (\beta_o P, f_{\gamma_s(|\ult|_s+1)}, \dots,
f_{\gamma_s(|\ult \circ_i \ulu|_s)}) \qr{for every} s \in \{\up, \down, \updown\}.
\end{align*}
It is easy to check from the definition of $\lambda$ and Theorem \ref{ConActsOnECiota-on-Dtw-} that
\[
\Lambda^R=L^i {K_\updown}^\prime \circ \Lambda^L \circ [ L^i
{K_\up}^\prime \amalg L^i {K_\down}^\prime  ].
\]
We get the following new expression for the whole right-hand side of
the equation:
\[
	 (K^R_\updown \circ L^i {K_\updown}^\prime )
\circ \Lambda^L
\circ [	L^i {K_\up}^\prime \circ K^R_\up
\amalg		L^i {K_\down}^\prime \circ K^R_\down ].
\]
Thus we are left to identify $K$ factors. We previously computed
\[
K^L_\up= \kappa (\alpha_\up L \oplus (L^i \circ \beta_\up P),
\gamma_\up \mathbold{f} ). 
\]
Recall that when evaluating $\kappa$, one chooses a permutation that orders
$\alpha_\up L \oplus (L^i \circ \beta_\up P)$ and composes the embeddings
accordingly.
Here, $L^i$ is a little $2$--cube that meets the lower face of the unit cube.
In other words, $L^i_t=-1$ and cannot get any lower. Thus, the factors
$(L^i \circ \beta_\up P)^j f_{\gamma_\up (|\ult|_\up+j)}$ can be placed
in first position
when computing~$K^L_\up$. This ultimately shows that
\[
K^L_\up= L^i {K_\up}^\prime \circ K^R_\up.
\]
One deals with $\down$ the exact same way.
For $\updown$, the same phenomenon with Budney's reverse action shows that
the factors
$(L^i \circ \beta_\updown P)^j f_{\gamma_\updown (|\ult|_\updown+j)}$
can be placed in last position when computing $K^L_\updown$, which again
shows the desired
\[
K^L_\updown=K^R_\updown \circ L^i  {K_\updown}^\prime.
\]
\end{proof}

Once again, the concatenation comes as a special case with side-by-side
cubes of equal width.
Budney's action on knots can be recovered and one can also turn a knot
into a double cable or a split link
using identity cubes in $\SCL(s;o)$ for $s \in \{\up, \down, \updown\}$.
More precisely, $\mu_{(s; o)}(\id_{J^2})=\hat\varphi^s$.
This shows that the $\pi_0\SCL$--algebra structure on the quadruplet $\pi_0X$
is the data of the usual monoids
$\pi_0EC(1, D^2)$ and $\pi_0EC_\iota(1, D^2)$,
together with the three distinct independent actions of $\pi_0EC(1, D^2)$ on
$\pi_0EC_\iota(1, D^2)$ given by the $\hat\varphi^s$, $s \in \{\up, \down,
\updown\}$.
Finally, the spaces of unframed knots and links are still stable:

\begin{Thm}
The quadruplet $(\LLhat^0, \KKhat, \KKhat, \KKhat)$ forms a
sub--$\SCL$--algebra of $X$.
\end{Thm}

\begin{proof}
Consider the two monoids $\mathbb{Z}$ and $\mathbb{Z}^{\times 2}$.
The first one acts on the second one in three different ways: on the first
factor of $\mathbb{Z}^{\times 2}$,
on the second factor or diagonally. The data of these three actions is
precisely that of a $\pi_0\SCL$--algebra
structure on the quadruplet $(\mathbb{Z}^{\times 2}, \mathbb{Z}, \mathbb{Z},
\mathbb{Z})$. One can think of this structure
as an $\SCL$--algebra structure.
Thanks to the additive properties of the linking number with respect to
the concatenation of curves,
one easily checks that the framing number $\omega$ turns into a morphism
of $\SCL$--algebras.
The result follows. 
\end{proof}

This action of $\SCL$ on $(\LLhat^0, \KKhat, \KKhat, \KKhat)$
combines all the structure we have met on long knots and $2$--string links
so far.
Moreover, the isotopies exhibiting the commutativity relations discussed
in Subsection \ref{EmbeddingSpacesLink}
can all be obtained with paths in $\SCL$ from a configuration of cubes
to another.
The next section aims to show that this correspondence actually follows
from a deeper result:
a homotopy equivalence between $(\LLhat^0, \KKhat, \KKhat, \KKhat)$ and
a free algebra over $\SCL$. 

\begin{Rem}
It is possible to extend Theorem \ref{SCLActs} to manifolds other than
the $2$--dimensional disk $D^{2}$.
More precisely, when $M$ is a manifold of dimension $n$, we can consider
the space $EC(k, M)$ consisting of the
embeddings from $\mathbb{R}^{k}\times M$ to itself that restrict to the
identity outside of $J^{k}\times M$.
The notation $EC$'' comes from 
the terminology ``embedding'' and
``cubical''.
This space has been intensively studied by Budney and proved to be an
algebra over the $(k{+}1)$--dimensional little
cubes operad $\mathcal{C}_{k+1}$ in \cite{Budney}. 

Similarly, for any fixed embedding $\iota \co M \amalg M \cof M$, one can
define $EC_{\iota}(k, M)$, the space
consisting of embeddings from $\mathbb{R}^{k}\times M^{\amalg 2}$ to
$\mathbb{R}^{k}\times M$ that
restrict to $\id_{\mathbb{R}^k} \times \iota$ outside of $J^{k}\times
M^{\amalg 2}$ and map the interior of
$J^{k}\times M^{\amalg 2}$ to the interior of $J^{k}\times M$.
Burke and Koytcheff mentioned these spaces in \cite{KoytcheffPrime},
alongside their work on the special case corresponding to framed string
links. 

In order to understand the structure on the quadruplet of spaces
\[
X_M = \bigl( EC_{\iota}(k, M), EC(k, M), EC(k, M),EC(k, M) \bigr),
\]
we need a higher dimensional version of the Swiss-cheese operad for links.
Roughly speaking, one can define $\SCL_{k}$ the same way we did
$\mathcal{SCL}$,
except that the operad $\cubes{k}$ is used in the construction, instead
of the $2$--dimensional little cubes operad $\cubes{2}$.
We can then extend Budney's action on $EC(k, M)$ in order to get an
$\mathcal{SCL}_{k+1}$--algebraic structure on the quadruplet $X_{M}$.
The precise formula for this action is the same as the one introduced
before Theorem \ref{SCLActs},
and checking that is specifies a valid $\mathcal{SCL}_{k+1}$--algebra
structure boils down to the verifications already carried out in the
proof above.
\end{Rem}

\hypertarget{secfo}{}
\section{Freeness results}
\label{FreenessResults}

We prove here that the operadic actions constructed in Section \ref{OperadicActions} lead to
free algebras over different operads.
More precisely, we first introduce the main result of 
Budney in~\cite{Budney}, which states that
$\KKhat$ is homotopy equivalent as a $\cubes{2}$--algebra to
$\cubes{2}[\PPhat]$.
A second theorem proved by Burke and Koytcheff in \cite{KoytcheffInfect}
provides an analogous
statement about the action of $\cubes{1}$ on a subspace of $\LLhat$.
We then combine these results to prove the main theorem of this paper,
Theorem \ref{freeSCL},
stating that $(\LLhat^0, \KKhat, \KKhat, \KKhat)$ is homotopy equivalent
to a free $\SCL$--algebra.
These three theorems are proved with very similar methods,
most of them coming from $3$--manifold topology and homotopy theory.
A first paragraph recalls the concepts we need from these fields, and the
following three are each
dedicated to a freeness theorem. The proofs of the results of Budney,
Burke and Koytcheff are only quickly
outlined, since thorough treatments are available in~\cite{Budney} and~\cite{KoytcheffInfect}. 
We still dispense sketches of proofs as the 
arguments they involve will be useful for Theorem \ref{freeSCL}. 

\subsection{Notions of $3$--dimensional topology}
\label{FreenessResultsthDimensionalTopology}

We introduce some basic concepts of $3$--manifold theory. The instances of
$3$--manifolds we will encounter mostly lie
in $\mathbb{R}^3$, so they inherit very nice features. Furthermore, they
are compact, orientable, connected and irreducible.
It is very common when studying $3$--manifolds to deal with embedded surfaces:
we denote by $S^2$ the $2$--sphere,
$D^2$ the disk, $A$ the annulus, $T^2$ the torus and $(T^2)^{\hash 2}$
the genus $2$ oriented surface.
We denote by $P_n$ the $n$--punctured disk, whose boundary splits as an
external component
$\del_{\ext}P_n$ and $n$ internal components $\del_{\inter}P_n$. As for common
$3$--manifolds,
we note $B=J \times D^2$ the cylinder, homeomorphic to a $3$--ball $D^3$,
$H_n= P_n \times I$ the
$n$--handlebody and $C_f \subset B$ the complement of a fat long knots or
a fat $2$--string link $f$.
The boundary of $C_f$ is a torus when $f$ is a fat long knot, and a
$2$--torus when $f$ is a fat $2$--string link.
A recurring procedure in the upcoming proofs is the cutting of $C_f$
along essential surfaces.
We define the latter now.

\begin{Defn}
Let $S$ be a (not necessarily connected) orientable surface embedded in
an orientable $3$--manifold $M$
properly (ie $S \cap \del M=\del S$).
A disk $D \subset M$ with $D \cap S=\del D$ is said to be a
\emph{compressing disk for} $S$ if its boundary does not bound a disk in $S$.
A surface that admits a compressing disk is said to be \emph{compressible},
and a surface different from $S^2$ or $D^2$ admitting no compressing disk
is said to be \emph{incompressible}.
\end{Defn}

\begin{Defn}
Let $S$ be a bordered surface properly embedded in a $3$--manifold $M$.
A \emph{$\del$--compressing disk for} $S$ is a disk $D \subset M$
whose boundary consists of two arcs $\alpha$ and $\beta$ with $\alpha
\subset S$ and $\beta \subset \del M$,
whose interior is disjoint from $S$ and $\del M$,
such that there is no arc $\gamma$ in $\del S$ such that
$\gamma \cup \alpha$ bounds a disk in $S$. A surface that admits a
$\del$--compressing disk is said to be
\emph{$\del$--compressible}. Otherwise, it is \emph{$\del$--incompressible}.
\end{Defn}

\begin{Defn}
A properly embedded surface $S$ in a $3$--manifold $M$ is said to be
\emph{$\del$--parallel} if it can be isotoped
to a piece of $\del M$.
\end{Defn}

\begin{Defn}
A properly embedded orientable surface $S$ in a $3$--manifold $M$ is
\emph{essential} if
one of the following holds.
\begin{enumerate}[leftmargin=*, label=(\roman*), font=\itshape]
\item $S$ is a sphere and does not bound a ball.
\item $S$ is a disk whose boundary does not bound a disk in $\del M$.
\item $S$ is not a sphere nor a disk, it is incompressible,
$\del$--incompressible and not $\del$--parallel.
\end{enumerate}
\end{Defn}

Spaces of embeddings of incompressible surfaces have been
extensively studied by Hatcher in
\cite{HatcherInc}. He describes in his paper how the homotopy type of such
a space depends on $S$.
This result will be used repeatedly so we formulate a precise version here.

\begin{Thm}[Hatcher~\cite{HatcherInc}]
\label{htpyTypeInc}
Let $M$ be an orientable compact connected irreducible $3$--manifold and
$S \cof M$
an essential orientable compact surface in $M$.
Let $\Emb(S, M, \del S)$ be the space of embeddings of $S$ in $M$ whose
values at
$\del S$ are fixed. Then, the component $\Emb(S, M, \del S)_{S}$ of $S
\cof M$ in  $\Emb(S, M, \del S)$
is weakly contractible unless $S$ is closed and the fiber of a bundle
structure on $M$, or if $S$ is a torus.
In these exceptional cases $\pi_i \Emb(S, M) = 0 $ for all $i >1$. In the
bundle case,
the inclusion of the subspace consisting of embeddings with image a fiber
induces an isomorphism on $\pi_1$.
When $S$ is a torus but not the fiber of a bundle structure, the inclusion
$\Diff(S) \cof \Emb(S, M)$
obtained by precomposing $S \cof M$ induces an isomorphism on $\pi_1$.
\end{Thm}

Another tool of $3$--manifold theory that will come in handy is the
JSJ--decomposition.
It provides a way to cut an irreducible manifold into simpler ones. The
cuts are performed along essential tori,
but if one keeps on cutting a manifold until no such torus is available,
the obtained decomposition might
not be unique. A manifold that admits no essential torus is said to be
\emph{atoroidal}.
In order to get a unique decomposition, one must agree not to cut the
pieces that are
Seifert-fibered. The latter are manifolds consisting of disjoint parallel
circles forming a
particular fibering. The precise definition of this fibering is looser
than the notion of fiber bundle with fiber
$S^1$. It is specified for example in \cite{Hatcher3M}. The decomposition
theorem we use is the following:

\begin{Thm}[Jaco, Shalen and Johannson~\cite{JS,J}]
\label{JSJDecomp}
Every orientable, compact, irreducible $3$--manifold $M$ contains a collection
of embedded,
incompressible tori $T$ so that if one removes an open tubular neighborhood
of
$T$ from $M$, the outcome is a disjoint union of Seifert-fibered and
atoroidal manifolds.
Moreover, a minimal collection of such tori is unique up to isotopy.
\end{Thm}

The minimal collection of tori $T$ from Theorem \ref{JSJDecomp}
(or sometimes its isotopy class) is called the JSJ--decomposition of $M$.
In the case where $\del M$ consists of a single component, the piece of
the cut $M$ containing $\del M$
is called the root of the decomposition. The tori of $T$ bounding the root
are referred to as the
base-level tori of $T$. 

Our main concern while studying an orientable compact $3$--manifold $M$
will actually be the homotopy type
of the group of its boundary-fixing diffeomorphisms $\Diff(M, \del M)$.
More precisely, we are interested of the subgroup $\Diffd(M, \del M)$
consisting of the diffeomorphisms
whose derivatives at $\del M$ agree with those of the identity.
This extra condition is relevant for our work because it enables one to
postcompose a fat long knot by an element
of $\Diffd(B, \del B)$ and still end up with a
fat long knot.
The main ingredient we use to prove the three upcoming freeness theorems
is the
following proposition:

\begin{Prop}
\label{cuttingProp}
Let $M$ be an orientable compact connected irreducible $3$--manifold and $S$
an essential surface 
which cuts $M$
into pieces $M_i$, such that the component $\Emb(S, M, \del S)_S$ is weakly
contractible and stable under the
postcomposition action of\/ $\Diff(M, \del M)$. Then the inclusion
\[
\prod_i \Diffd(M_i, \del M_i) \cof \Diffd(M, \del M)
\]
is a weak homotopy equivalence.
\end{Prop}

\begin{proof}
Thanks to the stability condition, we have a well-defined postcomposition
map
$\Diff(M, \del M) \rightarrow \Emb(S, M, \del S)_S$.
Restriction maps such as this one have been shown to be locally trivial,
and in particular fibrations. This problem, as well as the local triviality
of the restriction map
$\Emb(M, N) \rightarrow \Emb(M^\prime, N)$ for a submanifold $M^\prime
\subset M$,
has been treated in several articles: in \cite{Palais} for the case of
closed manifolds and in
\cite{Cerf} for the case of bordered manifolds. A more recent exposition
is provided in the third section
of \cite{McCullough}: the precise result we use here is formulated as
Corollary~3.7 in \cite{McCullough}.
The fiber over $S$ of this map is the subgroup of diffeomorphisms fixing
$S$, ie
\[
\begin{tikzcd}
[row sep=tiny, column sep=small]
\Diff(M, \del M \cup S)
\arrow[r, hook]
&\Diff(M, \del M)
\arrow[r]
&\Emb(S, M, \del S)_S.
\end{tikzcd}
\]
The base space is weakly contractible so the inclusion of the fiber is a
weak homotopy equivalence.
We are left to add the derivative condition on the diffeomorphisms.  

It is proved in Kupers' book on diffeomorphism groups \cite{Kupers}
that the inclusion of the subgroup
$\Diffd(N, \del N) \cof \Diff(N, \del N)$ is a weak homotopy equivalence
for every compact manifold $N$.
This justifies the bottom left equivalence in the diagram of inclusions
\[
\begin{tikzcd}
\Diffd(M, \del M \cup S) \arrow[r, hook] \arrow[d,equal]
&\Diff(M, \del M \cup S) \arrow[r, hook, "\simeq"]
&\Homeo(M, \del M \cup S) \arrow[d, equal] \\
\prod_i \Diffd(M_i, \del M_i) \arrow[r, hook, "\simeq"]
&\prod_i \Diff(M_i, \del M_i) \arrow[r, hook, "\simeq"]
&\prod_i\Homeo(M_i, \del M_i)
\end{tikzcd}
\]
while the right horizontal equivalences come from works of Cerf in \cite{Cerf}.
The two-out-of-three rule assures us that the top left inclusion is a weak
equivalence as well.
This same rule in the diagram
\[
\begin{tikzcd}
\Diffd(M, \del M \cup S) \arrow[r, hook] \arrow[d, hook, "\simeq"]
&\Diffd(M, \del M) \arrow[d, hook, "\simeq"] \\
\Diff(M, \del M \cup S) \arrow[r, hook, "\simeq"]
&\Diff(M, \del M)
\end{tikzcd} 
\]
concludes the proof. 
\end{proof}

The need to study these diffeomorphism groups arises from the following
classical result in modern
knot theory:

\begin{Prop}
\label{BDiff}
Let $f$ be a fat long knot or a fat $2$--string link. Then the component
$\KKhat_f$ or $\LLhat_f$ of $f$
in $\KKhat$ or $\LLhat$ is a model for the classifying space of\/ $\Diffd(C_f,
\del C_f)$.
Moreover, it is a $K(G, 1)$.
\end{Prop}

\begin{proof}
We treat the case where $f$ is a fat long knot, the other one is treated
identically.
When $B$ is the solid cylinder and $C_f$ the complement of $f$ in $B$,
we have the inclusion and restriction maps
\[
\begin{tikzcd}
[row sep=tiny, column sep=small]
\Diffd(C_f, \del C_f)
\arrow[r, hook]
&\Diffd(B, \del B)
\arrow[r]
&\KKhat_f.
\end{tikzcd}
\]
The application on the right-hand side precomposes a diffeomorphism by $f$
and is a fibration thanks to \cite[Corollary~3.7]{McCullough}.
Now, $\Diffd(B, \del B) \simeq \Diff(B, \del B)$ is weakly contractible,
as proved in \cite{HatcherSmale}.
Thus, $\Diffd(C_f, \del C_f)$ acts properly and freely on a contractible
space, and the quotient of this action
can be identified to $\KKhat_f$, so $\KKhat_f$ is a model for $B\Diffd(C_f,
\del C_f)$.
Moreover, spaces of diffeomorphisms of orientable Haken bordered
$3$--manifolds which preserve the
boundary pointwise always have vanishing higher homotopy groups by \cite[Theorem~2]{HatcherInc}.
Thus, the long exact sequence in homotopy coming from the fibration above
assures us that
$\KKhat_f$ is a $K(\pi_0 \Diffd(C_f, \del C_f), 1)$. 
\end{proof}

\subsection{Budney's freeness theorem}

We now recall Budney's main theorem in \cite{Budney}, which is a freeness
statement about the space of fat long
knots as a $\cubes{2}$--algebra. Recall from Subsection \ref{EmbeddingSpacesKnot}
that the prime fat long knots are denoted by $\PPhat \subset \KKhat$,
and that they are the embeddings whose isotopy classes are prime elements
in the monoid $\pi_0 \KKhat$.
The result we present appears as Theorem~11 in Budney's paper, and we only
dispense a
sketch of proof here, as the complete proof is fairly long.

\begin{Thm}[Budney~\cite{Budney}]
\label{BudneyFreeness}
The restriction of the structure map
\[
\kappa \co \cubes{2}[\PPhat] \rightarrow \KKhat
\]
from Theorem \ref{CtwActsOnEC-on-Dtw-} is a homotopy equivalence.
\end{Thm}

\begin{proof}[Sketch of Proof]
Thanks to Theorems \ref{KnotStructureThm}, \ref{pizeFreenessCommute} and the fact
that applying $\pi_0$ to the structure map $\kappa\co \cubes{2} \rightarrow
\EE_{\KKhat}$ endows
$\pi_0\KKhat$ with its usual monoid structure, we are assured that $\kappa$
induces a bijection on components.
Thus, we are left to prove that it is a homotopy equivalence on each of
these components. 

On the component of the unknot, $\kappa$ restricts to the map
$\cubes{2}(0) \times \PPhat^{\times 0} \rightarrow \KKhat_{\id_{\mathbb{R}
\times D^2}}$.
The complement of an unknot is a $1$--handlebody $H_1$, and the diffeomorphism
group
$\Diffd(H_1, \del H_1) \simeq \Diff(H_1, \del H_1)$ is contractible.
Proposition \ref{BDiff} therefore gives the contractibility of
$\smash{\KKhat_{\id_{\mathbb{R} \times D^2}}} = B\Diffd(H_1, \del H_1)$ 
so we have an equivalence in this case. 

On the component of a prime knot $f \in \PPhat$, $\kappa$
restricts to the map
$\cubes{2}(1) \times \PPhat_f \rightarrow \PPhat_f=\KKhat_f$. The homotopy
retracting $\cubes{2}(1)$
onto the identity little cube shows that this map is an equivalence as
well. 

Suppose now that $f$ is a composite knot $f = f_1 \hash \cdots
\hash f_n$.
Budney proved in \cite{BudneyJSJ} that the base-level tori $T$ of the
JSJ--decomposition
of the complement of $f$ split $C_f$ into $n+1$ pieces: the complements
of the prime factors $C_{f_i}$ and
a root homeomorphic to $S^1 \times P_n$. One would like to apply Proposition \ref{cuttingProp} in order to
split $\Diffd(C_f, \del C_f)$ as a product of diffeomorphism groups
involving each
$\Diffd(C_{f_i}, \del C_{f_i})$. But, cutting along $T$ is not possible
here as the component
$\Emb(T^2 \amalg \cdots \amalg T^2, C_f)_T$ is not contractible by Theorem \ref{htpyTypeInc}.
It is also not stable under the postcomposition action of $\Diff(C_f,
\del C_f)$.
Indeed, when $T_i$ is the torus bounding $C_{f_i}$, two tori $T_i$ and $T_j$
can be permuted
by some diffeomorphism of $C_f$ if and only if $f_i$ and $f_j$ are
isotopic. Let $\Sigma_f$
be the subgroup of $\Sigma_n$ preserving the partition of $\{1, \dots,
n\}$ given by $i \sim j$
if and only if $f_i$ and $f_j$ are isotopic. Then, by considering
the components of $\Emb(T^2 \amalg \cdots \amalg T^2, C_f)$ where the
image of the $i^{\text{th}}$ torus
is isotopic to a $T_j$ with $i \sim j$, and by quotienting out these
components
by the parametrization of each torus, one gets a space homotopy equivalent
to $\Sigma_f$.
It is stable under the postcomposition action of $\Diff(C_f, \del C_f)$
so one can use an argument
similar to the proof of Proposition \ref{cuttingProp} to split $\Diffd(C_f,
\del C_f)$ into
a product involving each $C_{f_i}$.
Namely, after some manipulations on the diffeomorphisms, Budney manages
to fit $\Diffd(C_f, \del C_f)$
up to homotopy in a fibration
\[
\KDiff(P_n, \del P_n) \times \prod_i \Diffd(C_{f_i}, \del C_{f_i})
\rightarrow
\Diffd(C_f, \del C_f) \rightarrow
\Sigma_f,
\]
for some subgroup $\KDiff(P_n, \del P_n)$ homotopy equivalent to the pure
braid group on $n$ strands
$KB_n$. Since $BKB_n = \conf_n(J^2) \simeq \cubes{2}(n)$,
applying the classifying space functor $B$ leads to
\begin{align*}
\KKhat_f & \simeq B \Diffd(C_f, \del C_f) \\
& \simeq B \KDiff(P_n, \del P_n)
\times_{\Sigma_f} \prod_i B\Diffd(C_{f_i},
\del C_{f_i}) \\
& \simeq \cubes{2}(n) \times_{\Sigma_f}
\prod_i \KKhat_{f_i} = \cubes{2}[\PPhat]_f.
\end{align*}
At this stage, this equivalence is merely an abstract map.
But the vast majority of the group morphisms we met are inclusion-based,
so Budney manages to show
that this equivalence coincides with $\kappa$ via explicit models. Weak
equivalences are finally
promoted to strong ones via an application of Whitehead's theorem. 
\end{proof}

This result can be thought of as a generalization of Schubert's theorem:
the connected sum operation $\hash$ is extended to an algebraic structure
on the space level
of $\KKhat$, and the isomorphism $\pi_0\KKhat = \Com[\pi_0\PPhat]$
is extended to the $\cubes{2}$--equivariant homotopy equivalence $\KKhat
\simeq \cubes{2}[\PPhat]$.
Note however that Budney uses Theorem \ref{KnotStructureThm} in the very
first sentence of the proof,
so that this 
generalization is in no way an alternative argument for
Schubert's result. 

\subsection{Burke and Koytcheff's freeness theorem}

We now get to Burke and Koytcheff's result about
fat $2$--string links. Recall from Subsection \ref{EmbeddingSpacesLink}
that $\QQhat^0 \subset \LLhat^0$ denotes the fat $2$--string links which
are prime
but not in the image of one of the maps $\hat\varphi^s$, $s \in \{\up,
\down, \updown\}$.
The theorem we present here is \cite[Theorem~6.8]{KoytcheffInfect}.
We again provide a quick sketch of proof, as the ideas involved will
reappear in the proof of our main result, Theorem \ref{freeSCL}.

\begin{Thm}[Burke and Koytcheff~\cite{KoytcheffInfect}]
\label{KoytcheffFreeness}
Let $\SShat^0$ be the subspace of $\LLhat^0$ consisting of the fat $2$--string
links whose prime factors lie in $\QQhat^0$.
Then, the restriction of the structure map
\[
\lambda \co \cubes{1}[\QQhat^0] \rightarrow \SShat^0
\]
from Theorem \ref{ConActsOnECiota-on-Dtw-} is a homotopy equivalence.
\end{Thm}

\begin{proof}[Sketch of proof]
Thanks to Theorem \ref{LinkStructureThm}, we are assured that $\lambda$
induces a bijection on components.
We are therefore again left to show that it is an equivalence on each of
these components. 

On the component of the trivial link, $\lambda$ restricts to
the map $\cubes{1}(0) \times  (\QQhat^0)^{\times 0} \rightarrow
\LLhat^0_{\id_{\mathbb{R}} \times \iota}$.
The complement of $\id_{\mathbb{R}} \times \iota$ is a $2$--handlebody $H_2$,
and the diffeomorphism group
$\Diffd(H_2, \del H_2) \simeq \Diff(H_2, \del H_2)$ is contractible so we
can conclude as in the case of knots. 

Let now $f$ be a nontrivial element of $\SShat^0$ and $f =
f_1 \hash \cdots \hash f_n$
its decomposition in noncentral prime fat $2$--string links.
According to \cite[Theorem~4.1]{KoytcheffPrime}, there are $n-1$
twice-punctured disks separating $C_f$ into the complements of the
$f_i$. Moreover,
as shown in the fourth step of the proof of this same theorem, these disks
are unique up to isotopy.
This enables us to split $\Diffd(C_f, \del C_f)$ as the product $\prod_i
\Diffd(C_{f_i}, \del C_{f_i})$
thanks to Proposition \ref{cuttingProp}. Applying the classifying space
functor yields the equivalences
\[
		\SShat^0_f
=		B\Diffd(C_f, \del C_f)
\simeq	\prod_i B\Diffd(C_{f_i},\del C_{f_i})
\simeq	\prod_i \LLhat_{f_i}.
\]
When we denote by $\cubes{1}(n)_{(1, \dots, n)}$ the (contractible)
component of
$\cubes{1}(n)$ where the intervals appear from left to right in the order
$(1, \dots, n)$,
we can go further and write
\[
\SShat^0_f
\simeq	\prod_i \LLhat_{f_i}
\simeq	\cubes{1}(1)_{(1, \dots, n)} \times \prod_i \LLhat_{f_i}
= \cubes{1}[\QQhat^0]_f.
\]
At this stage, the equivalence is still given by an abstract
map but it is easy to keep track of the
identifications and see that it coincides with $\lambda$. Weak equivalences
are finally
promoted to strong ones via an application of Whitehead's theorem. 
\end{proof}

Again, this result is a generalization of Theorem \ref{LinkStructureThm}
as its provides a free algebraic structure
on the space level of $\SShat^0$, which descends to the usual free monoid
on the basis $\QQhat^0$
on isotopy classes. 

\subsection{Fat long knots and $2$--string links form a free $\SCL$--algebra}

We combine in this last paragraph the theorems of Budney, Burke and
Koytcheff presented above to prove
a freeness result for the whole space of fat $2$--string links. Namely,
we prove:

\begin{Thm}
\label{freeSCL}
The restriction of the structure map
\[
\mu\co \SCL[\QQhat^0, \PPhat, \PPhat, \PPhat] \rightarrow (\LLhat^0,
\KKhat, \KKhat, \KKhat)
\]
from Theorem \ref{SCLActs} is a homotopy equivalence.
\end{Thm}

Just as in the 
preceding paragraphs, the proof mainly consists
in reducing ourselves to each connected
component and splitting the diffeomorphism group of the complement of a
link along suitable surfaces.
We state two technical lemmas to prepare this cutting process, then proceed
to the proof of the theorem.

\begin{Lem}
\label{LemthDiscs}
Let $f$ be a fat $2$--string link decomposing as
$$f^o \hash \hat\varphi^\up(f^\up) \hash \hat\varphi^\down(f^\down) \hash
\hat\varphi^\updown (f^\updown)$$
for some element $f^o$ in $\SShat^0$ and fat long knots $f^s$, $s \in \{\up,
\down, \updown\}$.
Then, the three vertical twice-punctured disks $D$ cutting $C_f$ into
$C_{f^o}$, $C_{\hat\varphi^\up(f^\up)}$, $C_{\hat\varphi^\up(f^\up)}$
and $C_{\hat\varphi^\up(f^\up)}$
are unique up to isotopy fixing the boundary. Moreover, the component
$\Emb(D, C_f, \del D)_D$ is
weakly contractible and stable under the postcomposition action
of $\Diff(C_f, \del C_f)$.
\end{Lem}

\begin{proof}
The uniqueness statement is proved in the steps $1$ and $3$ of Blair,
Burke and Koytcheff's proof of their Theorem~4.1 in \cite{KoytcheffPrime}. Applying a diffeomorphism of $\Diff(C,
\del C_f)$ to the punctured disks $D$ does not
change the fact that they split $C_f$ into pieces homeomorphic to the
complement of
$f^o$ and $\hat\varphi^s(f^s)$ for $s \in \{\up, \down, \updown\}$. Thus,
the component $\Emb(D, C_f, \del D)_D$ is stable under the postcomposition
action of this diffeomorphism
group. Its contractibility immediately follows from Hatcher's work on
incompressible surfaces
(Theorem \ref{htpyTypeInc}).
\end{proof}

\begin{figure}
\small
\def\svgwidth{\hsize}
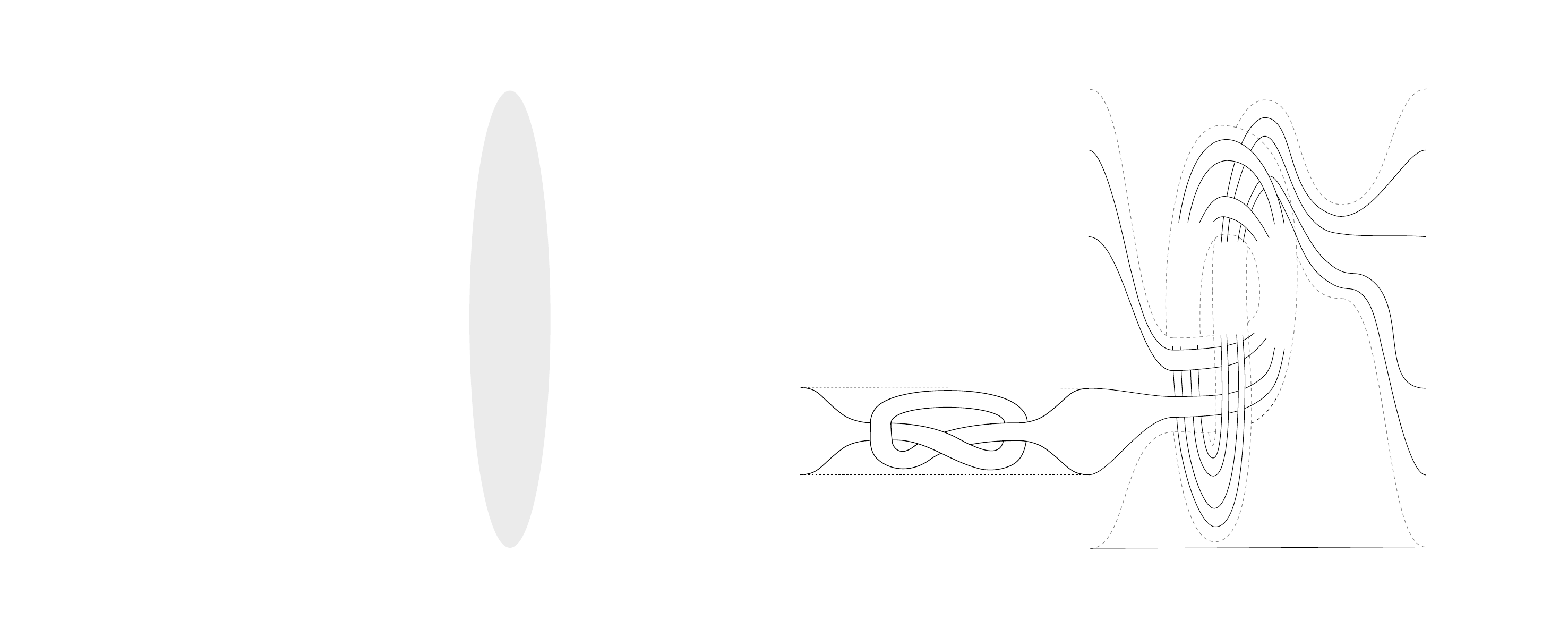
\caption{Illustration of the three disks $D=D_1 \amalg D_2 \amalg D_3$ in $C_f$.}
\end{figure}

\begin{Lem}
\label{LemthAnnuli}
Let $f$ be a fat long knot and $\hat\varphi^s(f)$ the central link obtained
from $f$ for 
$s \in \{\up, \down, \updown\}$. Consider the annulus $A_s$ in the complement
$C_{\hat\varphi^s(f)}$ specified by
\begin{enumerate}[leftmargin=*, label=(\roman*), font=\itshape]
\item $A_\up = (\id_{\mathbb{R}} \times \iota)((J \times \del D^2) \amalg
\emptyset)$,
\item $A_\down = (\id_{\mathbb{R}} \times \iota)(\emptyset \amalg (J
\times \del D^2))$,
\item $A_\updown = f(J \times \del D^2)$.
\end{enumerate}
Then the isotopy class of $A_s$ is stable under the postcomposition
action of
$$\Diff(C_{\hat\varphi^s(f)}, \del C_{\hat\varphi^s(f)})$$
and the component
$\Emb(A, C_{\hat\varphi^s(f)}, \del A)_{A_s}$ is weakly contractible.
\end{Lem}

\begin{figure} 
\small
\def\svgwidth{\hsize}
\begingroup%
  \makeatletter%
  \providecommand\color[2][]{%
    \errmessage{(Inkscape) Color is used for the text in Inkscape, but the package 'color.sty' is not loaded}%
    \renewcommand\color[2][]{}%
  }%
  \providecommand\transparent[1]{%
    \errmessage{(Inkscape) Transparency is used (non-zero) for the text in Inkscape, but the package 'transparent.sty' is not loaded}%
    \renewcommand\transparent[1]{}%
  }%
  \providecommand\rotatebox[2]{#2}%
  \newcommand*\fsize{\dimexpr\f@size pt\relax}%
  \newcommand*\lineheight[1]{\fontsize{\fsize}{#1\fsize}\selectfont}%
  \ifx\svgwidth\undefined%
    \setlength{\unitlength}{1785.82677165bp}%
    \ifx\svgscale\undefined%
      \relax%
    \else%
      \setlength{\unitlength}{\unitlength * \real{\svgscale}}%
    \fi%
  \else%
    \setlength{\unitlength}{\svgwidth}%
  \fi%
  \global\let\svgwidth\undefined%
  \global\let\svgscale\undefined%
  \makeatother%
  \begin{picture}(1,0.24603175)%
    \lineheight{1}%
    \setlength\tabcolsep{0pt}%
    \put(0,0){\includegraphics[width=\unitlength,page=1]{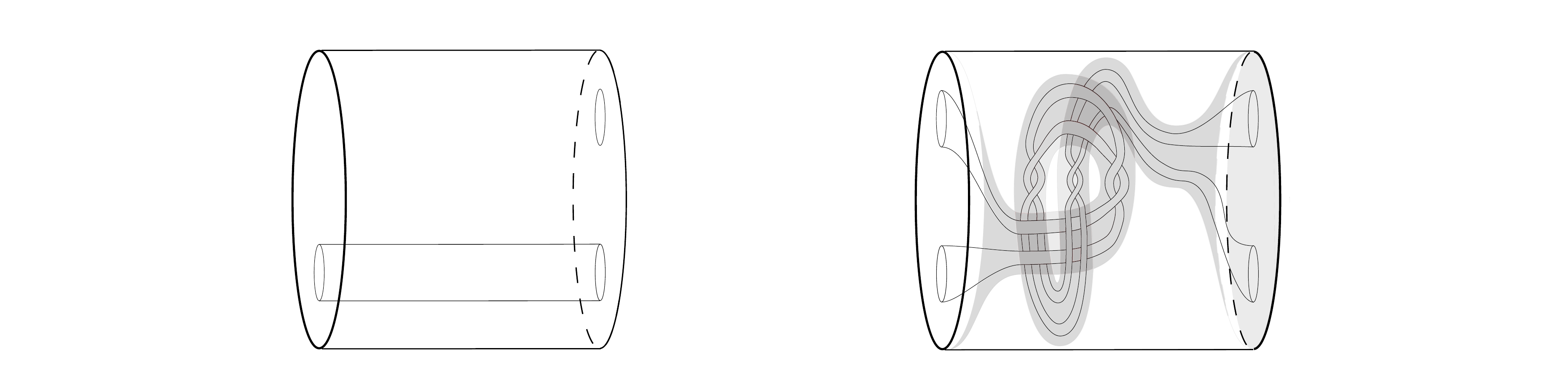}}%
    \put(0.31731707,0.15569242){\color[rgb]{0,0,0}\makebox(0,0)[lt]{\lineheight{1.25}\smash{\begin{tabular}[t]{l}$A_\up$\end{tabular}}}}%
    \put(0,0){\includegraphics[width=\unitlength,page=2]{AnnuliAs.pdf}}%
    \put(0.78598313,0.10159772){\color[rgb]{0,0,0}\makebox(0,0)[lt]{\lineheight{1.25}\smash{\begin{tabular}[t]{l}$A_\updown$\end{tabular}}}}%
    \put(0,0){\includegraphics[width=\unitlength,page=3]{AnnuliAs.pdf}}%
  \end{picture}%
\endgroup%

\caption{Illustration of the annuli $A_\up$ and $A_\updown$.}
\end{figure}

\begin{proof}
We first treat the case where $s=\up$. Consider a horizontal disk
$E \subset \smash{C_{\hat\varphi^\up(f)}}$
separating the two strands of
$\hat\varphi^\up(f)$.
Cutting along $E$ yields two manifolds: an upper piece containing $A_\up$,
homeomorphic to $C_f$, and a lower one that is a $1$--handlebody~$H_1$.
Any two disks in $C_{\hat\varphi^\up(f)}$ sharing their boundary are
isotopic because
$C_{\hat\varphi^\up(f)}$ is irreducible. Therefore, given a diffeomorphism
$g$ in
$\Diff(C_{\hat\varphi^\up(f)}, \del C_{\hat\varphi^\up(f)})$, there is an
isotopy from $g(E)$ to $E$.
It can be extended to a boundary preserving ambient isotopy and
postcomposing $g$ by the latter shows that we may assume that $g(E)=E$.
In other words, $g$ preserves the cut and in particular $g(A_\up)$ lies
in the upper piece.
But $A_\up$ is boundary parallel after the cut, so it is unique up to
boundary-fixing isotopy in its piece,
which concludes the proof in this case.
The weak contractibility statement immediately follows from Theorem \ref{htpyTypeInc}.
The case where $s=\down$ is treated the exact same way. 

We now deal with the case where $s=\updown$.
The proof of the stability statement uses the JSJ--decomposition recalled
in Theorem \ref{JSJDecomp},
especially the uniqueness part. The idea is to show that the
JSJ--decomposition of $C_{\hat\varphi^\updown(f)}$
admits a single base-level torus $T_\updown$, which must be unique up to
isotopy, and that $A_\updown$
is a suitable piece of it. 

Let ${A_\updown}^\prime$ be the annulus
$J \times \del D^2 \subset \del C_{\hat\varphi^\updown(f)}$.
The two annuli $A_\updown$ and ${A_\updown}^\prime$ share their boundary
so that their union
is a torus in $C_{\hat\varphi^\updown(f)}$. Let $T_\updown$ be the torus
obtained by pushing
$A_\updown \cup {A_\updown}^\prime$ in the interior of
$C_{\hat\varphi^\updown(f)}$.
It is essential and cuts $C_{\hat\varphi^\updown(f)}$ into two manifolds,
one containing $\del C_{\hat\varphi^\updown(f)}$, 
$A_\updown$ and ${A_\updown}^\prime$
that we denote by $V$,
the other one homeomorphic to $C_f$. We now proceed to show that $T_\updown$
is the base-level
torus of the JSJ--decomposition of $C_{\hat\varphi^\updown(f)}$.

\begin{Claim}
Let $P_2$ be a twice-punctured disk and $\gamma_\updown$ a curve in the
interior of $P_2$ parallel to the
external boundary circle. Then $V$ is homeomorphic to a $2$--handlebody
$J \times P_2$ deprived of a
solid torus that is a tubular neighborhood of $0 \times \gamma_\updown$.
\end{Claim}

\begin{proof}
This unknotting process is similar to Budney's ``untwisted reembedding''
described in the beginning of
his paper on knot complements \cite{BudneyJSJ}.
Cutting $V$ along $A_\updown$ results in two pieces: an external one
containing $T_\updown$
and an internal one that is a (knotted) $2$--handlebody $H_2$.
The torus $T_\updown$ is boundary parallel in the external part, so this
piece is a fattened torus $T^2 \times I$.
This shows that $V$ is the manifold obtained by gluing a $2$--handlebody
$H_2$ and a fattened torus
$T^2 \times I$ along specific annuli in their boundaries.
This description also matches our new model for $V$ so we are done. 
\end{proof}

This new model makes a lot of considerations easier since it
forgets all the complexity
of the knotting of $f$. We call the two unknotted annuli $J \times \del_{\inter}
P_2$ the \emph{strands} of $\del V$.
Through this identification, $A_\updown$ turns into the annulus
bounding a neighborhood of the two strands not containing $T_\updown$,
as illustrated in Figure \ref{AandV}. The second annulus
${A_\updown}^\prime$ remains in the 
boundary as $J \times \del D^2$.

\begin{figure}
\small
\def\svgwidth{\hsize}
\begingroup%
  \makeatletter%
  \providecommand\color[2][]{%
    \errmessage{(Inkscape) Color is used for the text in Inkscape, but the package 'color.sty' is not loaded}%
    \renewcommand\color[2][]{}%
  }%
  \providecommand\transparent[1]{%
    \errmessage{(Inkscape) Transparency is used (non-zero) for the text in Inkscape, but the package 'transparent.sty' is not loaded}%
    \renewcommand\transparent[1]{}%
  }%
  \providecommand\rotatebox[2]{#2}%
  \newcommand*\fsize{\dimexpr\f@size pt\relax}%
  \newcommand*\lineheight[1]{\fontsize{\fsize}{#1\fsize}\selectfont}%
  \ifx\svgwidth\undefined%
    \setlength{\unitlength}{1700.78740157bp}%
    \ifx\svgscale\undefined%
      \relax%
    \else%
      \setlength{\unitlength}{\unitlength * \real{\svgscale}}%
    \fi%
  \else%
    \setlength{\unitlength}{\svgwidth}%
  \fi%
  \global\let\svgwidth\undefined%
  \global\let\svgscale\undefined%
  \makeatother%
  \begin{picture}(1,0.25833333)%
    \lineheight{1}%
    \setlength\tabcolsep{0pt}%
    \put(0,0){\includegraphics[width=\unitlength,page=1]{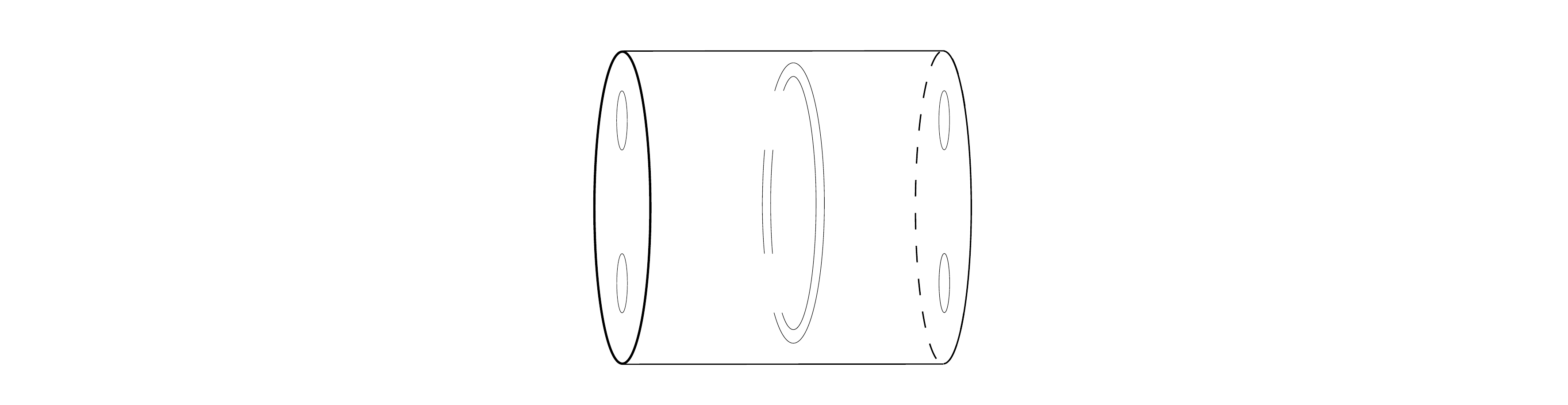}}%
    \put(0.53588064,0.12214193){\color[rgb]{0,0,0}\makebox(0,0)[lt]{\lineheight{1.25}\smash{\begin{tabular}[t]{l}$A_\updown$\end{tabular}}}}%
    \put(0,0){\includegraphics[width=\unitlength,page=2]{AandV.pdf}}%
  \end{picture}%
\endgroup%

\caption{Illustration of the new model for $V$ and $A_\updown$.}
\label{AandV}
\end{figure}

\begin{Claim}
\label{VAtoroidal}
$V$ is atoroidal.
\end{Claim}

\begin{proof}
In many cases including that of $V$, the definition of an atoroidal
$3$--manifold can be reformulated algebraically
using the fundamental group. More precisely, one defines a peripheral
subgroup of a $3$--manifold $M$ to be
a subgroup of $\pi_1M$ that lies in the image of the inclusion of a
boundary component,
then declares $M$ to be atoroidal if every subgroup of $\pi_1M$ isomorphic
to $\mathbb{Z}^{\times 2}$
is conjugate to a peripheral subgroup.
The equivalence between the two definitions is alluded for example in the
beginning of \cite{AFW}, but
the implication we are about to use follows from Corollary~5.5 in
Waldhausen's article \cite{Waldhausen}. 

\begin{figure}
\small
\def\svgwidth{\hsize}
\begingroup%
  \makeatletter%
  \providecommand\color[2][]{%
    \errmessage{(Inkscape) Color is used for the text in Inkscape, but the package 'color.sty' is not loaded}%
    \renewcommand\color[2][]{}%
  }%
  \providecommand\transparent[1]{%
    \errmessage{(Inkscape) Transparency is used (non-zero) for the text in Inkscape, but the package 'transparent.sty' is not loaded}%
    \renewcommand\transparent[1]{}%
  }%
  \providecommand\rotatebox[2]{#2}%
  \newcommand*\fsize{\dimexpr\f@size pt\relax}%
  \newcommand*\lineheight[1]{\fontsize{\fsize}{#1\fsize}\selectfont}%
  \ifx\svgwidth\undefined%
    \setlength{\unitlength}{1700.78740157bp}%
    \ifx\svgscale\undefined%
      \relax%
    \else%
      \setlength{\unitlength}{\unitlength * \real{\svgscale}}%
    \fi%
  \else%
    \setlength{\unitlength}{\svgwidth}%
  \fi%
  \global\let\svgwidth\undefined%
  \global\let\svgscale\undefined%
  \makeatother%
  \begin{picture}(1,0.3)%
    \lineheight{1}%
    \setlength\tabcolsep{0pt}%
    \put(0,0){\includegraphics[width=\unitlength,page=1]{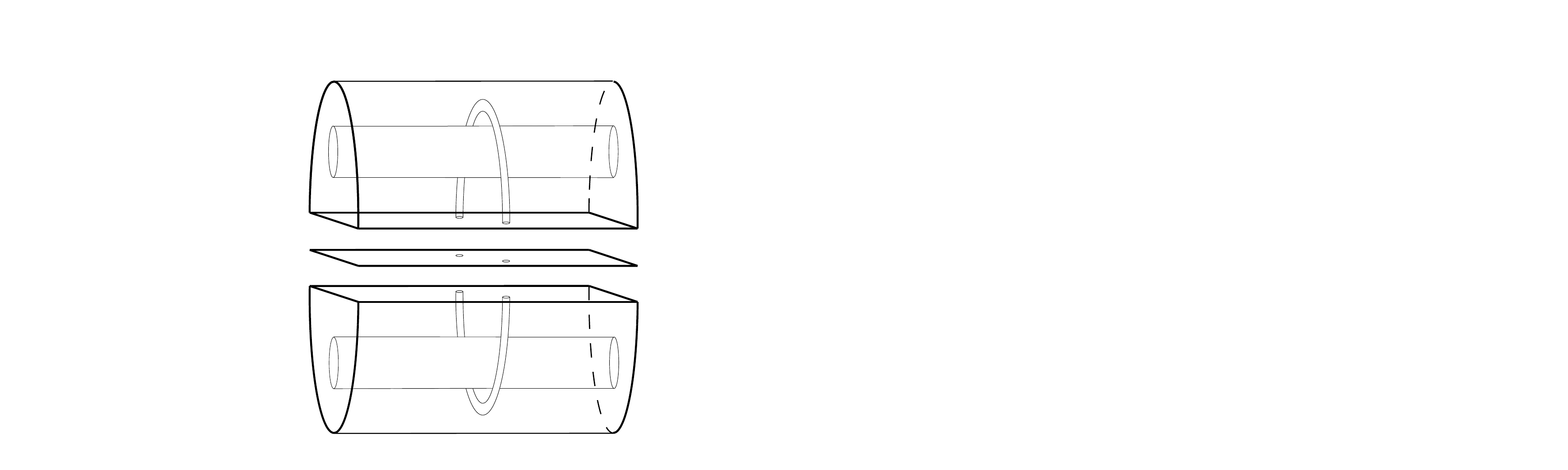}}%
    \put(0.11856063,0.19488977){\color[rgb]{0,0,0}\makebox(0,0)[lt]{\lineheight{1.25}\smash{\begin{tabular}[t]{l}$H_2 \cong$\\\end{tabular}}}}%
    \put(0.12120646,0.13316276){\color[rgb]{0,0,0}\makebox(0,0)[lt]{\lineheight{1.25}\smash{\begin{tabular}[t]{l}$P_2 \cong$\\\end{tabular}}}}%
    \put(0.11856063,0.06319954){\color[rgb]{0,0,0}\makebox(0,0)[lt]{\lineheight{1.25}\smash{\begin{tabular}[t]{l}$H_2 \cong$\\\end{tabular}}}}%
    \put(0,0){\includegraphics[width=\unitlength,page=2]{VanKampenV.pdf}}%
    \put(0.7294274,0.17799763){\color[rgb]{0,0,0}\makebox(0,0)[lt]{\lineheight{1.25}\smash{\begin{tabular}[t]{l}$\alpha_1$\\\end{tabular}}}}%
    \put(0.72940281,0.0863606){\color[rgb]{0,0,0}\makebox(0,0)[lt]{\lineheight{1.25}\smash{\begin{tabular}[t]{l}$\alpha_2$\\\end{tabular}}}}%
    \put(0.79314156,0.12850423){\color[rgb]{0,0,0}\makebox(0,0)[lt]{\lineheight{1.25}\smash{\begin{tabular}[t]{l}$\beta$\\\end{tabular}}}}%
    \put(0,0){\includegraphics[width=\unitlength,page=3]{VanKampenV.pdf}}%
  \end{picture}%
\endgroup%

\caption{Applying van Kampen's theorem to compute $\pi_1V$.}
\label{VanKampenV}
\end{figure}

We start by computing $\pi_1V$. The application of van Kampen's
theorem summarized in Figure \ref{VanKampenV} 
gives
\[
\pi_1V=\langle \alpha_1, \alpha_2, \beta \mid
\alpha_1\beta\alpha_1^{-1}=\alpha_2\beta\alpha_2^{-1}\rangle.
\]
We can further write 
\[ 
\pi_1V\cong\langle a, b, c \mid  ab=ba \rangle = \mathbb{Z}^{\times 2}
\ast \mathbb{Z}
\]
via $\alpha_1 \mapsto c$, $\alpha_2 \mapsto ca^{-1}$ and $\beta \mapsto b$. 

The inclusion of the toric boundary component $T_\updown \subset
V$ has image in $\pi_1$
the subgroup generated by $a$ and $b$. It corresponds to the
$\mathbb{Z}^{\times 2}$ factor
under the isomorphism $\pi_1V \cong \nobreak \mathbb{Z}^{\times 2} \ast \mathbb{Z}$.
We prove that any subgroup of $\mathbb{Z}^{\times 2} \ast \mathbb{Z}$
isomorphic to
$\mathbb{Z}^{\times 2}$ is conjugate to a subgroup of the latter.
Consider an arbitrary injection $\mathbb{Z}^{\times 2} \hookrightarrow
\mathbb{Z}^{\times 2} \ast \mathbb{Z}$
and denote by $x$ and $y$ the images of $(0, 1)$ and $(1, 0)$.
These $x$ and $y$ commute, they are nontrivial and they are not allowed
to be powers of some third element.
Now, applying Theorem~4.5 on page~209 of the book \cite{CGT} on groups
and presentations yields
the following three possibilities for $x$ and $y$:
\begin{enumerate}[leftmargin=*, label=(\roman*), font=\itshape]
\item $x$ or $y$ may be trivial;
\item if neither $x$ nor $y$ is trivial, but $x$ is in the conjugate of
a factor, then $y$ is in that same conjugate of a factor;
\item if neither $x$ nor $y$ is in a conjugate of a factor, then they are
powers of a third element of
$\mathbb{Z}^{\times 2} \ast \mathbb{Z}$.
\end{enumerate}
\unskip
Thanks to our observations, \textit{(i)} and \textit{(iii)} are ruled out and $x$ and $y$
must lie in the same conjugate of a factor of
$\mathbb{Z}^{\times 2} \ast \mathbb{Z}$.
The $\mathbb{Z}$ factor does not admit any subgroup 
isomorphic to
$\mathbb{Z}^{\times 2}$, so we are done.
\end{proof}

At this point, we showed that $T_\updown$ and the tori of the
JSJ--decomposition of $C_{f}$ cut
$C_{\hat\varphi^\up(f)}$ into atoroidal and Seifert-fibered pieces.
Finally, $T_\updown$ cannot be removed to obtain a smaller decomposition
because $V$ can never be part of a
Seifert-fibered manifold, having a boundary component homeomorphic to
a $2$--torus.
This shows that $T_\updown$ and the tori of the JSJ--decomposition of $C_f$
form a minimal decomposition
of $C_{\hat\varphi^\up(f)}$ into atoroidal and Seifert-fibered manifolds,
which proves that this collection is the JSJ--decomposition of
$C_{\hat\varphi^\up(f)}$.
The root $V$ is bounded by $\del C_{\hat\varphi^\updown(f)}$ and $T_\updown$
so $T_\updown$ is the only base-level torus. Its image is therefore unique
up to isotopy. 

Consider now $g \in \Diff(C_{\hat\varphi^\updown(f)}, \del
C_{\hat\varphi^\updown(f)})$.
Thanks to the work above, there is an isotopy from $g(T_\updown)$ to
$T_\updown$.
It can be extended to a boundary-fixing ambient isotopy, and postcomposing
$g$ by the latter shows that
we may assume that $g(T_\updown)=T_\updown$. In other words, $g$ preserves
the cut along $T_\updown$
and $g(A_\updown)$ lies in $V$.
Let us now take a look at the (boundary-fixing) isotopy classes of annuli
in $V$ with boundary $\del A_\updown$.
We prove that there are only two of them: the one of $A_\updown$ and the
one of ${A_\updown}^\prime$.
Let $A \subset V$ be an arbitrary annulus with adequate boundary.
We may assume that the interiors of $A$ and ${A_{\updown}}'$ are disjoint
because ${A_{\updown}}'$ is $\del$--parallel. 
Now, $A \cup {A_{\updown}}'$ is an embedded torus in $V$.
Its image in $\pi_1$ is generated by $\alpha_2^{-1}\alpha_1 = a$ and some
other element $y$ that
commutes with $a$. Using \cite[Theorem~4.5]{CGT} again, we see that $y$
is either trivial or in the
$\mathbb{Z}^{\times 2}$ factor of $\mathbb{Z}^{\times 2} \ast \mathbb{Z}
\cong \pi_1V$. If $y$ is trivial or a power of $a$,
then $A \cup {A_\updown}'$ bounds a solid torus and the two annuli are isotopic.
If $y$ is not just a power of $a$, then $A \cup {A_{\updown}}'$ is
incompressible and must parallel to
$T_{\updown}$ by Claim \ref{VAtoroidal}. In this situation, $A$ is isotopic
to $A_\updown$.
The two annuli cannot be permuted by a diffeomorphism of $V$ because
${A_\updown}^\prime$ is $\del$--parallel
and $A_\updown$ is not. This shows that $g|_{V}(A_\updown)$ is isotopic
to $A_\updown$ via an
isotopy that fixes the boundary in $V$. This concludes this second case.
The weak contractibility of the component in the embedding space again
follows from Theorem \ref{htpyTypeInc}.
\end{proof}

We now implement all the tools at our disposal to complete the proof of Theorem \ref{freeSCL}.

\begin{proof}[Proof of Theorem \ref{freeSCL}]
Thanks to Theorems~\ref{KnotStructureThm} and~\ref{LinkStructureThm} and
\ref{pizeFreenessCommute},
we are assured that $\mu$ induces a bijection on components. Thus, we are
left to prove that it is a
homotopy equivalence on each of these components. 

On the component of the unlink, $\mu$ restricts to the map
$$\SCL(\emptyset; o) \times (\QQhat^0, \PPhat, \PPhat, \PPhat)^{\times
\emptyset} \rightarrow
\LLhat_{\id_{\mathbb{R}} \times \iota}.$$
The space $\SCL(\emptyset; o)$
consists of a single point.
The complement of the unlink is a $2$--handlebody $H_2$, and the
diffeomorphism group
$\Diffd(H_2, \del H_2) \simeq \Diff(H_2, \del H_2)$ is contractible.
Proposition \ref{BDiff} 
then gives the contractibility of
$\LLhat_{\id_{\mathbb{R}} \times \iota}=B\! \Diffd(H_2, \del H_2)$ so we have
an equivalence in this case. 

On the component of a fat long knot or an element of $\SShat^0$,
Theorems~\ref{BudneyFreeness} and~\ref{KoytcheffFreeness} 
imply the result because the action of $\SCL$ restricts to $\kappa$
and $\lambda$ on these components. 

Now let $f$ be a nontrivial fat $2$--string link. The action
map $\mu$ restricts on the component of $f$
to the map $\SCL[\QQhat^0, \PPhat, \PPhat, \PPhat]_f \rightarrow \LLhat_f$.
Suppose $f$ decomposes as the concatenation
$f^o \hash \hat\varphi^\up(f^\up) \hash \hat\varphi^\down(f^\down) \hash
\hat\varphi^\updown(f^\updown)$
for some element $f^o$ of $\SShat^0$ and some fat long knots $f^s$ for $s
\in \{\up, \down, \updown\}$.
We denote by $f^o= \hash_{i \leq |f|_o} f^o_i$ and
$f^s= \hash_{i \leq |f|_s} f^s_i$ the prime decompositions of $f^o$
and $f^s$.
Then, one readily checks with the usual model presented in Subsection \ref{OperadsFreeAlg} that the component
$\SCL[\QQhat^0, \PPhat, \PPhat, \PPhat]_f$ in the free algebra is given by
\[
\SCL(o^{|f|_o}, \up^{|f|_\up}, \down^{|f|_\down},
\updown^{|f|_\updown}; o)_{(1, \dots, |f|_o)}
\times_{(\Sigma_{f^\up} \times \Sigma_{f^\down} \times \Sigma_{f^\updown})}\!
\prod_{i \leq |f|_o} \!\! \LLhat_{f^o_i} \times \prod_s
\prod_{i \leq |f|_s} \!\! \KKhat_{f^s_i},
\]
where $\SCL(o^{|f|_o}, \up^{|f|_\up}, \down^{|f|_\down},
\updown^{|f|_\updown}; o)_{(1, \dots, |f|_o)}$
is the component of 
$$\SCL(o^{|f|_o}, \up^{|f|_\up}, \down^{|f|_\down},
\updown^{|f|_\updown}; o)$$
where the cubes indexed by $o$ appear from left to right in the order $(1,
\dots, |f|_o)$
and where $\Sigma_{f^s}$ is the subgroup of $\Sigma_{|f|_s}$ preserving
the partition specified by $i \sim j$
if and only if $f^s_i$ is isotopic to \smash{$f^s_j$}. This group acts on the
cubes indexed by $s$ in
$\SCL(o^{|f|_o}, \up^{|f|_\up}, \down^{|f|_\down}, \updown^{|f|_\updown};
o)_{(1, \dots, |f|_o)}$
and permutes the entries in $\prod_{i \leq |f|_s} \KKhat_{f^s_i}$.
The inclusion 
$$\SCL(o^{|f|_o}, \up^{|f|_\up},
\down^{|f|_\down},\updown^{|f|_\updown};o)_{(1, \dots, |f|_o)}
\cof \cubeso{2}(|f|_o)_{(1, \dots, |f|_o)} \times \prod_s \cubes{2}(|f|_s)$$
is a homotopy equivalence,
so that we have, by rearranging terms in the product, the natural
equivalences
\begin{align*}
\SCL[\QQhat^0, \PPhat, \PPhat, \PPhat]_f
& \simeq \biggl[ \cubeso{2}(|f|_o)_{(1, \dots, |f|_o)} \times
\prod_i \LLhat_{f^o_i} \biggr] \times
\prod_s \biggl[ \cubes{2}(|f|_s) \times_{\Sigma_{f^s}}
\prod_i \KKhat_{f^s_i} \biggr] 
\cr
& \simeq \cubes{1}[\QQhat^0]_{f^o} \times \prod_s \cubes{2}[\PPhat]_{f^s}.
\end{align*}
On the other hand, thanks to Proposition \ref{BDiff}, we know
that $\LLhat_f$ is a model for the classifying space of
$\Diffd(C_f, \del C_f)$. The three vertical twice-punctured disks $D$
from Lemma \ref{LemthDiscs}
splitting $C_f$ as $C_{f^o}$, $C_{\hat\varphi^\up(f^\up)}$,
$C_{\hat\varphi^\down(f^\down)}$ and
$C_{\hat\varphi^\updown(f^\updown)}$ are stable under the action of
$\Diff(C_f, \del C_f)$ and have their component in $\Emb(D, C_f, \del D)$
weakly contractible. Therefore, Proposition \ref{cuttingProp} gives the
inclusion-based weak equivalence
\[
\Diffd(C_{f^o}, \del C_{f^o}) \times \prod_s
\Diffd(C_{\hat\varphi^s(f^s)}, \del C_{\hat\varphi^s(f^s)})
\cof \Diffd(C_f, \del C_f).
\]
Now, Lemma \ref{LemthAnnuli} states that the three annuli $A_s$
also satisfy the conditions of Proposition \ref{cuttingProp}. They each split $C_{\hat\varphi^s(f^s)}$ into a
$2$--handlebody $H_2$ and a manifold
diffeomorphic to $C_{f^s}$. 
Actually, this second piece is precisely the
image of $C_{f^s}$ 
under $\id_{\mathbb{R}} \times \iota$
when $s \in \{\up, \down\}$ and $C_{f^\updown}$ itself when $s = \updown$.
The diffeomorphism group $\Diffd(H_2, \del H_2)$ is contractible so we
get the further natural equivalences
\[
\Diffd(C_{f^s}, \del C_{f^s}) \cof
\Diffd(C_{f^s}, \del C_{f^s}) \times \Diffd(H_2, \del H_2) \cof
\Diffd(C_{\hat\varphi^s(f^s)}, \del C_{\hat\varphi^s(f^s)}).
\]
Composing these results with Proposition \ref{BDiff} yields the
natural equivalences
\begin{align*}
\LLhat_{f^o} \times \prod_s \KKhat_{f^s} &\simeq
B\Diffd(C_{f^o}, \del C_{f^o}) \times \prod_s B\Diffd(C_{f^s},
\del C_{f^s}) \\
&\simeq		B\Diffd(C_f, \del C_f) \\
&\simeq \LLhat_f.
\end{align*}
We are now able to use Budney, Burke and Koytcheff's freeness results
(Theorems \ref{BudneyFreeness} and \ref{KoytcheffFreeness}) and our previous
discussion to get the equivalence
\[
\SCL[\QQhat^0, \PPhat, \PPhat, \PPhat]_f
\simeq	\cubes{1}[\QQhat^0]_{f^o} \times \prod_s
\cubes{2}[\PPhat]_{f^s}
\simeq	\LLhat_{f^o} \times \prod_s \KKhat_{f^s}
\simeq	\LLhat_f.
\]
We merely have an abstract equivalence at this stage. To show that it
coincides with $\mu$,
we need to check the commutativity up to homotopy of the diagram
\[
\begin{tikzcd}
B\Diffd(C_{f^o}, \del C_{f^o}) \times \displaystyle\prod_s
B\Diffd(C_{f^s}, \del C_{f^s}) \arrow[r, "\simeq"]\arrow[d, "\simeq"]
&B\Diffd(C_f, \del C_f) \arrow[d, "\simeq"]\\
\LLhat_{f^o} \times \displaystyle\prod_s \KKhat_{f^s} \arrow[r,
"\mu"]
&\LLhat_{f}
\end{tikzcd}
\]
but the spaces at stake are $K(G, 1)$'s by Proposition \ref{BDiff}, so it
is enough to
check the commutativity in $\pi_1$. This verification is very similar to
the end of Budney's proof of 
Theorem~11 in \cite{Budney}.
An element of $\pi_1 \LLhat_f$ is (a homotopy class of) a based path in
$\LLhat_f$, ie an isotopy from $f$ to $f$.
The elements of $\pi_1 B\Diffd (C_f, \del C_f)$ can canonically be identified
with $\pi_0 \Diffd(C_f, \del C_f)$
in the long exact sequence of the fibration realizing $\LLhat_f$ as
$B\Diffd(C_f, \del C_f)$ in Proposition \ref{BDiff}.
In this framework, picking a class $\phi \in \pi_0\Diffd(C_{f^o}, \del
C_{f^o})$
and chasing the diagram along the clockwise route turns it into an element
of $\pi_0 \Diffd(C_f, \del C_f)$ with its support
lying between $-1 \times D^2$ and $D_1 \subset C_f$, then converts it into
an isotopy of $f$
according to the construction in Proposition \ref{BDiff}.
Chasing $\phi$ along the counterclockwise route converts it into an
isotopy of $f^o$ in $\pi_1\LLhat_{f^o}$,
then applies $\mu$ to it. The outcome is the same as each
$\hat\varphi^s(f^s)$ is fixed all along this last isotopy.
The same argument shows that picking a class in $\pi_0
\Diffd(C_{f^s}, \del C_{f^s})$
and chasing the diagram in either direction has the same effect.
When evaluated in $\pi_1$, the upper-left product is a direct product,
on which a factor-by-factor verification is thus sufficient to get the
commutativity.

Finally, Hatcher and McCullough proved in \cite{FiniteCWDiff} that the
classifying spaces of the
diffeomorphism groups at stake here have the homotopy types of (aspherical
finite) CW--complexes
(what we use here could also be deduced from Palais' earlier article~\cite{PalaisHtpy}).
Thus, Whitehead's theorem promotes the weak homotopy equivalence $\mu$
to a strong one.
\end{proof}

\bibliographystyle{plain}
\bibliography{references}

\begin{thebibliography}{10}

\bibitem{AFW}
Matthias Aschenbrenner, Stefan Friedl, and Henry Wilton.
\newblock {\em 3-manifold groups}.
\newblock EMS Series of Lectures in Mathematics. European Mathematical Society
  (EMS), Z\"{u}rich, 2015.

\bibitem{KoytcheffPrime}
Ryan Blair, John Burke, and Robin Koytcheff.
\newblock A prime decomposition theorem for the 2-string link monoid.
\newblock {\em J. Knot Theory Ramifications}, 24(2):1550005, 24, 2015.

\bibitem{Boavida13}
Pedro Boavida~de Brito and Michael Weiss.
\newblock Manifold calculus and homotopy sheaves.
\newblock {\em Homology Homotopy Appl.}, 15(2):361--383, 2013.

\bibitem{BudneyJSJ}
Ryan Budney.
\newblock J{SJ}-decompositions of knot and link complements in {$S^3$}.
\newblock {\em Enseign. Math. (2)}, 52(3-4):319--359, 2006.

\bibitem{Budney}
Ryan Budney.
\newblock Little cubes and long knots.
\newblock {\em Topology}, 46(1):1--27, 2007.

\bibitem{BudneySplicing}
Ryan Budney.
\newblock An operad for splicing.
\newblock {\em Journal of Topology}, 5(4):945--976, 2012.

\bibitem{Conant}
Ryan Budney, James Conant, Robin Koytcheff, and Dev Sinha.
\newblock Embedding calculus knot invariants are of finite type.
\newblock {\em Algebr. Geom. Topol.}, 17(3):1701--1742, 2017.

\bibitem{KoytcheffInfect}
John Burke and Robin Koytcheff.
\newblock A colored operad for string link infection.
\newblock {\em Algebr. Geom. Topol.}, 15(6):3371--3408, 2015.

\bibitem{Cerf}
Jean Cerf.
\newblock Topologie de certains espaces de plongements.
\newblock {\em Bull. Soc. Math. France}, 89:227--380, 1961.

\bibitem{Cromwell}
Peter~R. Cromwell.
\newblock {\em Knots and links}.
\newblock Cambridge University Press, Cambridge, 2004.

\bibitem{Ducoulombier}
Julien Ducoulombier and Danica Kosanovic.
\newblock On the $\mathcal{C}_{2}$-algebra structure of long knots.
\newblock In preparation.

\bibitem{Weiss2}
Thomas~G. Goodwillie and Michael Weiss.
\newblock Embeddings from the point of view of immersion theory. {II}.
\newblock {\em Geom. Topol.}, 3:103--118, 1999.

\bibitem{HatcherSmale}
Allen~E. Hatcher.
\newblock A proof of the {S}male conjecture, {${\rm Diff}(S^{3})\simeq {\rm
  O}(4)$}.
\newblock {\em Ann. of Math. (2)}, 117(3):553--607, 1983.

\bibitem{HatcherInc}
Allen~E. Hatcher.
\newblock Spaces of incompressible surfaces, 1999.
\newblock Available at
  \url{https://pi.math.cornell.edu/~hatcher/Papers/emb.pdf}.

\bibitem{Hatcher3M}
Allen~E. Hatcher.
\newblock {\em {Notes on Basic 3-Manifold Topology}}.
\newblock 2007.

\bibitem{FiniteCWDiff}
Allen~E. Hatcher and Darryl McCullough.
\newblock Finiteness of classifying spaces of relative diffeomorphism groups of
  {$3$}-manifolds.
\newblock {\em Geom. Topol.}, 1:91--109, 1997.

\bibitem{Hirsch}
Morris~W. Hirsch.
\newblock {\em Differential topology}, volume~33 of {\em Graduate Texts in
  Mathematics}.
\newblock Springer-Verlag, New York, 1994.
\newblock Corrected reprint of the 1976 original.

\bibitem{Hov}
Mark Hovey.
\newblock {\em Model categories}, volume~63 of {\em Mathematical Surveys and
  Monographs}.
\newblock American Mathematical Society, Providence, RI, 1999.

\bibitem{JS}
Jaco and Shalen.
\newblock A new decomposition theorem for irreducible sufficiently-large
  $3$-manifolds.
\newblock {\em Algebraic and geometric topology (Proc. Sympos. Pure Math.,
  Stanford Univ., Stanford, Calif., 1976), Part 2}, 32:71--84, 1976.

\bibitem{J}
Klaus Johannson.
\newblock {\em {Homotopy equivalences of $3$-manifolds with boundaries}}.
\newblock Lecture Notes in Mathematics. Springer, Berlin, 1979.

\bibitem{McCullough}
John {Kalliongis} and Darryl {McCullough}.
\newblock {Fiber-preserving diffeomorphisms and imbeddings}.
\newblock {\em arXiv Mathematics e-prints}, page math/9802003, Jan 1998.

\bibitem{Kupers}
Alexander Kupers.
\newblock {\em Lectures on Diffeomorphism Groups of Manifolds}.
\newblock 2019.
\newblock Available at
  \url{http://people.math.harvard.edu/~kupers/teaching/272x/book.pdf}.

\bibitem{CGT}
Wilhelm Magnus, Abraham Karrass, and Donald Solitar.
\newblock {\em Combinatorial group theory}.
\newblock Dover Publications, Inc., Mineola, NY, second edition, 2004.
\newblock Presentations of groups in terms of generators and relations.

\bibitem{May}
J.~P. May.
\newblock {\em The geometry of iterated loop spaces}.
\newblock Springer-Verlag, Berlin-New York, 1972.
\newblock Lectures Notes in Mathematics, Vol. 271.

\bibitem{Munson14}
Brian~A. Munson and Ismar Voli{\'c}.
\newblock Cosimplicial models for spaces of links.
\newblock {\em J. Homotopy Relat. Struct.}, 9(2):419--454, 2014.

\bibitem{Palais}
Richard~S. Palais.
\newblock Local triviality of the restriction map for embeddings.
\newblock {\em Comment. Math. Helv.}, 34:305--312, 1960.

\bibitem{PalaisHtpy}
Richard~S. Palais.
\newblock Homotopy theory of infinite dimensional manifolds.
\newblock {\em Topology}, 5:1--16, 1966.

\bibitem{Schubert}
Horst Schubert.
\newblock Die eindeutige {Z}erlegbarkeit eines {K}notens in {P}rimknoten.
\newblock {\em S.-B. Heidelberger Akad. Wiss. Math.-Nat. Kl.}, 1949(3):57--104,
  1949.

\bibitem{Sinha06}
Dev~P. Sinha.
\newblock Operads and knot spaces.
\newblock {\em J. Amer. Math. Soc.}, 19(2):461--486, 2006.

\bibitem{Steen}
N.~E. Steenrod.
\newblock A convenient category of topological spaces.
\newblock {\em Michigan Math. J.}, 14:133--152, 1967.

\bibitem{Thurston}
William~P. Thurston.
\newblock Three-dimensional manifolds, {K}leinian groups and hyperbolic
  geometry.
\newblock {\em Bull. Amer. Math. Soc. (N.S.)}, 6(3):357--381, 1982.

\bibitem{Volic}
Ismar Voli\'{c}.
\newblock Finite type knot invariants and the calculus of functors.
\newblock {\em Compos. Math.}, 142(1):222--250, 2006.

\bibitem{Waldhausen}
Friedhelm Waldhausen.
\newblock On irreducible {$3$}-manifolds which are sufficiently large.
\newblock {\em Ann. of Math. (2)}, 87:56--88, 1968.

\bibitem{Weiss}
Michael Weiss.
\newblock Embeddings from the point of view of immersion theory. {I}.
\newblock {\em Geom. Topol.}, 3:67--101, 1999.

\bibitem{Yau}
Donald Yau.
\newblock {\em {Colored operads}}.
\newblock Graduate studies in mathematics. American Mathematical Society,
  Providence, RI, 2016.

\end{thebibliography}

\end{document}